\documentclass[a4paper,11pt]{article}
\usepackage{amssymb}
\usepackage{latexsym}
\usepackage{amsmath}
\usepackage{amsthm}
\usepackage{amsfonts}
\usepackage{amssymb}
\allowdisplaybreaks
\usepackage{graphicx}
\usepackage[utf8x]{inputenc}
\usepackage{longtable}
\usepackage{cite}

\usepackage{verbatim}

\newtheorem{Th}{Theorem}
\newtheorem{Prop}{Proposition}
\newtheorem{Co}{Corollary}
\newtheorem{Lm}{Lemma}

\newtheorem{Dfi}{Definition}
\newtheorem{Rm}{Remark}

\newcommand{\be}{\begin{equation}}
\newcommand{\ee}{\end{equation}}
\newcommand{\R}{\mathbb{R}}
\newcommand{\N}{\mathbb{N}}
\newcommand{\C}{\mathbb{C}}
\newcommand{\CP}{\mathbb{CP}}

\def\lf{\left}
\def\rg{\right}

\def\e{\varepsilon}

\def\ov{\overline}

\def\om{\omega}
\def\p{\partial}

\def\delb{\ov{\p}}
\def\delbs{\ov{\p}^\ast}

\newcommand{\norm}[2]{\left\lVert#1 \right\rVert_{#2}}
\newcommand{\opnorm}[1]{{\left\vert\kern-0.25ex\left\vert\kern-0.25ex\left\vert #1 
    \right\vert\kern-0.25ex\right\vert\kern-0.25ex\right\vert}}
\usepackage{setspace}

\addtocontents{toc}{\protect\setstretch{0.9}}
\begin{document}
\title{Weak Holomorphic Structures over K\"ahler Surfaces.}
\author{ Alexandru P\unichar{259}unoiu, Tristan Rivi\`ere\footnote{Forschungsinstitut f\"ur Mathematik, ETH Zentrum,
CH-8093 Z\"urich, Switzerland.}}
\maketitle

{\bf Abstract:}{ In this work we prove that any unitary Sobolev $W^{1,2}$ connection of an Hermitian bundle over a 2-dimensional K\"ahler manifold whose curvature is $(1,1)$ defines a smooth holomorphic structure. We prove moreover that such a connection can be strongly approximated in any $W^{1,p}$ ($p<2$) norm by smooth connections satisfying the same integrability condition. }

\medskip

\noindent{\bf Math. Class. 32L05, 58E15, 58D99, 46T99, 35J47 }

\tableofcontents

\section{Introduction}
The calculus of variations of Yang-Mills in 4-dimensions has naturally lead to the definition of Sobolev connections \cite{freed2012instantons}. We consider this notion in the following complex framework. Let $E$ be $C^\infty$ complex vector bundle of rank $n$ over a  K\"ahler manifold $X$ and $h_0$ be some reference Hermitian
inner product in the fibers of $E$: i.e. $(E,h_0)$ defines an Hermitian vector bundle. We shall sometimes consider $E$ issued from it's associated $GL_n(\C)$ principal bundle or from it's associated unitary principal bundle. We are interested in the space of Sobolev  $W^{1,2}$ connections of $E$ which are defined as follows: Let $\nabla_0$ be a smooth connection of $E$, we denote
\[
{\mathcal S}^{1,2}(E,h_0):=\lf\{\nabla:=\nabla_0+\eta\quad\mbox{ where }\quad \eta\in W^{1,2}(\Omega^1(ad_{h_0}(E)))\rg\}
\]
where $W^{1,2}(\Omega^1(ad_{h_0}(E)))$ is the space of Sobolev $W^{1,2}$ 1-form sections into the sub-bundle of the endomorphism bundle $End(E)$ made of the unitary endomorphisms for the reference metric $h_0$. ${\mathcal S}^{1,2}(E,h_0)$ is called the {\it space of Sobolev unitary $W^{1,2}-$connections of $(E,h_0)$}.\\

Throughout the paper we will heavily use gauge theory in order to obtain our results. Thus, we find it useful to recall to the reader the notion of a gauge transformation of a connection $\nabla\in {\mathcal S}^{1,2}(E,h_0)$. Let $g$ be a section of the Hermitian vector bundle $E$, then the gauge transformation of $\nabla=\nabla_0+\eta$  by $g$ is defined as $\nabla^g = \nabla_0 + \eta^g$, where $\eta^g:= g^{-1}dg +g^{-1}\eta g$.\\

We will be interested in the convergence of Sobolev unitary connections, and their respective Sobolev structures in the case of closed K\"ahler surface. We will positively answer the question of strong convergence: unitary $W^{1,2}$ connections $\nabla$ that preserve the \textit{integrability condition} $$F_{\nabla}^{0,2}=0$$
can be strongly approximated by smooth connections which also satisfy the condition. Recall that in the smooth case that unitary connections satisfying the {\it integrability condition} are in one to one correspondance with holomorphic structures (see \cite[Theorem 2.1.53]{donaldson1990geometry}). The goal of this paper is to extend this identification to Sobolev connections.\\

More precisely, our first main result is the following:
\begin{Th}\label{thm1}
\label{4-dim}
Let $\nabla$ be a unitary $W^{1,2}$ connection of an hermitian bundle $(E,h_0)$ over a closed K\"ahler surface $X^2$. Assume $\nabla$ satisfies the integrability condition
\be
\label{II.1}
F^{0,2}_\nabla=0\quad
\ee
then there exists a \underbar{smooth} holomorphic structure ${\mathcal E}$ on $E$ and a $\bigcap\limits_{q<2} W^{2,q}$ section $h$ of the bundle of  positive Hermitian
endomorphisms of $E$ such that
\be
\label{II.1a}
\nabla=\p_0+h^{-1}\p_0 h+\ov{\p}_{\mathcal E}
\ee
where $\ov{\p}_{\mathcal E}$ is the $\ov{\partial}-$operator associated to the holomorphic bundle ${\mathcal E}$ and $\p_0$ is the $1$-$0$ part of the Chern connection associated\footnote{These connections are not necessarily unitary with respect to $h_0$ anymore.} to the holomorphic structure ${\mathcal E}$ and the chosen reference hermitian product $h_0$.

\end{Th}

The second main result of this paper asserts that Sobolev holomorphic structures associated to Sobolev unitary connections are strongly approximable by smooth ones in 2 complex dimension (the dimension for which the Yang-Mills energy is critical):
\begin{Th}\label{thm2}
Under the assumptions of Theorem \ref{thm1}, there exists a sequence of \underbar{smooth} connections $\nabla_k$ on a smooth holomorphic bundle $\mathcal E_k$ satisfying
\[
F^{0,2}_{\nabla_k}=0\quad,
\]
and converging to $\nabla$ in the sense of:
\be
\label{II.2}
d_p(\nabla_k,\nabla):=\inf_{\sigma\in {\mathcal G}^{1,2}(GL(n,{\C}))}\int_{X^2}|\nabla_k-\nabla^\sigma|^p\,\om^2\,+\int_{X^2}|F_{\nabla_k}-F_{{\nabla}^\sigma}|^p\,\om^2\rightarrow 0
\ee
for any $p<2$, where ${\mathcal G}^{1,2}(GL_n({\C}))$ is the space of $W^{1,2}$ gauge transformations on $E$ for the group $GL_n({\C})$. \\

Moreover, there exists a family of isomorphisms $\mathcal H_k$ such that $$\delb_{\mathcal E_k} = \mathcal H_k^{-1}\circ \delb_{\mathcal E}\circ \mathcal H_k.$$
That is, the sequence of connections $\nabla_k$ act on an equivalent bundles to $E$.
\hfill$\Box$
\end{Th}

The strong approximation of Sobolev connections by smooth ones has been proven in the case of Riemannian manifolds without the integrability condition (\ref{II.1}). This is less involved and hence, one of the novelties of this paper is exploring how the approximation can be achieved by adding the integrability condition.

\begin{Rm}
We have formulated these theorems by considering closed K\"ahler manifolds. This consideration has been done for simplicity, since it allows us to use the fact that $\delb\delbs+\delbs\delb$ is locally equal to the Hodge-Laplace operator $\Delta_d=dd^\ast +d^\ast d$. The reader should take into account the fact that the results are generalisable to closed complex manifolds by carefully dealing with error between the $\delb\delbs+\delbs\delb$ and $\Delta_d$ operators.
\end{Rm}

There has been a definition of {\it weak connections} with $L^2$ bounded curvature given by the second author in collaboration with M. Petrache in \cite{petrache2017resolution} and \cite{petrache2018space}. This definition was motivated in a search of the closure of Sobolev connections below a Yang-Mills energy level. Roughly speaking a weak connection in real $5$-dimensions is defined as being an $L^2$ $1$-form into a Lie algebra $\mathfrak g$ such that its restriction on a.e. $4$-sphere is a Sobolev connection. In higher dimensions weak connections are defined in an iterative way. That is, for $n>5$, a weak connection in $n$-dimensions is an $L^2$ form $A$ into the Lie Algebra such that  when restricted to a.e $n-1$ spheres is itself a weak connection. This space has been proved to be weakly sequentially closed under Yang-Mills Energy control. This was one of the main results in \cite{petrache2017resolution} and \cite{petrache2018space}. \\

In higher even dimensions, for the {\it weak connections} defined in \cite{petrache2018space} over a complex manifold and satisfying in addition the integrability condition $F^{0,2}_A=0$, we expect theorems \ref{thm1}  and \ref{thm2} to extend in the following way :  We expect  to have necessary singularities and the {\it smooth holomorphic structures} should be  replaced by the more general notion of {\it coherent sheaves}. The question remains to know how smooth these sheaves can be and if a weak holomorphic structure defines a {\it reflexive sheaf } or not.\\

The motivation for addressing these questions takes it's roots in a paper of G. Tian \cite{Ti} in which  the closure of the space of smooth Yang-Mills fields has been studied. It leads naturally to the study of Yang-Mills Fields on a bundle well defined away from a co-dimension 4 closed rectifiable set in the basis.
The attempts in \cite{petrache2017resolution} and \cite{petrache2018space} was to give a suitable notion of such singular bundles together with a singular connection that enjoys a {\it sequential weak closure property}. The attached singular ``bundle'' to these singular connections could be thought as a {\it real version of coherent sheaves}. The goal of mixing the  notion 
of weak connection with the integrability condition $F^{0,2}_A=0$ is to check whether the corresponding singular bundle coincide with the classical notion of {\it reflexive sheaves} in the complex framework. The present paper is bringing a positive answer to this question when the basis is a K\"ahler surface.




\vspace{20pt}

\textbf{Structure of the paper}\\

In \ref{sec:notation} we introduce some notation.\\

Section \ref{sec:lowenergy} is devoted to the proof of theorem \ref{thm2} in the case of small Yang-Mills energy. This proof is not going to be used for proving the theorem in its full generality. However, we thought that it could be useful for the reader to expose a different approach in this particular case and the scheme of the proof we are giving in this section is going to be used in later ones.\\

Under the smallness condition of Yang-Mills energy, in Section \ref{sec:holo} we prove that connections satisfying the integrability condition \eqref{II.1} are locally \textit{holomorphically trivialisable}, meaning that in any geodesic ball embedded in our manifold $X^2$ where we can write $\nabla = d+A$, we show the existence of $g\in GL_n(\C)$ such that $A^{0,1} = -\delb g\cdot g^{-1}$. Using this result, we prove theorem \ref{thm1} in Section \ref{sec:thm1}.\\

Sections \ref{sec:highen} and \ref{sec:thm2} are dedicated to proving theorem \ref{thm2} in the case of high Yang-Mills energy. The former section proves the strong approximation result, and the latter concludes the statement by proving that the connections $\nabla_k$ act on equivalent bundles to $E$.

\newpage 
\subsection{Notations}\label{sec:notation}

\begin{longtable}{rll}
$M_n(\C)$ & - & the space of $n\times n$ complex valued square matrices\\
${\mathcal G}^{1,2}(GL_n({\C}))$ & -& the space of $W^{1,2}$ gauge transformations on $E$ for the group $GL_n({\C})$\\
$\Gamma(E)$ & - & the space of global smooth sections of the vector bundle $E$\\
$\Gamma_{W^{p,q}}(E)$ & - & the space of global $W^{p,q}$ sections of the vector bundle $E$\\
$\mathcal A^{p,q}(E)$ & - & the space of global $(p,q)$-sections defined on the vector bundle $E$\\
$\Omega^{p,q} U\otimes \mathfrak{g}$ & - & the space of $\mathfrak{g}$-valued $(p,q)$-forms on $U$\\
$W_D^{2,p}$ & - & the space of Sobolev functions $W^{2,p}$ that vanish on the boundary of the domain\\
$\vartheta \omega$ & - & $\vartheta = -\ast \p\ast$, formal adjoint of $\delb$ (see \cite[p. 83]{folland1972neumann})\\
$N \omega$ & - & $N$ is the inverse operator of $\Box = \delb\delbs+\delbs\delb$ applied to the $(p,q)$-form $
\omega$\\
$B^4$ & - & 4-dimensional unit open ball\\
$B_r^4$ & - & 4-dimensional open ball of radius $r>0$\\
$[A,B]$ & - & $[A,B] = A\wedge B + B\wedge A$, if $A, B$ are $(p,q)$-forms\\
$\left(A^{0,1}\right)^g$ & - & $g^{-1}\delb g + g^{-1}A^{0,1}g$\\
$\opnorm{T}$ & - & the norm of the operator $T:X\rightarrow Y$, for $X,Y$ Banach spaces\\
$\sigma(T)$ & - & the spectrum of the operator $T$\\
$\rho(T)$ & - & the resolvent of the operator $T$, defined as $\C\setminus \sigma(T)$\\
$A(x)$ & - & for $k$-forms $A = \sum_{|I|=k} x_I dx_I$ denote $A(x) = \sum_{I} a_i(x) dx_I$\\
$d_p(\nabla_k,\nabla)$ & - & $\inf_{\sigma\in {\mathcal G}^{1,2}(GL(n,{\C}))}\int_{X^2}|\nabla_k-\nabla^\sigma|^p\,\om^2\,+\int_{X^2}|F_{\nabla_k}-F_{{\nabla}^\sigma}|^p\,\om^2$
\end{longtable}

\section{Density under low energy}\label{sec:lowenergy}

Given a unitary $W^{1,2}$ connection $\nabla$ of the hermitian bundle $(E,h_0)$ over a closed K\"ahler surface $X^2$ satisfying $F_\nabla^{0,2}=0$, we assume without loss of generality that $B^4$ is a geodesic ball in $X^2$ and that $\nabla$ trivialises as $\nabla = d+A$ in $B^4$, where $A$ is a connection $1$-form.  Moreover, in this section we will work with low Yang-Mills energy and hence assume that the $L^2$ norm of curvature $2$-form is controlled by the $W^{1,2}$ norm of $A$ - which satisfies the smallness condition $$\norm{A}{W^{1,2}(B^4)}\leq \e_0(X^2,\omega)$$ for some $ \e_0(X^2,\omega)>0$ depending on the surface $X^2$ and the K\"ahler form $\omega$. We will use the smallness assumption throughout this section. Moreover, to fix ideas we will assume that $B^4$ is the flat closed unit ball.\\ 

We start by showing how to smooth $1$-forms, keeping the approximating sequence unitary. This method however, does not ensure the integrability condition (\ref{II.1}). Let $p>1$ and $A\in W^{1,p}(\Omega^1 B^4\otimes \mathfrak{u}(n))$ then we can always find a smooth sequence of unitary $1$-forms $A_k\in C^\infty(\Omega^1 B^4\otimes \mathfrak{u}(n))$ such that
$$A_k\rightarrow A\text{ in }W^{1,p}(B^4).$$
Indeed, we can write $A$ as $A = A^{0,1} - \ov{A^{0,1}}^T$, where $A^{0,1}=\alpha_1d\ov{z}_1 + \alpha_2d\ov{z}_2$. Since for each $i=1,2$, we have $\alpha_i\in W^{1,p}(B^4,\mathfrak{u}(n))$, then by the density of $C^\infty$ functions into $W^{1,p}$, there exist sequences
$$\alpha_{1,k}\rightarrow \alpha_1\text{ in }W^{1,p}(B^4)$$
and
$$\alpha_{2,k}\rightarrow \alpha_2\text{ in }W^{1,p}(B^4).$$
By defining
$A_k^{0,1}:=\alpha_{1,k} d\ov{z}_1 + \alpha_{2,k}d\ov{z}_2$
and $A_k:= A_k^{0,1}-\ov{A_k^{0,1}}^T$, we obtain by construction the convergence of $A_k$ to our initial form $A$ in $W^{1,p}$. Moreover, $A_k$ is a unitary $1$-form.
\vspace{4mm}

The next lemma helps us to prove that we can always find a perturbation of a given a connection $1$-form $A\in W^{1,p}$ with low norm such that the integrability condition \eqref{II.1} is satisfied.
\begin{Lm}{\label{low:p-approxthm}}
Let $p\geq 2$. There exists $\e>0$ depending on $p$ such that  for any $A \in W^{1,p}( \Omega^1 B^4\otimes \mathfrak{u}(n))$ satisfying $\norm{A}{W^{1,p}(B^4)} \leq \e$, there exists a $1$-form $$\tilde{A}\in W^{1,p}( \Omega^1 B^4\otimes \mathfrak{u}(n))$$ that satisfies the integrability condition $$F^{0,2}_{\tilde{A}} = 0$$
and
$$\norm{\tilde{A}-A}{W^{1,p}(B^4)}\leq C\norm{F_A^{0,2}}{L^p(B^4)},$$
for some constant $C>0$ depending on $p$. Moreover, if $A$ is smooth then $\tilde A$ is also.
\end{Lm}
In the proof we will use Sobolev embeddings under the assumption that $p\in [2,4)$ - which is the more delicate case. If $p\geq 4$, the results hold by considering the corresponding Sobolev embeddings.
\begin{proof}[Proof of Lemma \ref{low:p-approxthm}]
In order to obtain a form satisfying the integrability condition, we want to perturb the $A$ with a form $V\in C^\infty( \Omega^{1} B^4, \mathfrak{u}(n)) $ such that $F_{A+V}^{0,2} = 0$. We express $V$ by $V = v -\ov{v}^T$, where $v \in W^{1,p}(\Omega^{0,1} B^4\otimes M_n(\C))$. By expanding $F_{A+V}^{0,2} = 0$ we, thus, get the following equation
$$\delb v + [v,A^{0,1}] + v\wedge v = - F_{A}^{0,2}.$$
By the fact that we work on a K\"ahler manifold, we know that $\delb\delbs \cdot = \frac{1}{2}\Delta\cdot d\ov{z_1}\wedge d\ov{z_2}$ on the space $\Omega ^{0,2} B^4$. Thus, we want to transform the PDE above into an elliptic one by taking $v$ of the form
$$v = \delbs \om$$
with $\om = 0$ on $\p B^4$ - so that $\delbs$ is well defined. We solve the following elliptic system:

$$\delb\delbs \om + [\delbs\om,A^{0,1}] + \delbs\om\wedge \delbs \om = - F_{A}^{0,2}.$$

Since $A^{0,1}$ and $F_{A}^{0,2}$ have small norms, we can solve it using a fixed point argument. We consider the following Dirichlet problem:
$$\left\{\begin{array}{rll}
\delb\delbs \om &= -  [\delbs\om,A^{0,1}] - \delbs\om\wedge\delbs \om  - F_{A}^{0,2}&\text{ in }B^4\\
\om &=0&\text{ on }\p B^4.
\end{array}\right.$$
We fix $k$ and we build the following sequence $\{\om_j\}_{j=1}^\infty$ of forms that solve the PDEs:
$$\begin{array}{ll}
\delb\delbs \omega_1 &= - F_{A}^{0,2}\\
\delb\delbs\omega_2 &= -  [\delbs\omega_1,A^{0,1}] - \delbs\omega_1\wedge\delbs \omega_1  - F_{A}^{0,2}\\
\cdots\\
\delb\delbs\omega_{j+1}&= -  [\delbs\omega_j,A^{0,1}] - \delbs\omega_j\wedge\delbs \omega_j  - F_{A}^{0,2}\\
\cdots
\end{array}$$
where $\omega_j = 0$ on $\p B^4$ for all $j\geq 1$.\\

\textbf{Claim.} $\{\omega_j\}_{j=1}^\infty$ exists and is a bounded sequence in $W^{2,p}$.\\

By classical elliptic theory, since $F_{A}^{0,2}\in L^p$, there exists a constant $C_1>0$ depending only on $p$ such that  
$$\norm{\om_1}{W^{2,p}(B^4)}\leq C_1\norm {\delb\delbs \om_1}{L^p(B^4)} = C_1\norm{F_A^{0,2}}{L^p(B^4)}<2C_1\norm{F_A^{0,2}}{L^p(B^4)}.$$
By induction we prove that $\omega_j$ exists and satisfies the uniform bound $\norm{\omega_j}{W^{2,p}(B^4)}\leq 2C_1\norm{F_A^{0,2}}{L^p(B^4)}$. We assume that $\om_j$ exists and $\norm{\omega_j}{W^{2,p}(B^4)}\leq 2C_1\norm{F_A^{0,2}}{L^p(B^4)}$ and prove that $\om_{j+1}$ exists with the same $W^{2,p}$ bound. By the Sobolev embedding $W^{1,p}\hookrightarrow L^{4p/(4-p)}$, there exists constants $C_2>0, C_3>0$ so that
$$\norm{\delbs\om_j}{L^{4p/(4-p)}(B^4)}\leq C_2\norm{\nabla \om_j}{W^{1,p}(B^4)}\leq C_2\norm{\om_j}{W^{2,p}(B^4)}\leq 2C_1\cdot C_2\norm{F_A^{0,2}}{L^p(B^4)}$$
and
$$\norm{A^{0,1}}{L^{4p/(4-p)}(B^4)}\leq C_3\norm{A^{0,1}}{W^{1,p}(B^4)}\leq C_3\e$$
In addition, since $4p/(4-p)\geq 2p$ for any $p\geq 2$, then $W^{1,p}$ continuously embeds into $L^{2p}$ and we can bound $\norm{F_A}{L^p(B^4)}$ as such:
\begin{align*}
\norm{F_A}{L^p(B^4)}&\leq \norm{dA}{L^p(B^4)} + \norm{A}{L^{2p}(B^4)}^2\leq \norm{A}{W^{1,p}(B^4)} + C_3\norm{A}{W^{1,p}(B^4)}^2\\
& \leq \norm{A}{W^{1,p}(B^4)} + C_3\e\norm{A}{W^{1,p}(B^4)}\\
&= (1+C_3\e)\norm{A}{W^{1,p}(B^4)}.
\end{align*}
Define the constant $C_4:=1+C_3\e$. Moreover, since $p\geq 2$, we have the embedding $L^{2p/(4-p)}\hookrightarrow L^p$. Denote $$f_j:= -[\delbs\omega_j,A^{0,1}] - \delbs\omega_j\wedge\delbs \omega_j  - F_{A}^{0,2}.$$ Using the estimates above we obtain:
$$\begin{array}{rl}
\norm{f_j
}{L^p(B^4)}&\leq\norm{[\delbs\omega_j,A^{0,1}]}{L^p(B^4)} +  \norm{\delbs\omega_j\wedge\delbs \omega_j}{L^p(B^4)} + \norm{F_A^{0,2}}{L^p(B^4)}\\
&\leq \norm{[\delbs\omega_j,A^{0,1}]}{L^{2p/(4-p)}(B^4)} +  \norm{\delbs\omega_j\wedge\delbs \omega_j}{L^{2p/(4-p)}(B^4)} + \norm{F_A^{0,2}}{L^p(B^4)}\\
&\leq  \norm{\delbs\omega_j}{L^{4p/(4-p)}(B^4)}\norm{A^{0,1}}{L^{4p/(4-p)}(B^4)} +  \norm{\delbs\omega_j}{L^{4p/(4-p)}(B^4)}^2+ \norm{F_A^{0,2}}{L^p(B^4)}\\
&\leq C_2C_3\norm{F_A^{0,2}}{L^p(B^4)}\e + 4C_1^2 \cdot  C_2^2\norm{F_A^{0,2}}{L^p(B^4)}^2 +\norm{F_A^{0,2}}{L^p(B^4)}\\
&\leq C_2C_3\norm{F_A^{0,2}}{L^p(B^4)}\e + 4C_1^2 \cdot  C_2^2C_4\norm{F_A^{0,2}}{L^p(B^4)}\norm{A}{W^{1,p}(B^4)} +\norm{F_A^{0,2}}{L^p(B^4)}\\
& \leq (C_2C_3\e+4C_1^2C_2^2C_4\e+1)\norm{F_A^{0,2}}{L^p(B^4)}
\end{array}$$

Hence, $ -[\delbs\omega_j,A^{0,1}] - \delbs\omega_j\wedge\delbs \omega_j  - F_{A}^{0,2}\in L^p$ and the solution $\omega_{j+1}$ to the PDE 
\be\label{low:recurrpde}\left\{
\begin{array}{rll}
\delb\delbs\omega_{j+1}&= -[\delbs\omega_j,A^{0,1}] - \delbs\omega_j\wedge\delbs \omega_j  - F_{A}^{0,2}&\text{ in } B^4\\
\omega_{j+1}& = 0 & \text{ on }\p B^4
\end{array}\right.\ee exists.
Choosing $\e>0$ such that $$C_2C_3\e+4C_1^2C_2^2C_4\e<1$$ is satisfied, it follows that  we can obtain the required bound:
$$ \norm{\omega_{j+1}}{W^{2,p}(B^4)}\leq C_1\norm{\delbs\delb \omega_{j+1}}{L^p(B^4)} =C_1\norm{f_j}{L^p(B^4)} \leq 2C_1\norm{F_A^{0,2}}{L^p(B^4)}.$$
Hence, by induction, we have proven the claim.\\

\textbf{Claim.} $\{\omega_j\}_{j=1}^\infty$ is a Cauchy sequence.\\

Since each $\omega_j$ satisfies the elliptic PDE (\ref{low:recurrpde}), we can estimate the difference $\omega_{j+1} - \omega_j$ as such:
$$\begin{array}{rl}
\norm{\omega_{j+1}-\omega_j}{W^{2,p}(B^4)} &\leq C\left(\norm{\omega_{j}-\omega_{j-1}}{W^{2,p}(B^4)}\norm{A^{0,1}}{W^{1,p}(B^4)} + \norm{F_A^{0,2}}{L^p} \norm{\omega_{j}-\omega_{j-1}}{W^{2,p}(B^4)}\right)\\
&\leq 2C\e \norm{\omega_{j}-\omega_{j-1}}{W^{2,p}(B^4)}\\
\end{array}$$
where $C>0$ is a constant depending on $p$. Choosing $\e$ such that in addition $2C\e<1$ is satisfied, it then follows that the sequence is Cauchy.\\

Because $W^{2,p}$ is a Banach space and the sequence $\{\omega_j\}_{j=1}^\infty$ is Cauchy, we have that the sequence converges strongly in $W^{2,p}$ to a limit which we denote by $\omega$. Moreover, by construction $\omega$ satisfies the PDE:

$$\left\{\begin{array}{rll}
\delb\delbs \omega &= -  [\delbs\omega,A^{0,1}] - \delbs\omega\wedge\delbs \omega  - F_{A}^{0,2}&\text{ in }B^4\\
\omega &=0&\text{ on }\p B^4.
\end{array}\right.$$

Define $\tilde{A} = A + \delbs\om-\ov{\delbs\om}^T$. Then $F_{\tilde{A}}^{0,2}=0$ and
$$\norm{\tilde{A}-A}{W^{1,p}} = \norm{\delbs\omega-\ov{\delbs\om}^T}{W^{1,p}}\leq 4C_1\norm{F_A^{0,2}}{L^p}.$$
We have proven the first result of this article.
\end{proof}

Using the result above, we can prove the first theorem of this paper. 

\begin{Th}{\label{approxthm}} There exists $\e_0>0$ such that if $A \in W^{1,2}(\Omega^{1} B^4,\mathfrak{u}(n))$ satisfies the smallness condition $\norm{A}{W^{1,2}(B^4)}\leq \e_0$ and the integrability condition $F^{0,2}_{A}= 0$, then there exists a smooth sequence $A_k\in C^\infty(\Omega^1 B^4,\mathfrak{u}(n))$ so that:
$$A_k \rightarrow A\text{ in }W^{1,2}(\Omega^1 B^4,\mathfrak{u}(n))$$
and satisfies the integrability condition $F^{0,2}_{A_k} = 0.$
\end{Th}

\begin{proof}[Proof of Theorem \ref{approxthm}]
As we have discussed at the start of this section, we can always construct a sequence smooth sequence of forms $\hat{A}_k$ that converge in $W^{1,2}$ to A and $F_{\hat A_k}^{0,2}\rightarrow 0= F_A^{0,2}$ in $L^2$ as $k\rightarrow \infty$. Let $\e>0$ be the constant given by Lemma \ref{low:p-approxthm} and pick $\e_0=\e/2$. Then there exists $k_0\geq 0$ such that $\norm{\hat A_k-A}{W^{1,2}(B^4)}\leq \e_0$ for all $k\geq k_0$ and:
$$\norm{\hat A_k}{W^{1,2}(B^4)}\leq \norm{\hat A_k-A}{W^{1,2}(B^4)}+\norm{A}{W^{1,2}(B^4)}\leq 2\e_0=\e.$$
Thus, for each $k\geq k_0$ we can apply Lemma \ref{low:p-approxthm} in order to obtain a perturbed sequence $A_k$ satisfying the integrability condition $F_{A_k}^{0,2}=0$ and there exists a constant $C>0$ such that 
$$\norm{A_k-\hat{A}_k}{W^{1,2}}\leq C\norm{F_{\hat A_k}^{0,2}}{L^2}\rightarrow 0.$$
Thus, 
$$\begin{array}{rl}
\norm{A_k-A}{W^{1,2}}&\leq \norm{A_k-\hat{A}_k}{W^{1,2}} + \norm{\hat{A}_k-A}{W^{1,2}}\\
&\leq C\norm{F_{\hat A_k}^{0,2}}{L^2}+ \norm{A-\hat{A}_k}{W^{1,2}}\rightarrow 0
\end{array}$$
as $k\rightarrow \infty$. This concludes the statement.
\end{proof}

\subsection{Existence of holomorphic trivialisations}\label{sec:holo}

In this section we show an application of our density results and prove that under the integrability condition $F_A^{0,2} = 0$ we obtain the existence of local holomorphic trivialisations assuming low $W^{1,2}$ norm for $A$ as before. We state the result:

\begin{Th}
There exists $\e_0>0$ such that if $A \in W^{1,2}(\Omega^{1} B^4\otimes\mathfrak{u}(n))$ satisfies $\norm{A}{W^{1,2}(B^4)}\leq \e_0$, and the integrability condition $F^{0,2}_{A}= 0$. There exists $r>0$ and $g\in W^{2,q}(B_r^4, GL_n(\C))$ for all $q<2$ such that 
\be\label{holotriv}A^{0,1} = -\delb g \cdot g^{-1}\qquad \text{ in }B_{r}^4,\ee
and there exists a constant $C_q>0$ such that 
$$\norm{g-id}{W^{2,q}(B_r^4)}\leq C_q\norm{A}{W^{1,2}(B^4)}\qquad\text{and}\qquad\norm{g^{-1}-id}{W^{2,q}(B_r^4)}\leq C_q\norm{A}{W^{1,2}(B^4)}.$$
 Moreover, $A^g=h^{-1} \p h$ where $h=\ov{g}^Tg$.
\end{Th}
\addtocounter{Th}{-1}

This result is an analog of the real case framework. Indeed, the flatness condition $F_A = 0$ together with the compactness of the Lie group $G$ imply that $A= - dg \cdot g^{-1}$ where $g\in W^{2,2}\cap L^\infty$. This can be easily done by using Uhlenbeck's gauge extraction procedure \cite{uhlenbeck1982removable}. In the complex framework, however, due to the lack of compactness of the group $GL_n(\C)$, we fail to obtain $W^{2,2}\cap L^\infty$ regularity of $g$.
\\

\textbf{Strategy:}\\

Since this proof is quite technical, we start by describing the strategy. We will first prove in Proposition \ref{CP2extension} that we can extend a small perturbation of our connection 1-form $A$ to $\CP^2$ while also keeping the integrability condition \eqref{II.1}. Secondly, Lemma \ref{holo:gaugeestimates} shows that this extended form is holomorphically trivialisable in the sense of \eqref{holotriv}. Thirdly, Lemma \ref{stronger} proves a technical result which shows the existence of holomorphic trivialisations of forms that are more regular than $W^{1,2}$ and this will help us later to cancel the initial perturbation we have added. \\

By combining all these steps, we obtain in Theorem \ref{holo:mainthm} the existence of holomorphic trivialisation of our initial form $A^{0,1}$ in $B_r^4$ for some $r>0$. We conclude the section with Corollary \ref{holo:stability} which proves a stability result.\\

We can assume without loss of generality that the ball of radius $2$, $B_2^4$, is holomorphically embedded into $\CP^2$, by Kodaira's embedding theorem \cite{griffiths2014principles}. Before we start we need to prove the following technical proposition:

\begin{Prop}{\label{low:fredholm}}
There exists $\e>0$ such that for any $\tilde A\in W^{1,2}(\Omega^1 \CP^2\otimes \mathfrak{u}(n))$ satisfying the bound  $\norm{\tilde A}{W^{1,2}(\CP^2)}\leq \e$, the operator $L_{\tilde A}: W^{2,2}(\Omega^2 \CP^2\otimes M_n(\C)) \rightarrow L^2(\Omega^2 \CP^2\otimes M_n(\C))$ defined by
\be\label{low:la}L_{\tilde A}(\omega) = \delb\delbs \omega + [{\tilde A}^{0,1},\delbs\omega]\ee
is Fredholm and invertible.
\end{Prop}

\begin{proof}[Proof of Proposition \ref{low:fredholm}]
It follows from G\r arding's Inequality, that the operator $\delb\delbs$ is elliptic over $\CP^2$ (see \cite[p. 93]{griffiths2014principles}), and hence it is also Fredholm. By choosing $\e>0$ so that $\tilde A$ is small in norm, it follows that the operator $[{\tilde A}^{0,1},\delbs\cdot]$ has small operator norm. Hence, from the continuity of the index maps\cite[Theorem 4.4.2, p.185]{salamon2018functional}, we have that $L_{\tilde A}$ is Fredholm and has the same index as $\delb\delbs$ as an operator mapping $W^{2,2}( \CP^2, M_n(\C))$ to $L^2(\CP^2, M_n(\C))$.\\

It is well-known that there are no global nonzero holomorphic $(0,2)$-forms on $\CP^2$ \cite[p. 118]{griffiths2014principles}. This implies that $\delb\delbs$ is an invertible operator on the space of $(0,2)$-forms and consequently has index 0. Thus, it follows that $index(L_{\tilde A}) = index(\delb\delbs) = 0$. \\

It remains to show that $L_{\tilde A}$ has trivial kernel. Once we have shown this, we can use the zero index of $L_{\tilde A}$ in order to conclude that $L_{\tilde A}$ is invertible. Assume $\omega\in Ker L_{\tilde A}$. Hence, $\omega$ satisfies
$$\delb\delbs \omega = -[{\tilde A}^{0,1},\delbs \omega].$$
By the Fredholm Lemma, we obtain
$$\begin{array}{lll}
\norm{\omega}{W^{2,2}}&\leq& C\norm{\delb\delbs\omega}{L^2} \leq C(\norm{L_{\tilde A}(\omega)}{L^2} + \norm{[{\tilde A}^{0,1},\delbs\omega]}{L^2})\\
&\leq & C\norm{L_{\tilde A}(\omega)}{L^2} + \norm{{\tilde A}^{0,1}}{L^4}\norm{\delbs\omega}{L^4}\\
&\leq &C\norm{L_{\tilde A}(\omega)}{L^2} + C'\e\norm{\omega}{W^{2,2}}
\end{array}$$
for some constants $C,C'$. We can take the term $C'\e\norm{\omega}{W^{2,2}}$ on the left hand side of the inequality:
$$(1-C'\e)\norm{\omega}{W^{2,2}}\leq C\norm{L_{\tilde A}(\omega)}{L^2}.$$
Choosing $\e>0$ such that $1-C'\e>\frac{1}{2}$, then we can divide by the positive factor $1-C'\e$. We obtain the bound:
$$\norm{\omega}{W^{2,2}}\leq \frac{C}{1-C'\e} \norm{L_{\tilde A}(\omega)}{L^2}.$$
Because $\omega\in Ker L_{\tilde A}$, we have that $\omega = 0$. Since $\omega$ was arbitrarily chosen from the kernel, it follows that the kernel of $L_{\tilde A}$ is trivial: $Ker L_{\tilde A} = \{0\}$. This finishes the proof. 
\end{proof}

\vspace{4mm}
Having this result at our disposal, we can prove the existence of a $\CP^2$ extension of our connection form $A$, keeping the integrability condition (\ref{II.1}).

\begin{Prop}{\label{CP2extension}}
There exists $\e>0$ such that for any  $A\in W^{1,2}(\Omega^1 B^4\otimes\mathfrak{u}(n))$ satisfying $F_A^{0,2}=0$ and $\norm{A}{W^{1,2}(B^4)}<\e$, there exists $\tilde{A}\in W^{1,2}(\Omega^1 \CP^2\otimes\mathfrak{u}(n))$ that satisfies $F_{\tilde{A}}^{0,2} = 0$ in $\CP^2$ and $\omega\in W^{2,2}(\Omega^{0,2} \CP^2\otimes M_n(\C))$  such that $\tilde{A}^{0,1} = A^{0,1}+\vartheta \omega$ in $B^4$.\\

Moreover, $\omega$ satisfies the estimate $\norm{\omega}{W^{2,2}(\CP^2)}\leq C\norm{A}{W^{1,2}}$ for some constant $C>0$.
\end{Prop}

\begin{proof}[Proof of Proposition \ref{CP2extension}]\ \\
\textit{Step 1}. We can decompose $A$ into it's $(0,1)$ and $(1,0)$ parts: $A = A^{0,1} - \ov{A^{0,1}}^T$ where $$A^{0,1}= \alpha_1 d\ov{z_1} + \alpha_2d\ov{z_2}$$ and $\alpha_i\in W^{1,2}(B^4,\mathfrak{u}(n))$ for $i=1,2$. We extend each $\alpha_i$ into $B_2^4$ to a compactly supported function $\hat{\alpha}_i$, so that $\hat{\alpha}_i = 0$ in $B_2^4\setminus B_{3/2}^4$. Indeed, for each $i=1,2$ we solve:
$$\left\{\begin{array}{rlll}
\Delta \phi_i &=& 0 & \text{ in } B_{3/2}^4\setminus B_1^4\\
\phi_i &=& \alpha_i & \text{ on }\p B_1^4\\
\phi_i &=& 0 & \text{ on }\p B_{3/2}^4
\end{array}\right.$$
Such solutions exist by \cite[Remark 7.2, Chapter 2]{lions2012non} and satisfy
$$\norm{\phi_i}{W^{1,2}(B_{3/2}^4\setminus B_1^4)}\leq C \norm{\alpha_i}{H^{1/2}(\p B_1^4)}\leq C'\norm{\alpha_i}{W^{1,2}(B_1^4)}$$
for some constants $C,C'>0$. We can now define the following extensions on $B_2^4$:
$$\hat{\alpha}_i =\begin{cases}
\alpha_i&\text{ in }B_1^4\\
\phi_i&\text{ in }B_{3/2}^4\setminus B_1^4\\
0&\text{ in }B_2^4\setminus B_{3/2}^4.
\end{cases}$$
By the construction of $\phi_i$, the functions $\hat{\alpha}_i$ are well-defined $W^{1,2}(B_2^4)$ Sobolev functions that satisfy the estimate:
$$\norm{\hat{\alpha}_i}{W^{1,2}(B_2^4)}\leq C\norm{\alpha_i}{W^{1,2}(B_1^4)}.$$
Define the $(0,1)$-form $\hat{A}^{0,1} = \hat{\alpha}_1 d\ov{z_1} + \hat{\alpha}_2d\ov{z_2}$ and $$\hat{A} := \hat{A}^{0,1} -  \ov{\hat{A}^{0,1}}^T \in W^{1,2}(\Omega^1 B_2^4\otimes\mathfrak{u}(n)).$$
By covering $\CP^2\setminus B_2^4$ with coordinate charts, we can trivially extend $\hat{A}$ by $0$ on $\CP^2\setminus B_2^4$. Thus, we have obtained $\hat{A}\in W^{1,2}(\Omega^1 \CP^2\otimes\mathfrak{u}(n))$ and there exists a constant $\hat C>0$ such that $\norm{\hat A}{W^{1,2}}\leq \hat C\norm{A}{W^{1,2}}.$\\

\textit{Step 2}. It remains to perturb the form $\hat{A}$ so that we obtain the integrability condition. This can be done by finding a $(0,2)$-form solution $\omega$ to the integrability condition:
$$F_{\hat A + \delbs\omega - \ov{\delbs\omega}^T}^{0,2} = 0.$$
This amounts to solving the following PDE globally on the complex projective space $\CP^2$:

\be\label{holo:bilap}
\begin{array}{rll}
\delb\delbs\omega +  [\hat{A}^{0,1},\delbs\omega]&= -\delbs\omega\wedge \delbs\omega -  F_{\hat{A}}^{0,2}
\end{array}\ee

where $\omega$ is a $(0,2)$ form on $\CP^2$. Using the invertibility of the operator $L_{\hat A}$ proven in Proposition \ref{low:fredholm}, we can solve equation (\ref{holo:bilap}) using a fixed point method. This is done by mimicking the procedure we have employed before, in Lemma \ref{low:p-approxthm}. Indeed, consider the sequence given by:

$$\begin{array}{lll}
L_{\hat A}(\omega_1) &=& -F_{\hat A}^{0,2}\\
L_{\hat A}(\omega_2) &= & - \delbs \omega_1 \wedge \delbs\omega_1-F_{\hat A}^{0,2}\\
\cdots\\
L_{\hat A}(\omega_k) &=& - \delbs\omega_{k-1}\wedge \delbs\omega_{k-1} - F_{\hat A}^{0,2}\\
\cdots
\end{array}$$

By showing that the sequence $\omega_k$ converges strongly in $W^{2,2}$, we obtain a $W^{2,2}$ solution to the required equation (\ref{holo:bilap}). Since $L_{\hat A}$ is invertible as an operator $W^{2,2}$ to $L^2$, it is clear that existence holds for each $\omega_k$, $k\geq 1$. We need to show that the sequence $\{\omega_k\}_{k=1}^\infty$ is a Cauchy in $W^{2,2}$.\\

Let $e_0 :=C \norm{F_{\hat A}^{0,2}}{L^2(\CP^2)}$, where $C>0$ is the constant appearing in Fredholm inequality:
$$\norm{\phi}{W^{2,2}(\CP^2)}\leq C\norm{L_{\hat A}(\phi)}{L^2(\CP^2)}.$$

\textbf{Claim}. $\{\omega_k\}_{k=1}^\infty$ is a Cauchy sequence in $W^{2,2}(\CP^2)$.\\\\
We first show by induction the uniform bound on the sequence $\norm{\omega_k}{W^{2,2}(\CP^2)}\leq 2\e_0$. By the Fredholm Lemma \cite[Lemma 4.3.9]{salamon2018functional} we have that
$$\norm{\omega_1}{W^{2,2}(\CP^2)} \leq C\norm{L_{\hat A}(\omega_1)}{L^2(\CP^2)}= C\norm{F_{\hat A}^{0,2}}{L^2(\CP^2)} = \varepsilon_0<2\varepsilon_0$$
Let $k\geq 1$. By the Sobolev embedding $W^{1,2}(\CP^2)\hookrightarrow L^4(\CP^2)$ there exists a constant $C_1>0$ so that
$$
\norm{\delbs \omega_k}{L^4(\CP^2)}\leq C_1\norm{\delbs \omega_k}{W^{1,2}(\CP^2)}\leq C_1^2 \norm{\omega_k}{W^{2,2}(\CP^2)}.$$ Thus, the following inequalities follow:
$$\begin{array}{rl}
\norm{\omega_{k+1}}{W^{2,2}(\CP^2)} &\leq  C\norm{L_{\hat A}(\omega_{k+1})}{L^2(\CP^2)}\\
&\leq C\norm{\delbs\omega_{k}\wedge\delbs\omega_{k}}{L^2(\CP^2)} +C\norm{F_{\hat A}^{0,2}}{L^2(\CP^2)}\\
&\leq C\norm{\omega_{k}}{L^4(\CP^2)}^2 +\e_0\\
&\leq C\cdot C_1^4\norm{\omega_{k}}{W^{2,2}(\CP^2)}^2 +\e_0
\end{array}$$
By the induction hypothesis, we assume the bound $\norm{\omega_k}{W^{2,2}(\CP^2)}<2\e_0$. Thus, 
$$\norm{\omega_{k+1}}{W^{2,2}(\CP^2)}\leq 4(C\cdot C_1^4)\e_0^2+\e_0.$$
Having chosen $\e>0$ such that $4(C\cdot C_1^4)\e<1$ and  $\norm{\hat A}{W^{1,2}(\CP^2)}\leq \hat C\norm{A}{W^{1,2}(B^4)}\leq \e$, it follows that $4(C\cdot C_1^4)\e_0<1$ and we conclude
$$\norm{\omega_{k+1}}{W^{2,2}(\CP^2)}\leq 2\e_0.$$
By induction, we have proven that we have a uniform bound for the sequence of $2$-forms $\{\omega_k\}$:
$$\norm{\omega_k}{W^{2,2}(\CP^2)}\leq 2\varepsilon_0.$$ for all $k\geq 1$.
It remains to show that $\{\omega_k\}$ is a Cauchy sequence. Let $k\geq 2$. Thus, we derive the following bounds from the recurrence relation satisfied by the sequence:
$$\begin{array}{rl}
\norm{\omega_{k+1} - \omega_{k}}{W^{2,2}(\CP^2)} &\leq C\norm{L_{\hat A}(\omega_{k+1}-\omega_k)}{L^2(\CP^2)}\\
&\leq C\norm{\delbs(\omega_k-\omega_{k-1})\wedge\delbs\omega_k}{L^2(\CP^2)}+
C\norm{\delbs\omega_{k-1}\wedge \delbs(\omega_k-\omega_{k-1})}{L^2(\CP^2)}\\
&\leq 4C\cdot  C_1^2\e_0 \norm{\omega_k-\omega_{k-1}}{W^{2,2}(\CP^2)}
\end{array}
$$
To simplify notation, denote $\e_1 := 4C\cdot C_1^2\e_0<1$. We further expand our estimate above:
$$
\norm{\omega_{k+1} - \omega_{k}}{W^{2,2}(\CP^2)} \leq \e\norm{\omega_k-\omega_{k-1}}{W^{2,2}(\CP^2)}\leq \ldots \leq \e_1^k\norm{\omega_1-\omega_{0}}{W^{2,2}(\CP^2)}\leq 4\e_1^k\e_0^2
$$
Let $\ell>k>0$. It follows that
$$\begin{array}{rl}
\norm{\omega_{\ell} - \omega_{k}}{W^{2,2}(\CP^2)}& \leq \norm{\omega_{\ell} - \omega_{\ell-1}}{W^{2,2}(\CP^2)} +  \norm{\omega_{\ell-1} - \omega_{k}}{W^{2,2}(\CP^2)}\\
&\leq 4\e_1^{\ell-1}\e_0^2  + \norm{\omega_{\ell-1} - \omega_{\ell-2}}{W^{2,2}(\CP^2)} + \norm{\omega_{\ell-2} - \omega_{k}}{W^{2,2}(\CP^2)}\\
&\leq 4\e_1^{\ell-1} \e_0^2 +4 \varepsilon_1^{\ell-2 } \varepsilon_0^2 +\ldots + 4\e_1^k \varepsilon_0^2\\
& = 4\e_0^2 \cdot \e_1^k\cdot \frac{1-\e_1^{\ell-k}}{1-\e_1} \leq \e_1^k
\end{array}$$
This is clearly a Cauchy sequence by the inequality above and the claim is proven. \\

Hence, since $\{\omega_k\}_{k=1}^\infty$ is a Cauchy sequence in the Banach space $W^{2,2}(\Omega^{0,2}\CP^2\otimes M_n(\C))$, it has a limit $\omega$ and converges strongly in $W^{2,2}$ to it. Hence, by defining $\tilde A = \hat A + \delbs \omega - \ov{\delbs \omega}^T$, we obtain a skew-Hermitian $1$-form, satisfying $F_{\tilde A}^{0,2}=0$ such that $\tilde A^{0,1} = \hat A^{0,1} +\vartheta \omega = A^{0,1} + \vartheta\omega$ in $B^4$.\\

Moreover, by convergence, we have that the uniform bound is satisfied by the limiting form $\omega$, indeed $\norm{\omega}{W^{2,2}}<2\e_0 = 2C\norm{F_{\hat A}}{L^2}$. By construction of $\hat A$, it is clear that there exists a constant $C'>0$ so that $\norm{F_{\hat A}}{L^2}\leq C'\norm{A}{W^{1,2}}$. This leads to the required estimate on $\omega$, $\norm{\omega}{W^{2,2}}\leq C\norm{A}{W^{1,2}}$, where $C>0$ is some constant.

\end{proof}

\begin{Lm}\label{holo:gaugeestimates}
There exists $\e>0$ such that for any form $\tilde A\in W^{1,2}(\Omega^1 \CP^2\otimes \mathfrak{u}(n))$ satisfying the integrability condition (\ref{II.1}) and $\norm{\tilde A}{W^{1,2}}<\e$, there exists a gauge $\tilde g\in W^{2,q}(\CP^2,GL_n(\C))$ for any $q<2$ such that
$$\tilde A^{0,1} = -\delb \tilde g\cdot \tilde g^{-1}$$
and there exists a constant $C_q>0$ such that 
$$\norm{\tilde g-id}{W^{2,q}(\CP^2)}\leq C_q\norm{\tilde A}{W^{1,2}(\CP^2)}$$
and
$$\norm{\tilde g^{-1}-id}{W^{2,q}(\CP^2)}\leq C_q\norm{\tilde A}{W^{1,2}(\CP^2)}.$$
\end{Lm}

\begin{proof}[Proof of Lemma \ref{holo:gaugeestimates}]\

\textit{Step 1}. We prove this result using a fixed point argument. Indeed, consider the linear operator:
$$T:W^{1,2}(\CP^2, M_n(\C))\rightarrow W^{1,2}(\CP^2, M_n(\C))$$
given by $$T(\tilde g) = -\delbs N(\tilde A^{0,1}\tilde g) + id.$$
We verify that this operator is well-defined. It follows from the G\aa rding inequality on $\CP^2$ (see for example \cite[p. 93]{griffiths2014principles}) that we have the elliptic estimate
$$\norm{\delbs N(\tilde A^{0,1}\tilde g)}{W^{1,2}(\CP^2)}\leq C\norm{\delb\delbs N(\tilde A^{0,1}\tilde g)}{L^2(\CP^2)}$$
for some constant $C$. Moreover, using the $\delb$-Hodge decomposition we can decompose $\tilde A^{0,1}\tilde g$ as such:
\be\label{low:hodgedec}
\tilde A^{0,1}\tilde g = \delbs\delb N(\tilde A^{0,1}\tilde g) + \delb\delbs N(\tilde A^{0,1}\tilde g)\ee
and because $\delbs\perp\delb$, it follows that
$$\norm{\tilde A^{0,1}\tilde g}{L^2} = \norm{ \delbs\delb N(\tilde A^{0,1}\tilde g) }{L^2} + \norm{ \delb\delbs N(\tilde A^{0,1}\tilde g)}{L^2}.$$
Consequently,
$$\norm{ \delb\delbs N(\tilde A^{0,1}\tilde g)}{L^2} \leq \norm{\tilde A^{0,1}\tilde g}{L^2}.$$
Putting the above together, we obtain:
$$\norm{\delbs N(\tilde A^{0,1}\tilde g)}{W^{1,2}(\CP^2)} \leq C\norm{\tilde A^{0,1}\tilde g}{L^2(\CP^2)}\leq C\norm{\tilde A^{0,1}}{L^4(\CP^2)}\norm{\tilde g}{L^4(\CP^2)}.$$
Furthermore, using the Sobolev embedding in $4$-dimensions $W^{1,2}\hookrightarrow L^4$, there exists a constant $C'$ so that
\be\label{holo:estcp2} \norm{\delbs N(\tilde A^{0,1}\tilde g)}{W^{1,2}} \leq  C'\norm{\tilde A^{0,1}}{W^{1,2}}\norm{\tilde g}{W^{1,2}}.\ee
Thus, the operator $T$ is well-defined, mapping $W^{1,2}$ functions to $W^{1,2}$ function.\\

We can now show that $T$ has a unique fixed point. Consider $\tilde g_1,\tilde g_2\in W^{1,2}(\CP^2, M_n(\C))$. Then
$$\norm{T(\tilde g_1)-T(\tilde g_2)}{W^{1,2}(\CP^2)} = \norm{\delbs N(\tilde A^{0,1}(\tilde g_1-\tilde g_2))}{W^{1,2}(\CP^2)}.$$
Using the above inequalities, we obtain 
$$\norm{\delbs N(\tilde A^{0,1}(\tilde g_1-\tilde g_2))}{W^{1,2}(\CP^2)}\leq C' \norm{\tilde A^{0,1}}{W^{1,2}(\CP^2)}\norm{\tilde g_1-\tilde g_2}{W^{1,2}(\CP^2)}$$
and we can choose $\e>0$ such that the bound $\norm{\tilde A^{0,1}}{W^{1,2}(\CP^2)}<\e$ is small gives that the factor $C' \norm{\tilde A^{0,1}}{W^{1,2}(\CP^2)}$ is strictly smaller than $1$. Hence, $T$ is a contraction operator and there exists a unique fixed point $\tilde g\in W^{1,2}(\CP^2, M_n(\C))$, $T(\tilde g)=\tilde g$. Thus, we have 
$$\delb\tilde g=-\delb \delbs N (\tilde A^{0,1}\tilde g).$$

\textit{Step 2}. We can now show that the equation above coupled with the integrability condition satisfied by $\tilde A^{0,1}$, imply that $\tilde g$ solves the required PDE: $\delb \tilde g=-\tilde A^{0,1}\tilde g$. The $\delb$-Hodge decomposition (\ref{low:hodgedec}) gives
\be\label{low:recurr}\delb \tilde g= -\tilde A^{0,1}\tilde g + \delbs\delb N(\tilde A^{0,1}\tilde g).\ee
Since the operators $N$ and $\delb$ compute, $N\delb = N\delb$ (see \cite{chen2001partial}), we can further compute the term $\delbs\delb N(\tilde A^{0,1}\tilde g)$:
$$\delbs\delb N(\tilde A^{0,1}\tilde g) = \delbs N\delb(\tilde A^{0,1}\tilde g) = \delbs N(\delb \tilde A^{0,1}\tilde g - \tilde A^{0,1}\wedge \delb \tilde g).$$
Using the above equation (\ref{low:recurr}), the equation becomes
$$\delbs\delb N(\tilde A^{0,1}\tilde g)=  \delbs N(\delb \tilde A^{0,1}\tilde g - \tilde A^{0,1}\wedge \delb \tilde g) =  \delbs N(\delb \tilde A^{0,1}\tilde g + \tilde A^{0,1}\wedge \tilde A^{0,1}\tilde g - \tilde A^{0,1}\wedge \delbs\delb N(\tilde A^{0,1}\tilde g)).$$
Since $\tilde A$ satisfies the integrability condition $F_{\tilde A}^{0,2} = 0$, we have the recurrence relation:
\be\label{low:recurr2}\delbs\delb N(\tilde A^{0,1}\tilde g) = -\delbs N(\tilde A^{0,1}\wedge\delbs\delb N(\tilde A^{0,1}\tilde g)).\ee
Thus, it is natural to consider the operator
$$\mathcal L:L^2(\Omega^1 \CP^2\otimes M_n(\C)) \rightarrow L^2(\Omega^1 \CP^2\otimes M_n(\C))$$
$$V\mapsto -\delbs N (\tilde A^{0,1}\wedge V).$$
We need to establish whether $\mathcal L$ is a well-defined operator and find its fixed points in order to analyse equation (\ref{low:recurr2}). By the Sobolev embedding $W^{1,4/3}\hookrightarrow L^2$ it follows that
$$\norm{\mathcal L(V)}{L^2(\CP^2)}\leq C\norm{\mathcal L(V)}{W^{1,4/3}(\CP^2)}$$
for some constant $C>0$. We also have that 
$$\norm{\nabla \mathcal  L(V)}{L^{4/3}(\CP^2)}\leq C\norm{\nabla^2 N(\tilde A^{0,1}\wedge V)}{L^{4/3}(\CP^2)}$$
and consequently, since $ N(\tilde A^{0,1}\wedge V)$ is a $(0,2)$-form in $4$-dimensions, the elliptic estimate holds:
$$\norm{\nabla^2 N(\tilde A^{0,1}\wedge V)}{L^{4/3}(\CP^2)}\leq C\norm{\delb\delbs N(\tilde A^{0,1}\wedge V)}{L^{4/3}(\CP^2)} = C\norm{\tilde A^{0,1}\wedge V}{L^{4/3}(\CP^2)}.$$
By the H\"older inequality and the estimates above, it immediately follows that:
$$\norm{\mathcal  L(V)}{L^2(\CP^2)}\leq C\norm{\tilde A^{0,1}}{L^4(\CP^2)}\norm{V}{L^2(\CP^2)}.$$
Similarly as before, this means that $\mathcal  L$ is a well-defined contraction operator and has a unique fixed point. In particular, $0$ is its fixed point. We know from equation  (\ref{low:recurr2})   that $\delbs\delb N(\tilde A^{0,1}\tilde g)$ is also a fixed point for $\mathcal  L$ and, thus, we have that the term $\delbs\delb N(\tilde A^{0,1}\tilde g) $ vanishes. In particular, the equation is solved:
$$\delb \tilde g=-\tilde A^{0,1}\tilde g.$$
\vspace{2pt}

\textit{Step 3}. It remains to show that $\tilde g\in W^{2,q}(\CP^2,GL_n(\C))$ for all $q<2$ and the required estimates. Let $q<2$. We know that $\tilde g$ is a $W^{1,2}$ function and satisfies:
$$\tilde g-id = \delbs N(\tilde A^{0,1}\tilde g).$$
Since $\tilde g$ is a fixed point of $T$, then it satisfies the estimate (\ref{holo:estcp2}), which means:
$$\norm{\tilde g-id}{W^{1,2}(\CP^2)}\leq C\norm{\tilde A^{0,1}}{W^{1,2}(\CP^2)}\norm{\tilde g}{W^{1,2}}\leq C\norm{\tilde A^{0,1}}{W^{1,2}(\CP^2)}\norm{\tilde g - id}{W^{1,2}} + C\norm{\tilde A^{0,1}}{W^{1,2}(\CP^2)}.$$
Because $\norm{\tilde A^{0,1}}{W^{1,2}(\CP^2)}<\e$, where $\e$ is small, then there exists a constant $C>0$ such that
$$\norm{\tilde g-id}{W^{1,2}(\CP^2)}\leq C\norm{\tilde A^{0,1}}{W^{1,2}(\CP^2)}.$$
Since $\tilde g$ satisfies this estimate, we can bootstrap using Lemma \ref{bootstrapcriticalgauge} and Remark \ref{bootstrapsubcritical}(i), from which follow the required estimate and regularity: 
\be\label{holo:estcp2}\norm{\tilde g-id}{W^{2,q}(\CP^2)}\leq C_q\norm{\tilde A}{W^{1,2}(\CP^2)},\ee
for some constant $C_q>0$. \\

We need to show that $\tilde g$ is in $GL_n(\C)$ over $\CP^2$ and that its inverse satisfies a similar estimate as \eqref{holo:estcp2}.
Arguing in a similar way, we can show that there exists $\tilde u\in W^{2,q}(\CP^2, GL_n(\C))$ for any $q<2$ such that 
$$\delb \tilde u = \tilde u\tilde A^{0,1}$$ 
and $\norm{\tilde u -id}{W^{2,q}(\CP^2)}\leq C_q\norm{\tilde A}{W^{1,2}(\CP^2)}$ for some constant $C_q>0$. In particular, we have that $\delb (\tilde u \tilde g) = 0.$ Hence, there exists a holomorphic function $\tilde h$ such that $\tilde u\tilde g=\tilde h$. However, since the only holomorphic functions on $\CP^2$ are the constant ones \cite{griffiths2014principles}, then $\tilde h$ is a constant. \\

We can pick $q_0<2$ so that we obtain the Sobolev embedding $W^{2,q_0}\hookrightarrow L^\infty$ on any $3$-dimensional hypersurface. Moreover, by \cite[Section 4.8.2, Theorem 1]{runst2011sobolev}, there exists $q_1\in (q_0,2)$ such that $\tilde u,\tilde g\in W^{2,q_1}(\CP^2, M_n(\C))$ and $\tilde u\tilde g\in W^{2,q_0}$ and $\norm{\tilde u\tilde g-id}{W^{2,q_0}(\CP^2)}\leq C_{q_0}\norm{\tilde A}{W^{1,2}(\CP^2)}$ for some constant $C_{q_0}>0$. By Fubini, there exists a radius $r>0$ and $z_0\in \CP^2$ such that 
$$\norm{\tilde u\tilde g-id}{W^{2,q_0}(\p B_r^4(z_0))}<2C_{q_0}'\norm{\tilde A}{W^{1,2}(\CP^2)}$$
where $B_r^4(z_0)$ is holomorphically embedded in $\CP^2$. Thus, by the embedding of $W^{2,{q_0}}$ into $L^\infty$ in $3$-dimensions, there exists a constant $C_{q_0}''>0$ so that $\norm{\tilde u\tilde g-id}{L^\infty(\p B_r^4(z_0))}\leq C_{q_0}''\norm{\tilde A}{W^{1,2}(\CP^2)}$. Having chosen $\e>0$ such that $\norm{\tilde A}{W^{1,2}(\CP^2)}<\e$ is small enough, we obtain that $\tilde h = \tilde u\tilde g\in GL_n(\C)$ over $\p B_r(z_0)$. However, because $\tilde h$ is a constant, then $\tilde h\in GL_n(\C)$ and satisfies the estimate:
$$\norm{\tilde h - id}{L^\infty(\CP^2)}\leq C\norm{\tilde A}{W^{1,2}(B^4)},$$
for some constant $C>0$.\\

Hence, we can define $\tilde g^{-1} := \tilde h^{-1} \tilde u$. Since $\tilde g^{-1}\tilde g = id$ by construction, we obtain that $\tilde g$ maps into $GL_n(\C)$. Moreover, it follows that $\tilde g^{-1}\in W^{2,q}(\CP^2,GL_n(\C))$ for any $q<2$, and by the estimates on $\tilde u$, we obtain that for each $q<2$ there exists a constant $C_q>0$ such that
$$\norm{\tilde g^{-1}-id}{W^{2,q}(\CP^2)}\leq C_q\norm{A}{W^{1,2}(\CP^2)}.$$
This conclude the  proof.
\end{proof}

Before proving the existence of a local holomorphic trivialisation for our initial $W^{1,2}$ form, we need to show a stronger version of existence. We consider forms of small norm in $W^{1,p}$, $p>3$. This will be a useful result for our final theorem.

\begin{Lm}\label{stronger}
Let $p>3$. There exists $\e>0$ such that for any $\omega\in W^{1,p}(\Omega^{0,1} B^4\otimes M_n(\C))$ satisfying $F_{\omega}^{0,2}=0$ and $\norm{\omega}{W^{1,p}(B^4)}\leq \e$, there exists $r\in (1/2,1)$ and a gauge $u\in W^{2,p}(B_r^4, GL_n(\C))$ so that
$$\omega = -\delb u\cdot u^{-1}\qquad \text{ in }B_r^4,$$
with estimates $\norm{u-id}{W^{2,p}(B_r^4)}\leq C\norm{\omega}{W^{1,p}(B^4)}$ and $\norm{u^{-1}-id}{W^{2,p}(B_r^4)}\leq C\norm{\omega}{W^{1,p}(B^4)}$.
\end{Lm}

\begin{Rm}The reader can note the fact that the  technique to solve this Lemma is similar to the ideas used in the previous one. However, this proof will rely more on regularity results from the literature on the analysis of several complex variables.
\end{Rm}

\begin{proof}[Proof of Lemma \ref{stronger}]\ 

\textit{Step 1.} Let $q=4p/(4-p)$. We show the existence of a gauge $u\in GL_n(\C)$ that "almost" solves our equation modulo a perturbation term. Indeed, in Step 2 we can show that the perturbation term vanishes and consequentially $u$ is the solution.  Let $T_1, T_2$ be the operators defined as in (\ref{app:T1}) and (\ref{app:T_q}). Note that we can extend $T_1$ and $T_2$ to operators defined on Sobolev spaces by density arguments.\\

We define the operator $$\mathcal H:L^\infty(B^4,M_n(\C))\cap \{f: \delb f \in L^{q}\}\rightarrow L^\infty(B^4,M_n(\C))\cap \{f: \delb f\in L^{q}\}$$
given by $$\mathcal H(u) = id+T_1(-\omega\cdot u).$$ 

\textbf{Claim.} $\mathcal H$ is well-defined. \\

Since $T_1$ takes $(0,1)$-forms to functions we only need to check that $\mathcal H$ maps $L^\infty\cap \{f: \delb f \in L^{q}\} $to $L^\infty\cap \{f: \delb f \in L^{q}\}$. By the Sobolev embedding $W^{1,p}\hookrightarrow L^{q}$, there exists a constant $C_1>0$ so that
$$\norm{\omega}{L^q(B^4)}\leq C_1 \norm{\omega}{W^{1,p}(B^4)}.$$
The assumption $p>3$ implies that $q=4p/(4-p)>12$. Consequently, for $u \in L^\infty(B^4,M_n(\C))\cap \{f: \delb f \in L^{q}\}$ we have $\omega\cdot u\in L^{q}$. Moreover $\delb (\omega\cdot u)\in L^{q/2}$. We prove this. 
\be\begin{array}{lll}
\norm{\delb(\omega\cdot u)}{L^{q/2}(B^4)}
&\leq& \norm{\delb\omega}{L^{q/2}(B^4)}\norm{u}{L^\infty(B^4)} + \norm{\omega\wedge \delb u}{L^{q/2}(B^4)}\\
&\leq& \norm{\delb\omega}{L^{q/2}(B^4)}\norm{u}{L^\infty(B^4)}  + \norm{\omega}{L^q(B^4)}\norm{\delb u}{L^q(B^4)}
\end{array}\ee
Crucially, we have that $F^{0,2}_\omega = \delb\omega+\omega\wedge\omega= 0$. Because $\omega\in L^q(B^4)$, then $\delb\omega\in L^{q/2}$. Thus,
 \be\label{holo:prodest}\begin{array}{ll}
\norm{\delb(\omega\cdot u)}{L^{q/2}(B^4)} &\leq  \norm{\omega}{L^{q}(B^4)}^2\norm{u}{L^\infty(B^4)}  + \norm{\omega}{L^q(B^4)}\norm{\delb u}{L^q(B^4)}\\
&\leq  C_1\e\norm{\omega}{L^{q}(B^4)}\norm{u}{L^\infty(B^4)}  + \norm{\omega}{L^q(B^4)}\norm{\delb u}{L^q(B^4)}\\
&\leq C_1^2\norm{\omega}{W^{1,p}(B^4)}\left(\norm{u}{L^\infty(B^4)}  + \norm{\delb u}{L^q(B^4)}\right).
\end{array}
\ee
where we have implicitly used the fact that we can choose $\e<1$ so that $\norm{\omega}{W^{1,p}(B^4)}\leq \e$.
Hence, we have shown that $\delb (\omega\cdot u)\in L^{q/2}$. Taking into account that $q/2>6$ and the embedding $W^{1,{q/2}}\hookrightarrow L^q$, we can apply Proposition \ref{app:T1res} to the $(0,1)$-form $\omega\cdot u$ and obtain the estimate:
\be\label{holo:t1_est}\norm{T_1(\omega\cdot u)}{L^\infty(B^4)} + \norm{\delb T_1(\omega\cdot u)}{L^{q}(B^4)}\leq C\left(\norm{\omega\cdot u}{L^q(B^4)} + \norm{\delb(\omega\cdot u)}{L^{q/2}(B^4)}\right).\ee
This shows that, $\mathcal H$ is well-defined, since the operator $T_1(\omega\cdot )$ is a well-defined map from   $L^\infty\cap \{f: \delb f \in L^q\} $ to $L^\infty\cap \{f: \delb f \in L^q\}$. We have proven the claim. \\

Next, we show that $\mathcal H$ has a fixed point. From \eqref{holo:t1_est}, it follows that
$$\norm{T_1(\omega\cdot u)}{L^\infty(B^4)} + \norm{\delb T_1(\omega\cdot u)}{L^q(B^4)}\leq C\left(\norm{\omega}{L^q(B^4)}\norm{u}{L^\infty(B^4)} + \norm{\delb(\omega\cdot u)}{L^{q/2}(B^4)}\right).$$
Since $\omega$ is less than $\e$ in $W^{1,p}$ norm and using \eqref{holo:prodest}, for any $u_1,u_2\in L^\infty(B^4,M_n(\C))\cap \{f: \delb f \in L^{q}\}$ we have:
$$\begin{array}{ll}&\norm{\mathcal H(u_1) - \mathcal H(u_2)}{L^\infty(B^4)} +\norm{\delb \mathcal H(u_1) - \delb \mathcal H(u_2)}{L^q(B^4)}\\
&=\norm{T_1(-\omega\cdot (u_1-u_2))}{L^\infty(B^4)} + \norm{\delb T_1(-\omega\cdot (u_1-u_2))}{L^\infty(B^4)} \\
&\leq C\e \left(\norm{u_1-u_2}{L^\infty(B^4)} + \norm{\delb(u_1-u_2)}{L^q(B^4)}\right).
\end{array}$$
Choosing $\e>0$  such that $C\e<1$ , we obtain that $\mathcal H$ is a contraction and therefore there exists $u\in L^\infty(B^4, M_n(\C)) \cap \{f: \delb f \in L^{q}\}$ satisfying $$u = id+T_1(-\omega\cdot u) = \mathcal H(u).$$ This fixed point "almost" solves the required equation. We will show in the next step that the error we obtain vanishes in light of the integrability condition $F^{0,2}_\omega = 0$.\\

\textit{Step 2}. Having obtained this fixed point, we show that $u$ satisfies $\delb u = -\omega \cdot u$. Since we have proven that $u-id =T_1(-\omega\cdot u)$, we get $\delb u = \delb T_1(-\omega \cdot u) \in L^q$. We can apply Theorem \ref{app:repr} from the Appendix to get the integral representation of $-\omega\cdot u$:
$$
-\omega\cdot u = \delb T_1(-\omega\cdot u) +T_2(\delb (-\omega\cdot u))$$
and expand the last term in the following way:
$$\begin{array}{rl}
T_2(\delb (-\omega\cdot u)) &= T_2(- \delb \omega \cdot u + \omega\wedge \delb u)\\
&=T_2(- \delb  \omega \cdot u +  \omega\wedge \delb T_1(-\omega\cdot u) )\\
&=T_2(- \delb  \omega \cdot u + \omega\wedge (-\omega \cdot u-  T_2(\delb (-\omega \cdot u)))\\
&=T_2(- (\delb \omega+ \omega\wedge \omega)u-  \omega\wedge T_2(\delb (-\omega\cdot u))).
\end{array}$$
By using the fact that $\omega$ satisfies the integrability condition $F_{\omega}^{0,2} = \delb \omega +  \omega \wedge \omega= 0$, we obtain:
$$
T_2(\delb (-\omega\cdot u)) = T_2(\omega \wedge T_2(\delb (-\omega \cdot u))).$$
We want to show that this recurrence equation implies that $T_2(\delb (-\omega\cdot u)) = 0$. From Proposition \ref{app:T2res}, $T_2$ is a well-defined operator mapping $L^s$ to $W^{1,s}$ for any $s>1$ and the following estimate holds:
$$\begin{array}{rl}
\norm{T_2(\delb (-\omega\cdot u))}{L^\infty(B^4)}&\leq C\norm{T_2(\delb (-\omega\cdot u))}{W^{1,q}(B^4)}\\
&\leq C\norm{\omega\wedge T_2(\delb (-\omega\cdot u))}{L^q(B^4)}\\
&\leq C\norm{\omega}{L^q(B^4)}\norm{T_2(\delb (-\omega \cdot u))}{L^\infty(B^4)}\\
&\leq CC_1\e \norm{T_2(\delb (-\omega\cdot u))}{L^\infty(B^4)},
\end{array}$$
where $C$ is the Sobolev constant given by the Sobolev embedding $W^{1,q}\hookrightarrow L^\infty$ ($q>6$) in 4 dimensions. Moreover, we also choose $\e>0$ so that $1-C\cdot C_1\e>0$ and 
$$(1-C\cdot C_1\e)\norm{T_2(\delb (-\omega\cdot u))}{L^\infty(B^4)}\leq 0.$$
Thus, $T_2(\delb (-\omega\cdot u)) = 0$ and we can conclude that the $\delb$-equation is solved:
$$\delb u = -\omega \cdot u\qquad\text{ in }B^4.$$

\textit{Step 3}. It remains to that $u\in GL_n(\C)$ and satisfies the required estimates. We have:
$$\norm{u-id}{L^\infty(B^4)} = \norm{\mathcal H(u)-id}{L^\infty(B^4)}\leq C\norm{\omega}{W^{1,p}(B^4)} \norm{u}{L^\infty(B^4)}\leq C\e\norm{u - id}{L^\infty(B^4)} + C\norm{\omega}{W^{1,p}(B^4)}.$$
Thus, since $\e>0$ is small, we get the $L^\infty$ bound:
$$\norm{u-id}{L^\infty(B^4)} \leq \frac{C}{1-C\e}\norm{\omega}{W^{1,p}(B^4)}.$$
Because we can assume that $1-C\e>\tfrac 1 2$ for $\e$ small enough, then $\norm{u-id}{L^\infty(B^4)} \leq 2C\norm{\omega}{W^{1,p}(B^4)}$. This implies that $u\in GL_n(\C)$. Remark \ref{bootstrapsubcritical}(iii) gives the existence of $r\in (1/2,1)$ and a constant $C>0$ such that
$$\norm{u-id}{W^{2,p}(B_r^4)}\leq C\norm{\omega}{W^{1,p}(B^4)}.$$ Moreover, since $u^{-1}$ exists, we have the following $L^\infty$ estimate:
$$\begin{array}{rl}
\norm{u^{-1}-id}{L^\infty(B^4)}& \leq \norm{u^{-1}-u^{-1}u}{L^\infty(B^4)}\leq
\norm{u^{-1}}{L^\infty(B^4)}\norm{u-id}{L^\infty(B^4)}\\
&\leq \norm{u^{-1}-id}{L^\infty(B^4)}\norm{u-id}{L^\infty(B^4)} + \norm{u-id}{L^\infty(B^4)}
\end{array}.$$
The estimate on $u$ also implies that the norm of $u-id$ in $L^\infty$ is small. Hence, the estimate of $u^{-1}$ then follows: $$\norm{u^{-1}-id}{L^\infty(B^4)}\leq \frac{  \norm{u-id}{L^\infty(B^4)}}{1- \norm{u-id}{L^\infty(B^4)}}\leq C\norm{\omega}{W^{1,p}(B^4)},$$
for some constant $C>0$. By Remark \ref{bootstrapsubcritical} applied to $u^{-1}$, we obtain a similar estimate. This finishes the proof.
\end{proof}
\vspace{4mm}

Having the results above at our disposal, we are ready to proceed at showing the existence of local holomorphic trivialisations in $B_r^4$ for some $r>0$.

\begin{Th}\label{holo:mainthm}
There exists $\e_0>0$ such that if $A \in W^{1,2}(\Omega^{1} B^4\otimes\mathfrak{u}(n))$ satisfies $\norm{A}{W^{1,2}(B^4)}\leq \e_0$, and the integrability condition $F^{0,2}_{A}= 0$. There exists $r>0$ and $g\in W^{2,q}(B_r^4, GL_n(\C))$ for all $q<2$ such that 
\be\label{holotriv}A^{0,1} = -\delb g \cdot g^{-1}\qquad \text{ in }B_{r}^4,\ee
and there exists a constant $C_q>0$ such that 
$$\norm{g-id}{W^{2,q}(B_r^4)}\leq C_q\norm{A}{W^{1,2}(B^4)}\qquad\text{and}\qquad\norm{g^{-1}-id}{W^{2,q}(B_r^4)}\leq C_q\norm{A}{W^{1,2}(B^4)}.$$
 Moreover, $A^g=h^{-1} \p h$ where $h=\ov{g}^Tg$.
\end{Th}

\begin{proof}[Proof of Theorem \ref{holo:mainthm}]
From Proposition \ref{CP2extension}, there exists a 1-form $\tilde A\in W^{1,2}(\Omega^1 \CP^2\otimes\mathfrak{u}(n))$ satisfying the integrability condition so that $\tilde A^{0,1} = A^{0,1}+\vartheta \omega$ in $B^4$, where $\omega\in W^{2,2}(\Omega^{0,2}\CP^2\otimes M_n(\C))$ with estimate $\norm{\omega}{W^{2,2}(\CP^2)}\leq \norm{A}{W^{1,2}(B^4)}$. This implies that 
\be\label{holo:mainthm:eq1}\norm{\tilde A^{0,1}}{W^{1,2}(\CP^2)}\leq C\norm{A}{W^{1,2}(B^4)}\ee for some constant $C>0$. \\

Lemma \ref{holo:gaugeestimates} applied to the form $\tilde A$ gives the existence of a gauge $\tilde{g}\in W^{2,q}(\CP^2, GL_n(\C))$ for all $q<2$ so that $$\delb \tilde{g} = -\tilde{A}^{0,1} \tilde{g}\qquad\text{ in }\CP^2$$
and for each $q<2$ there exists $C_q>0$ such that \be\label{holo:mainthm:eq2}\norm{\tilde g - id}{W^{2,q}(\CP^2)}\leq C_q\norm{\tilde A}{W^{1,2}(\CP^2)}\qquad\text{and}\qquad \norm{\tilde g^{-1} - id}{W^{2,q}(\CP^2)}\leq C_q\norm{\tilde A}{W^{1,2}(\CP^2)}.\ee
On the unit ball $B^4$ we can rewrite $\left(A^{0,1}\right)^{\tilde g}$ as such:
$$\left(A^{0,1}\right)^{\tilde g} =\tilde g^{-1}\delb \tilde g + \tilde g^{-1} A^{0,1}\tilde g = \tilde g^{-1}\delb \tilde g + \tilde g^{-1}\tilde A^{0,1}\tilde g - \tilde g^{-1}\vartheta \omega \tilde g = - \tilde g^{-1}\left(\vartheta \omega\right) \tilde g.$$ 
In order to find a gauge $g$ for $A^{0,1}$ that gives a holomorphical trivialisation, it remains to find a gauge change $u$ that cancels perturbation term $ - \tilde g^{-1}\left(\vartheta \omega\right) \tilde g$:
\be\label{holo:requiredeq}\delb u = \tilde g^{-1}\left(\vartheta \omega\right) \tilde g\cdot  u.\ee
We claim that the composition of gauges $\tilde g\cdot u$ satisfies the statement. \\

Since the Sobolev embedding $W^{2,q}\hookrightarrow L^{2q/(2-q)}$ holds for any $q<2$, it implies that $\tilde g,\tilde g^{-1}\in \bigcap\limits_{q<\infty} L^q$. Because $A$ and $\tilde A$ satisfy the integrability condition on $B^4$: $F_A^{0,2} = 0$ and $F_{\tilde A}^{0,2} = 0$, imply that $\omega\in W^{2,2}(\Omega^{0,2}B^4) $ satisfies the following PDE:
$$\frac{1}{2}\Delta \omega = -[A^{0,1},\vartheta\omega]-\vartheta\omega\wedge \vartheta\omega.$$
Proposition \ref{bootstrapmorrey} applied to this PDE improves on the regularity of $\omega$ inside $B^4$. Indeed, we have a much better regularity $\omega\in W^{2,q}_{loc}(B^4, M_n(\C))$ for any $q<4$. Sobolev embeddings yield: $$\vartheta\omega\in \bigcap\limits_{q<4} W_{loc}^{1,q}(B^4,M_n(\C))\hookrightarrow \bigcap\limits_{q<\infty}L_{loc}^q.$$ Putting together the regularity of $\vartheta\omega$, $\tilde g$ and $\tilde g^{-1}$ we can obtain the regularity of $\tilde g^{-1}\left(\vartheta\omega\right)\tilde g$:
\be\label{holo:emb}
\tilde g^{-1}\left(\vartheta\omega\right)\tilde g\in \bigcap\limits_{q<4}W^{1,q}_{loc}\hookrightarrow \bigcap\limits_{q<\infty}L_{loc}^q.\ee

Fix $p>3$ and $\delta>0$ small,. There exists $r_0\in (0,1)$ so that $\norm{\tilde g^{-1}\left(\vartheta\omega\right) \tilde g}{W^{1,p}(B_{r_0}^4)}<\delta$. This $(0,1)$-form also solves $F_ {\tilde g^{-1}\left(\vartheta\omega\right) \tilde g}^{0,2}=0$ in $B_{r_0}^4$. Hence, we apply Lemma \ref{stronger} to $\tilde g^{-1}\left(\vartheta\omega\right) \tilde g$ in $B_{r_0}^4$ (by rescaling) to get the existence of $r\in (r_0/2,r_0)$ and $u\in W^{2,p}(B_{r}^4, GL_n(\C))$ that solves the $\delb$-equation above (\ref{holo:requiredeq}):
$$\delb u = \tilde g^{-1}\left(\vartheta \omega\right) \tilde g\cdot  u\qquad \text{ in } B_{r}^4.$$
and satisfies the estimates \be\label{holo:mainthm:eq3}
\norm{u-id}{W^{2,p}(B_{r}^4)}\leq C\norm{ \tilde g^{-1}\left(\vartheta \omega\right) \tilde g}{W^{1,2}(B_{r}^4)}\leq C\norm{A}{W^{1,2}(B^4)}\ee
and
\be\label{holo:mainthm:eq4}
\norm{u^{-1}-id}{W^{2,p}(B_{r}^4)}\leq C\norm{ \tilde g^{-1}\left(\vartheta \omega\right) \tilde g}{W^{1,2}(B_{r}^4)}\leq C\norm{A}{W^{1,2}(B^4)},
\ee
for some constant $C>0$.\\

Define $g:= \tilde g u$ in $B_{r}^4$. By construction, the required $\delb$-equation is solved:
\be\label{holo:gaugeeq} \delb g = - A^{0,1} g\qquad \text{ in } B_{r}^4\ee
We show that $g$ satisfies the required estimate: for any $q<2$ there exists $C_q>0$ such that $\norm{g-id}{W^{2,q}(B_{r_0}^4)}\leq C\norm{A}{W^{1,2}(B_{r}^4)}$. Let $q<2$ arbitrary. The triangle inequality applied on the norm $W^{2,q}$ gives:
$$\norm{g-id}{W^{2,q}(B_{r}^4)}\leq \norm{(\tilde g-id)(u-id)}{W^{2,q}(B_{r}^4)}+\norm{\tilde g-id}{W^{2,q}(B_{r}^4)}+\norm{u-id}{W^{2,q}(B_{r}^4)}.$$
Using the results of \cite[Section 4.8.2, Theorem 1]{runst2011sobolev} and the regularity of $\tilde g-id\in \in \bigcap_{q<2} W^{2,q}(B_{r}^4)$ and $u-id\in W^{2,p}(B_{r}^4)$, it follows that $(\tilde g-id)(u-id)\in \bigcap_{q<2} W^{2,q}(B_{r}^4)$ with
$$ \norm{(\tilde g-id)(u-id)}{W^{2,q}(B_{r}^4)}\leq C\norm{\tilde g-id}{W^{2,q_1}(B_{r}^4)}\cdot \norm{u-id}{W^{2,p}(B_{r}^4)},$$
for some $q_1\in (q,2)$ and constant $C>0$. Hence, from \eqref{holo:mainthm:eq1}, \eqref{holo:mainthm:eq2} and \eqref{holo:mainthm:eq3} it immediately follows that there exists a constant $C_q>0$ such that:
$$\norm{g-id}{W^{2,q}(B_{r}^4)}\leq C_q\norm{A}{W^{1,2}(B^4)}.$$
By arguing in a completely analogous way, we obtain the fact that
$$\norm{g^{-1}-id}{W^{2,q}(B_{r}^4)}\leq C_q\norm{A}{W^{1,2}(B^4)}.$$

It remains to show the existence of $h$. We apply $g$ to $A$ in $B_r^4$ to get:
$$A^g = g^{-1} (\delb g+\p g) + g^{-1}A^{0,1}g - g^{-1}\ov{A^{0,1}}^Tg = g^{-1}\p g - g^{-1}\ov{A^{0,1}}^Tg.$$
Since (\ref{holo:gaugeeq}) holds, then $\p \ov{g}^T = -\ov{g}^T\ov{A^{0,1}}^T.$ Hence, $\left(\ov{g}^T\right)^{-1}\p \ov{g}^T = -\ov{A^{0,1}}^T$. By plugging this into the equation above, we get
$$A^g = g^{-1}\p g + g^{-1}\left(\ov{g}^T\right)^{-1}\p \ov{g}^T g = (\ov{g}^Tg)^{-1}\p (\ov{g}^Tg).$$
We conclude the proof by defining $h:=\ov{g}^Tg$, and $h \in W^{2,q} (B_r^4, i\mathfrak{u}(n)).$ for any $q<2$.
\end{proof}
\vspace{4mm}

\begin{Rm}\label{holo:rm}\ \\
\vspace{-11pt}
\begin{enumerate}
\item[(i)] Firstly, from the proof of the theorem above that the radius $r>0$ can be chosen to be the same under small perturbations of the $1$-form $A$. 
\item[(ii)] Secondly, all the above estimates on $g$ hold also for $g^{-1}$. They can be similarly computed using the ones of $g$ as in the proof of Lemma \ref{stronger}.
\end{enumerate}
\end{Rm}

Having made the above remarks, we end the section by proving a stability result for holomorphic trivialisations. Later on, this Corollary will be used to show the convergence of holomorphic structures. 

\begin{Co}\label{holo:stability}
Let $r<1$, $A_1\in W^{1,2}(\Omega^1 B^4\otimes \mathfrak{u}(n))$ and $g_1\in W^{2,q}(B_r^4, GL_n(\C))$ satisfying Theorem \ref{holo:mainthm}. There exists $\delta>0$ such that for all $A_2\in W^{1,2}(\Omega^1 B^4\otimes \mathfrak{u}(n))$ with $F_{A_2}^{0,2}=0$ satisfying $$\norm{A_1- A_2}{W^{1,2}(B^4)}\leq \delta,$$there exists a radius $r_0\in (r/2,r)$ depending only on $A_1$ and a gauge  
$g_2\in \bigcap\limits_{q<2} W^{2,q}(B_{r_0}^4, GL_n(\C))$ that trivialises $A_2$ in the sense:
$$A_2=-\delb g_2\cdot g_2^{-1}\qquad \text{ in }B_{r_0}^4$$
with the following estimates:  for any $q<2$ there exists $C_q>0$ such that $$\norm{g_2-id}{W^{2,q}(B_{r_0}^4)}\leq C_q\left(\norm{A_1}{W^{1,2}(B^4)}+\norm{A_2}{W^{1,2}(B^4)}\right)$$ and there exists $C>0$ such that $$\norm{g_1-g_2}{L^{p}(B_{r_0}^4)}\leq C\norm{A_1- A_2}{W^{1,2}(B^4)}$$ for any $p<12$. 
\end{Co}

\begin{proof}[Proof of Corollary \ref{holo:stability}]
Choose $\delta>0$ such that $A_2$ is a small perturbation of $A_1$. By Remark \ref{holo:rm}(i)  and Theorem \ref{holo:mainthm} applied to the forms $A_1$ and $A_2$ we obtain the existence of $r>0$ and gauges $g_1,g_2\in W^{2,q}(B_r^4, GL_n(\C))$ for any $q<2$ so that
$$\delb g_1 = -A_1^{0,1}\cdot g_1\text{ and } \delb g_2=-A_2^{0,1}\cdot g_2\text{ in }B_r^4$$
and there exists a constant $C_q>0$ such that 
\be\label{holo:basicest}
\begin{array}{l}
\norm{g_1-id}{W^{2,q}(B_{r}^4)}\leq C_q\norm{A_1}{W^{1,2}(B^4) },\\\norm{g_2-id}{W^{2,q}(B_{r}^4)}\leq C_q\norm{A_2}{W^{1,2}(B^4)}\leq C_q\left(\norm{A_1}{W^{1,2}(B^4) } +  \norm{A_1-A_2}{W^{1,2}(B^4) }\right)\quad\text{ and }\\
\norm{g_2^{-1}-id}{W^{2,q}(B_{r}^4)}\leq C_q\norm{A_2}{W^{1,2}(B^4)}\leq C_q\left(\norm{A_1}{W^{1,2}(B^4) } +  \norm{A_1-A_2}{W^{1,2}(B^4) }\right)
\end{array}\ee
Since $g_1$ and $g_2$ holomorphically trivialise $A_1$ and $A_2$ respectively, we can relate the \textit{transition gauge} $g_2^{-1}g_1$ with the difference 1-form $A_2-A_1$ through the following $\delb$-equation:
\be\label{holo:transitioneq}\delb (g_2^{-1}g_1) = g_2^{-1}(A_2-A_1)^{0,1}g_2\cdot (g_2^{-1}g_1).\ee
We first estimate $g_2^{-1}g_1-id$ using the inequalities (\ref{holo:basicest}) and then use the equation to show that $g_2^{-1}g_1-id$ is only bounded by the norm of $A_2-A_1$. Fix $q<2$. The triangle inequality gives:
$$
\norm{g_2^{-1}g_1-id}{W^{2,q}(B_{r}^4)}\leq  \norm{(g_2^{-1}-id)(g_1-id)}{W^{2,q}(B_{r}^4)} + \norm{g_2^{-1}-id}{W^{2,q}(B_{r}^4)} + \norm{g_1-id}{W^{2,q}(B_{r}^4)}$$
Hence, by the results of  \cite[Section 4.8.2, Theorem 1]{runst2011sobolev} applied to the product $(g_2^{-1}-id)(g_1-id)$ and estimates (\ref{holo:basicest}), there exists a constant $C_q>0$ so that
\be\label{holo:transitioneq2}\norm{g_2^{-1}g_1-id}{W^{2,q}(B_{r}^4)}\leq C_q(\norm{A_1}{W^{1,2}} + \norm{A_2}{W^{1,2}})\leq 2C_q(\norm{A_1}{W^{1,2}} + \norm{A_1 - A_2}{W^{1,2}}) .\ee
for any $q<2$.
We can use equation (\ref{holo:transitioneq}) in order to find an a-posteriori estimate of $g_2^{-1}g_1-id$ involving \textit{only} the 1-form $A_2-A_1$. 
Let $s<4$. By the regularity of $\delb$ in $L^s$ (see \cite{krantz1976optimal}) there exists a holomorphic function $h$ and a constant $C_s>0$ such that
$$\norm{g_2^{-1}g_1 - h}{L^{6s/(6-s)}(B_r^4)}\leq C_s\norm{\delb (g_2^{-1}g_1)}{L^s(B_r^4)}\leq C_s\norm{g_2^{-1}(A_2-A_1)^{0,1}g_2}{L^{sp/(p-s)}(B_r^4)}\norm{g_2^{-1}g_1}{L^p(B_r^4)},$$
where $s<p<\infty$ arbitrary. Hence, it follows that there exists $C>0$ depending on $A_1$ such that
$$\norm{g_2^{-1}g_1 - h}{L^{6s/(6-s)}(B_r^4)}\leq C \norm{A_1-A_2}{W^{1,2}(B^4)}\norm{g_2^{-1}g_1 - id }{L^p(B_r^4)} + C \norm{A_1-A_2}{W^{1,2}(B^4)}.$$
There exists $q<2$ such that $W^{2,q}\hookrightarrow L^p$. Since $g_2^{-1}g_1 - id$ is bounded in $W^{2,q}$ \eqref{holo:transitioneq2}, then it is also bounded in $L^p$. Hence,
$$\norm{g_2^{-1}g_1 - h}{L^{6s/(6-s)}(B_r^4)}\leq  C \norm{A_1-A_2}{W^{1,2}(B^4)}.$$
Since $s<4$, there exists a constant $C>0$ independent of $p$ such that
$$\norm{g_2^{-1}g_1 - h}{L^{p}(B_r^4)}\leq  C \norm{A_1-A_2}{W^{1,2}(B^4)},$$
for any $p<12$. Having this inequality at our disposal, we can turn to estimate $g_1-g_2\cdot h$. Let $p<12$, then:
$$\norm{g_1-g_2\cdot h}{L^p(B_r^4)} = \norm{(g_2-id)(h-g_2^{-1}g_1) + h-g_1^{-1}g_2}{L^p(B_r^4)}.$$
For $v\in (p,12)$, we get:
$$\norm{g_1-g_2\cdot h}{L^p(B_r^4)} \leq \norm{g_2-id}{L^{vp/(v-p)}(B_r^4)}\norm{g_2^{-1}g_1-h}{L^v(B_r^4)}+ \norm{g_2^{-1}g_1-h}{L^p(B_r^4)}.$$
Thus, there exists a constant $C_{vp}>0$ depending on $v,p$ and $A_1$ such that
$$\norm{g_1-g_2\cdot h}{L^p(B_r^4)} \leq C_{vp}\norm{A_1-A_2}{W^{1,2}(B^4)}.$$
Moreover, $g_2\cdot h$ solves the equation:
$$\delb (g_2\cdot h) = A_2^{0,1} (g_2\cdot h).$$
in a distributional sense. It remains to show that the $g_2\cdot h$ is bounded in $W^{2,q}$ by the norms of $A_1$ and $A_2$ in a possible slightly smaller ball. Let $r_0\in (r/2,r)$, then there exists a constant $C>0$ such that
$$\norm{g_2\cdot h-id}{W^{1,2}(B_{r_0}^4)}\leq C\left( \norm{\delb g_2\cdot h}{L^2(B_r^4)} + \norm{g_2\cdot h-id}{L^2(B_r^4)}\right).$$
Consequently, by using the $\delb$-equation satisfied by $g_2\cdot h$, it follows that:
$$\norm{g_2\cdot h-id}{W^{1,2}(B_{r_0}^4)}\leq C\left( \norm{A_2}{L^4(B_r^4)}\norm{g_2\cdot h}{L^4(B_r^4)} + \norm{g_2\cdot h-g_1}{L^2(B_r^4)} + \norm{id-g_1}{L^2(B_r^4)} \right).$$
Having shown that $g_2\cdot h\in L^p$ for all $p<12$, we obtain $$\norm{g_2\cdot h-id}{W^{1,2}(B_{r_0}^4)}\leq C\left(\norm{A_1}{W^{1,2}(B^4)}+\norm{A_2}{W^{1,2}(B^4)}\right).$$
Hence, given that $g_2\cdot h\in W^{1,2}$ and $g_2\cdot h-id$ is bounded by $A_1$ and $A_2$, we get from Lemma \ref{bootstrapcriticalgauge} and Remark \ref{bootstrapsubcritical}(ii) the estimate: for any $q<2$ there exists a constant $C_q>0$ such that:
$$\norm{g_2\cdot h-id}{W^{2,q}(B_{r_0}^4)}\leq C_q\left(\norm{A_1}{W^{1,2}(B^4)}+\norm{A_2}{W^{1,2}(B^4)}\right).$$
By redefining $g_2$ as $g_2\cdot h$, we have proven our stability result.
\end{proof}

\section{Proof of Theorem \ref{thm1}}\label{sec:thm1}

We pick geodesic balls $B_r^4(x_i)$ covering $X^2$ on which the connection can be trivialised: $\nabla = d+A_i$ and
$$\norm{A_i}{W^{1,2}(B_r^4(x_i))}\leq \e_0(X^2,\omega),$$
where $\e_0(X^2,\om)$ is given by Theorem \ref{holo:mainthm}.
Because $X^2$ is a compact manifold, there are finitely many such balls covering $X^2$. 
By Theorem \ref{holo:mainthm} there exists $r'\in (0,r)$,  $g_i\in W^{2,p}(B_{r'}^4(x_i), GL_n(\C))$ and $h_i = g_i^Tg_i\in W^{2,p}(B_{r'}^4, iU(n))$ for any $p<2$ so that $A_i^{g_i} = h_{i}^{-1}\p h_{i}$. Hence
$$\nabla^{g_i} =d  + h_{i}^{-1}\p h_{i}\qquad\text{ in }B_r'^4(x_i).$$ 
We conclude that there exists $W^{2,p}$ global sections $h$ and $g$ such that
$$\nabla^g = \p_0 + h^{-1}\p_0 h+ \delb_{\mathcal E}.$$

Thus, there exists a smooth holomorphic structure $\mathcal E$ on our hermitian bundle $(E,h_0)$.
\section{Density under high energy}\label{sec:highen}

Until now we have looked at density results under the assumption of low Yang-Mills energy. As before we start by assuming that we work on the flat unit ball $B^4$ and in section \ref{highenergy:globalmainsec} we show how to generalise our results on the compact manifold $X^2$.  We start by investigating the case when $A\in W^{1,2}(\Omega^1B^4\otimes \mathfrak{u}(n))$ and $\norm{A}{W^{1,2}(B^4)}<\infty$. Furthermore we are assuming the integrability condition $F_{A}^{0,2}=0$ is satisfied.\\

\textbf{Difficulty}:\\

If we want to proceed as in the case of low Yang-Mills energy, we start by smoothing $A$ inside $B^4$ by simple convolution and thus, obtain a sequence of smooth forms $A_k$ converging to $A$ in $W^{1,2}$ as $k\rightarrow\infty$. The integrability condition (\ref{II.1}) is, however, lost for $A_k$. Furthermore, since we want to preserve the condition for each $k$, the argument reduces to finding a sequence of perturbations $\omega_k\in C^\infty(\Omega^{0,2} B^4\otimes M_n(\C))$ uniformly bounded in $W^{2,2}$ that solve 
$$\left\{
\begin{array}{rll}
\delb\delbs\omega_k & = - \left[\delbs\omega_k,A_k^{0,1}\right]-\delbs\omega_k\wedge\delbs\omega_j^k  - F_{A_k}^{0,2}& \text{ in }B^4\\
\omega_k&=0& \text{ on }\p B^4
\end{array}\right.$$
Since $A_k^{0,1}$ is not small in $W^{1,2}$ norm, we cannot hope to apply a fixed point argument even if $F_{A_k}^{0,2}$ is very small in $L^2$ norm (it converges to $F_A^{0,2}=0$). To make the situation worse, the linear operator $\delb\delbs \cdot + \left[\delbs\cdot ,A_k^{0,1}\right]$ might have non-trivial kernel.\\

Hence, in this section we have developed a method that deals with the case of $A_k$ having high $W^{1,2}$ norm. We present it below:\\

\textbf{Strategy}:\\

We recall that in Section 2, Proposition \ref{low:fredholm} and Lemma \ref{CP2extension}, we were able to find a perturbation $\delbs \omega$ to the form $\tilde A$ in order to obtain the integrability condition \eqref{II.1} over $\CP^2$. This method, however, heavily used that the operator $L_{\tilde A}$ \eqref{low:la} is invertible under the smallness condition of the $W^{1,2}$ norm of $\tilde A$. Following this blueprint, our idea is to find a unitary gauge change $g$ of $A$ such that the operator $\delb\delbs \cdot + \left[\delbs\cdot ,\left(A^{0,1}\right)^g\right]$ is invertible.\\

Firstly, we will need to acquaint ourselves with this idea. We found it natural to start by considering the case of linear perturbations of $A$ and show that we can always find a smooth perturbation $U$ such that the operator $\delb\delbs \cdot + \left[\delbs\cdot ,A^{0,1} +\beta  \delb U\right]$ acting on $(0,2)$ forms has a trivial kernel for some $\beta>0$. \\

Having this idea, we search for a unitary gauge change $g$ that forces the operator $\delb\delbs \cdot + \left[\delbs\cdot ,\left(A^{0,1}\right)^g\right]$ to have trivial kernel. Moreover, we show that for $k$ large enough, the  \underline{same} gauge $g$ gives that the operators $\delb\delbs \cdot + \left[\delbs\cdot ,\left(A_k^{0,1}\right)^g\right]$ are also invertible.
This enables us to find a perturbation $\omega_k$ that solves
$$\delb\delbs\omega_k + \left[\delbs\omega_j^k,\left(A_k^{0,1}\right)^g\right]+\delbs\omega_k\wedge\delbs\omega_k = - F_{A_k^g}^{0,2}$$
with $\omega_k = 0$ on $\p B^4$, and satisfies the estimate:
$$\norm{\omega_k}{W^{2,2}_0(B^4)}\leq C\opnorm{T_{\left(A_k^{0,1}\right)^g}^{-1}}\norm{F_{A_k^g}^{0,2}}{L^2(B^4)},$$
for some constant $C>0$. Using this estimate, together with the convergence of $A_k$ to $A$ in $W^{1,2}$ and $F_{A}^{0,2}=0$, we obtain the strong convergence of the sequence $\omega_k$ to $0$. \\

Hence, we prove the first theorem of this section by also taking into account that $g$ is a unitary gauge transformation, and we can thus, use the invariance of the $L^2$ norm under the action of $g$: $\norm{F_{A_k}^{0,2}}{L^2} = \norm{g^{-1}F_{A_k^g}^{0,2}g}{L^2} = \norm{F_{A_k^g}^{0,2} }{L^2}$. \\

The local theorem on $B^4$ is stated as follows:
\begin{Th}{\label{highenergy:mainthm}}
Let $A\in W^{1,2}(\Omega^1 B^4\otimes \mathfrak{u}(n))$, with $F_A^{0,2} = 0$. There exists a smooth sequence of forms $A_k \in C^\infty(\Omega^1 B^4\otimes\mathfrak{u}(n))$, $F_{A_k}^{0,2}=0$ and $A_k \rightarrow A$ in $W^{1,2}(B^4)$.
\end{Th}

Moreover, using the local theorem, we will show that it implies the global existence of an approximating smooth sequence. In particular, we obtain:
\begin{Th}{\label{highenergy:globalmainthm}}
Let $\nabla$ a $W^{1,2}$ unitary connection over $X^2$, satisfying the integrability condition $$F_{\nabla}^{0,2} = 0.$$ Then there exists a sequence of smooth unitary connections $\nabla_k$, with $F_{\nabla_k}^{0,2}=0$ such that
$$d_2(\nabla_k,\nabla)\rightarrow 0.$$
\end{Th}

\subsection{Linear perturbation}

We will be looking at finding a small linear perturbation that forces the operator $$L_{A^{0,1}} = \delb\delbs \cdot + [A^{0,1},\delbs \cdot]$$ to have trivial kernel, assuming a $0$ boundary condition. Let $U\in C^\infty(B^4,M_n(\C))$. We define the following operators
$$\begin{array}{c}
L_0:W_D^{2,2}(\Omega^{0,2} B^4\otimes M_n(\C))\rightarrow L^2(\Omega^{0,2}B^4\otimes M_n(\C))\\
\omega\mapsto\delb\delbs \omega + [A^{0,1},\delbs \omega]
\end{array}$$
$$
\begin{array}{c}
L_{\beta,U}:W_D^{2,2}(\Omega^{0,2} B^4\otimes M_n(\C))\rightarrow L^2(\Omega^{0,2}B^4\otimes M_n(\C))\\
\omega\mapsto  L_0\omega + \beta B_U\omega
\end{array}
$$
where $B_U = [\delb U,\delbs \cdot].$

\begin{Prop}{\label{highenergy:fred}}
$L_0$ and $L_{\beta,U}$ are Fredholm operators of index zero from the space $W_D^{2,2}(\Omega^{0,2} B^4\otimes M_n(\C))$ to $L^2(B^4,M_n(\C)).$
\end{Prop}

\begin{proof}[Proof of Proposition \ref{highenergy:fred}]
It is sufficient to prove this statement for $L_0$. $L_0$ is Fredholm since it is elliptic. Moreover, $\delb\delbs=\frac{1}{2}\Delta d\ov{z_1}\wedge d\ov{z_2}$ is an elliptic operator of Fredholm index zero from $W_D^{2,2}(\Omega^{0,2} B^4\otimes M_n(\C))$ into $L^2(\Omega^{0,2} B^4\otimes M_n(\C))$.\\

Let $A_k$ be a smooth sequence of $1$-forms converging strongly in $W^{1,2}$ to $A$. Then the bracket operator
$$\omega\mapsto [A_k^{0,1},\delbs\omega]$$
is compact $W^{2,2}(\Omega^{0,2}(B^4))$ to $L^2(\Omega^{0,2}(B^4))$. Indeed, $A_k$ is bounded in $L^\infty$ and hence:
$$\norm{[A_k^{0,1},\delbs\omega]}{L^2(B^4)}\leq C\norm{A_k}{L^\infty(B^4)}\norm{\delbs \omega}{W^{1,2}(B^4)},$$
for some constant $C>0$, where we have used that $W^{1,2}$ is compactly embedded in $L^2$ in $4$-dimensions, by Rellich-Kondrachov \cite{adams2003sobolev}. By the compact embeddedness, it follows that 
the operators $\omega\mapsto [A_k^{0,1},\delbs\omega ]$ are compact $W^{2,2}$ to $L^2$ for all $k$. Hence, using the compactness of these operators and the fact that $\delb\delbs$ is Fredholm, then by \cite[Theorem 4.4.2, p.185]{salamon2018functional}) we have $$index(\delb\delbs\cdot + [A_k^{0,1},\delbs\cdot ])=index(\delb\delbs) = 0.$$
Moreover, for fix $\e>0$ given by \cite[Theorem 4.4.2, p.185]{salamon2018functional}, then there exists $k_0>0$ such that for all  $k\geq k_0$, we have that
$\opnorm{[A_k^{0,1}-A^{0,1},\delbs]}\leq \e$
since $A_k$ converges strongly to $A$ in $W^{1,2}$.
Thus, by applying \cite[Theorem 4.4.2, p.185]{salamon2018functional} to the perturbation operator $[A_k^{0,1}-A^{0,1},\delbs]$ and to $L_0$, we obtain that
$$index(L_0) = index(L_0+[A_k^{0,1}-A^{0,1},\delbs]) =index(\delb\delbs\cdot + [A_k^{0,1},\delbs\cdot ])=0$$
This proves the statement.
\end{proof}

We will be working with operators of the form $L_0$ and $L_{\beta,U}$, where $\beta\in\mathbb{R}$ is small and $U\in W^{2,2}(B^4,M_n(\C))$. By the Proposition above the kernel of $L_0$ is finite dimensional. 
On the space $(Ker L_0)^\perp$, there exists a compact operator $S = L_0^{-1}$ from $L^2$ into $W_D^{2,2}$ such that
$$S:Ran L_0 \rightarrow (Ker L_0)^\perp.$$ By classical spectral theory, $S$ has discrete spectrum with a possible accumulation point at 0 (see for example \cite[Theorem VI.15]{reed1980methods}). This means that the spectrum of $S$ is
$\{\lambda_1,\lambda_2,\ldots, \lambda_n,\ldots\}$
where $\lambda_n\rightarrow 0$.\\

Thus, on $(Ker L_0)^\perp$, the spectrum of $L_0$ is 
$$\left\{\frac{1}{\lambda_1},\frac{1}{\lambda_2},\ldots,\frac{1}{\lambda_n},\ldots\right\}$$
where  $\lambda_n\rightarrow 0$.
If we assume $Ker L_0 \neq \{0\}$, we have that $0$ has to be in the spectrum as well and then 
$$\left\{0,\frac{1}{\lambda_1},\frac{1}{\lambda_2},\ldots,\frac{1}{\lambda_n},\ldots\right\}$$
is the spectrum of the operator $ L_0$. In addition, $L_{\beta,U}$ has discrete spectrum by arguing as before with $A^{0,1}+\beta \delb U$ instead of $A^{0,1}$. 
\\

In the next Proposition we show that the number of eigenvalues near $0$ of $L_{\beta,U}$ cannot exceed the multiplicity of the $0$ eigenvalue of $L_0$ for $\beta$ small enough. Let $m$ be the multiplicity of $0$ for the operator $L_0$.

\begin{Prop}{\label{highenergy:no_ev}}
There exists $\beta_0>0$ so that for each $0\leq\beta\leq\beta_0$, $L_{\beta,U}$ has at most $m$ finitely many not necessarily distinct eigenvalues $\lambda_U^{(1)}(\beta),\ldots \lambda_U^{(m)}(\beta)$ near $0$ that converge to $0$ as $\beta\rightarrow 0$.
\end{Prop}

\begin{proof}[Proof of Proposition \ref{highenergy:no_ev}]
Denote $\varepsilon_0 := \left|\frac{1}{\lambda_1}\right|.$ and let $\varepsilon<\varepsilon_0$. Since the ball $B_\varepsilon(0) = \{|\lambda|\leq\varepsilon\}\subseteq \C$ is compact, then $L_{\beta,U}$ has finitely many not necessarily distinct $m_{\beta}$ eigenvalues in $B_\varepsilon(0)$, $\lambda_{U}^{(1)}(\beta),\ldots,\lambda_U^{(m_{\beta})}(\beta)$. Moreover, by \cite{butler1959perturbation} there are at most $m$ eigenvalues, where $m$ is the multiplicity of the $0$ eigenvalue of $L_0$. We need to show that $\lambda_U^{(i)}(\beta)\rightarrow 0$ as $\beta\rightarrow 0$ for each $1\leq i\leq m_\beta$.\\

We argue by contradiction. If this wouldn't be the case, then there exists $\lambda_0\neq 0$ and $i_0$ such that $\lambda_U^{i_0}(\beta)\rightarrow \lambda_0$ as $\beta\rightarrow 0$, with $\lambda_0\in\sigma(L_0)$. Since $L_{\beta,U}$ is an analytic family of operators in $\beta$, we know that the set $\Gamma = \{(\beta,\lambda)| \lambda\in\rho(L_{\beta,U})\}$  is open (see \cite[Theorem XII.7]{reed1978iv}). \\

Because $\Gamma$ is open and $B_\varepsilon(0)$ contains only one point of the spectrum of $L_0$, there exists $\delta>0$ and an open tubular neighbourhood $$\mathcal{T} = (-\delta,\delta)\times (B_{\e+\delta}(0)\setminus B_{\e-\delta}(0))$$ of $\{0\} \times \p B_\varepsilon(0)$ that does not contain any elements of the spectrum of $L_{\beta,U}$. By definition of $\Gamma$, we have that $\mathcal{T}\subset \Gamma$. In particular, $\delta$ can be chosen small enough such that for each $\beta$ satisfying $|\beta|<\delta$, there is no $\lambda_{\varepsilon,\beta,U} \in \p B_\varepsilon(0)\cap \sigma(L_{\beta,U})$. Thus, we have a contradiction since $|\lambda_0|\geq\varepsilon_0$. We conclude that for all $i\leq m_\beta$, $\lambda_U^{(i)}(\beta)\rightarrow 0$ as $\beta\rightarrow 0$.
\end{proof}
\vspace{4mm}
Define the following operator:
$$P_{\beta,U}:= -\frac{1}{2\pi i} \oint_{|\lambda|=\varepsilon} ( L_{\beta,U}  - \lambda)^{-1} \,d\lambda$$
where $\varepsilon$ is chosen as in Proposition \ref{highenergy:no_ev}. Since we have isolated branched points of the spectrum, we can rewrite this as
$$P_{\beta,U}:= -\frac{1}{2\pi i} \sum_{i=1}^m \oint_{|\lambda - \lambda_U^{(i)}(\beta)|=\varepsilon_1} ( L_{\beta,U}  - \lambda)^{-1} \,d\lambda.$$
for some $\varepsilon_1$ small enough. Each term of this sum is the projection onto the generalised eigenspace of $L_{\beta,U}$ corresponding to the eigenvalue $\lambda_U^{(i)}(\beta)$ (see\cite[Chapter XII]{reed1978iv}). By Proposition \ref{highenergy:no_ev} on the circle $|\lambda|=\varepsilon$, we have that $\lambda \in \rho(L_{\beta,U})$. Thus, the resolvent $R_\lambda =  ( L_{\beta,U}  - \lambda)^{-1}$ is analytic on the $\varepsilon$ circle. It follows that $P_{\beta,U}$ is also analytic in terms of $\beta$ for $\beta$ small enough. Similarly, we can define the operator $P_0$ associated to $L_0$.\\
 
We denote the generalised eigenspace (as defined in \cite[Chapter 5]{salamon2018functional}) corresponding to $0$ for $L_0$  by $$G_0 = \bigcup\limits_{k=1}^\infty \left\{v\in W_D^{2,2}(\Omega^{0,2} B^4\otimes M_n(\C)) | L_0^k v = 0\right\}$$
and the range of $P_{\beta, U}$ by 
$$G_{\beta,U} = \bigcup\limits_{k=1}^\infty\bigoplus\limits_{i=1}^m \left\{v\in W_D^{2,2}(\Omega^{0,2} B^4\otimes M_n(\C)) | (L_{\beta,U} - \lambda_U^{(i)}(\beta))^k v = 0\right\}$$
respectively. The following proposition will show that these two spaces are isomorphic for small values of $\beta$.
\begin{Prop}\label{highenergy:p_iso}
There exists $\beta_0>0$ so that $P_{\beta,U}:G_0\rightarrow G_{\beta,U}$ is an isomorphism for all $0<\beta<\beta_0$.
\end{Prop}

\begin{proof}[Proof of Proposition \ref{highenergy:p_iso}]
From \cite[Theorem XII.5]{reed1978iv}, $P_{\beta,U}$ and $P_0$ define two surjective projection operators:
$$P_{\beta,U}:W_D^{2,2}(\Omega^{0,2} B^4\otimes M_n(\C))\rightarrow G_{\beta,U}$$
and
$$P_0:W_D^{2,2}(\Omega^{0,2} B^4\otimes M_n(\C))\rightarrow G_0.$$
\textbf{Claim 1}. There exists $\beta_0>0$ so that for all $\beta<\beta_0$, $P_{\beta,U}$ is surjective as an operator from $G_0$ to $G_{\beta,U}$.
$$\begin{array}{rll}
P_{\beta,U}-P_0 &=&  -\frac{1}{2\pi i} \oint_{|\lambda|=\epsilon} ( L_{\beta,U}  - \lambda)^{-1} -  ( L_0  - \lambda)^{-1}  \,d\lambda\\
&=&-\frac{1}{2\pi i} \oint_{|\lambda|=\epsilon}  \left(( L_0  - \lambda)^{-1} - \sum\limits_{i=1}^\infty \beta^i (L_0 - \lambda)^{-1} (B_U (L_0-\lambda)^{-1})^i\right) \\
 & & -  ( L_0  - \lambda)^{-1}\,d\lambda\\
& = &O(\beta)
\end{array}$$
Thus, there exists $\beta_0$ such that for all $\beta<\beta_0$ we have the following bound for the norm of the operator $P(\beta)-P(0)$:
$$\norm{P_{\beta,U}-P_0}{W_D^{2,2}} < \frac{1}{16}.$$
Since $L_{\beta,U} - \lambda$ is a continuous operator from $W_D^{2,2}$ to $L^2$ for any $\lambda,\beta\in \mathbb{R}$, we have that 
$$(L_{\beta,U} - \lambda)^{-1}G_0$$
is closed, since $Ran P_0 = G_0$ is closed (see \cite[Theorem XII.5]{reed1978iv}).
This implies that $P_{\beta,U}G_0$ is closed as $P_{\beta,U}$ is a composition of two continuous operators.\\

Assume that $P_{\beta,U}G_0\neq G_{\beta,U}$. Since $P_{\beta,U}G_0$ is closed in $G_{\beta,U}$, we can apply Riesz Lemma (see for example \cite[Lemma 1.2.13]{salamon2018functional}). Then there exists $u\in G_{\beta,U}$ where $\norm{u}{W_D^{2,2}} = 1$ and
$$\inf\limits_{v\in P_{\beta,U}G_0} \norm{u - v}{W_D^{2,2}}>\tfrac{1}{2}.$$

$P_{\beta,U}$ is a projection operator and $u\in G_{\beta,U}$, then $u$ is its own projection - $u=P_{\beta,U}u$. Moreover, the norm distance between $u$ and $P_0u$ satisfies the following inequality
$$\begin{array}{rl}
\norm{u-P_0u}{W_D^{2,2}} &= \norm{P_{\beta,U}u- P_0u}{W_D^{2,2}} \leq \norm{P_{\beta,U}-P_0}{W_D^{2,2}}\norm{u}{W_D^{2,2}}\\
& = \norm{P_{\beta,U}-P_0}{W_D^{2,2}} <\frac{1}{16}.
\end{array}$$
The last inequality holds, because $\beta$ is chosen to be small. From the above estimate, we can easily estimate the norm of $P_0u$:
$$\norm{P_0u}{W_D^{2,2}} \leq \norm{u}{W_D^{2,2}}+\norm{u-P_0u}{W_D^{2,2}}<1+\tfrac{1}{16}.$$
By further computing, we get
$$\begin{array}{rcl}\norm{P_{\beta,U}P_0u - P_0u}{W_D^{2,2}} &=& \norm{P_{\beta,U}P_0u - P_0P_0u}{W_D^{2,2}} \\
&\stackrel{P_0^2 = P_0}{\leq}& \norm{P_{\beta,U}-P_0}{W_D^{2,2}} \norm{P_0u}{W_D^{2,2}}\\
&<&\tfrac{1}{16^2} +\tfrac{1}{16}.
\end{array}$$
Thus,
$$\norm{u - P_{\beta,U}P_0u}{W_D^{2,2}} \leq \norm{P_{\beta,U}P_0u - P_0u}{W_D^{2,2}}+\norm{u-P_0u}{W_D^{2,2}}< \tfrac{2}{16} +\tfrac{1}{16^2}. $$
Since $P_{\beta,U}P_0u\in P_{\beta,U}G_0$, we get a contradiction with $\inf\limits_{v\in P_{\beta,U}G_0} \norm{u - v}{W^{2,2}}>\frac{1}{2}.$
Hence, $P_{\beta,U}G_0 = G_{\beta,U}$ and we have proven the claim that $P_{\beta,U}$ is surjective from $G_0$ to $G_{\beta,U}$.\\

\textbf{Claim 2}. $P_{\beta,U}$ is injective as an operator from $G_0$ to $G_{\beta,U}$. \\

Let $\omega_0,\omega_1\in G_0$ such that $P_{\beta,U}\omega_1 = P_{\beta,U}\omega_0$ . Then
$$\begin{array}{rll}
0 = P_{\beta,U}(\omega_1-\omega_0) &=&  -\frac{1}{2\pi i} \oint_{|\lambda|=\epsilon} ( L_{\beta,U}  - \lambda)^{-1} (\omega_1-\omega_0)  \,d\lambda\\
&=&-\frac{1}{2\pi i} \oint_{|\lambda|=\epsilon}  \left(( L_0  - \lambda)^{-1} - \sum\limits_{i=1}^\infty \beta^i (L_0 - \lambda)^{-1} (B_U (L_0-\lambda)^{-1})^i\right) \\
 & & (\omega_1-\omega_0)\,d\lambda\\
& = &\omega_1-\omega_0 + O(\beta).
\end{array}$$
Because this holds for any $\beta<\beta_0$ we conclude that $\omega_0=\omega_1$.\\

Hence, from the two claims above, there exists $\beta_0>0$ so that
$$P_{\beta,U}:G_0\rightarrow G_{\beta,U}$$ is an isomorphism for all $\beta<\beta_0$.
\end{proof}
\vspace{4mm}

The last ingredient we need to prove, in order to obtain the existence of a perturbation that makes $Ker L_{\beta,U}$ trivial, is the next statement. It gives us the perturbation $U$ which will satisfy the necessary condition to make the kernel trivial. This next Proposition together with the existence of the isomorphism $P_{\beta,U}$ will be key to proving the result.

\begin{Prop}\label{highenergy:k_null_u}
There exists a smooth Hermitian function $U\in C^\infty(B^4,\mathfrak{u}(n))$ such that $B_U$ is injective on $Ker L_0$.
\end{Prop}

\begin{proof}[Proof of Proposition \ref{highenergy:k_null_u}]
Since $Ker L_0$ is finite dimensional, let $\{e_1,\ldots, e_N\}$ be an orthonormal basis of it. \\

Let $v\in Ker L_0$, $v\neq 0$. We show that for each such $v$, we can find $U_v$ such that $B_{U_v}v\neq 0$. Assume by contradiction that $B_Uv= 0$ for all smooth functions $U$ on $B^4$. Define the linear operator $H(\omega) := \left([\omega, \delbs v]\right)^{0,2}:C^\infty(\Omega^1 B^4\otimes M_n(\C))\rightarrow C^\infty(\Omega^2 B^4\otimes M_n(\C))$ which satisfies the fact that $H(\omega(z))=H(\omega)(z)$, where $\omega(z)$ means that each component of $\omega$ is applied to $z$. Moreover, we have that $$0= B_Uv= [\delb U,\delbs v] = H(dU)$$ for all smooth functions $U$ on $B^4$. Applying Proposition \ref{app:vanishing1} to $H$, we obtain that $H=0$. By density of smooth $(0,1)$-forms into $W^{1,2}$ $(0,1)$-forms, it follows in particular that
$$H(A^{0,1}) = [A^{0,1},\delbs v] = 0.$$
Putting this together with the fact that $v\in Ker L_0$, we obtain:
$$0= L_0 v = \delb\delbs v.$$
Since $v = 0$ on $\p B^4$ and $\delb\delbs v = 0$, then $v = 0$ in $B^4$. This is a contradiction because $\norm{v}{W_D^{2,2}} = 1$. Hence, there exists $U_v\in  C^\infty(B^4,M_n(\C))$ so that $B_{U_v} v \neq 0$.\\

Next, we show that such a $U_v$ can be chosen to be Hermitian. Indeed since $U_v\in M_n(\C)$, there exists a decomposition in terms of its Hermitian and anti-Hermitian part: $$U_v = U_1+U_2,$$ where $U_1\in C^\infty(B^4,\mathfrak{u}(n))$ and $U_2\in C^\infty(B^4,i\mathfrak{u}(n))$. Assume that $B_{U_1}v= 0$, otherwise we redefine $U_v := U_1$. Under this assumption, by linearity it then necessarily follows that $$B_{U_2}v\neq 0.$$ If this condition holds, then by multiplying with $i$, $$iB_{U_2}v = B_{iU_2}v\neq 0.$$ Moreover, $iU_2\in C^\infty(B^4,\mathfrak{u}(n))$ and in this case we redefine $U_v:=iU_2$. Hence, there exists 
\be\label{highenergy:lin_uv}U_v\in  C^\infty(B^4,\mathfrak{u}(n)) \text{ so that } B_{U_v}v\neq 0, \text{ for any }v\in Ker L_0, v\neq 0.\ee

\textbf{Claim.} There exists $U$ smooth Hermitian function such that  $B_U$ is injective on $Ker L_0$.\\

We formulate the following inductive hypothesis: 
$$
\mathcal I(k) =\left\{\begin{array}{ll} &\text{there exists } U^k\in C^\infty(B^4,\mathfrak u(n))\text{ supported in } V^k\subsetneq B^4 \text{ such that }\\
& \{B_{U^k}e_j\}_{j=1}^k \text{ is linearly independent},
\end{array}\right.$$ where $k\leq N$
We show by induction that $\mathcal I(N)$ holds from which it follows that $B_{U^N}$ is injective on $Ker L_0$.\\

By \eqref{highenergy:lin_uv}, there exists $U_1$ such that $B_{U_1}e_1\neq 0$. Without loss of generality, by multiplying with a compactly supported function $\rho_1$, we can localise $U_1$ in $V_1\subsetneq B^4$. Hence $\mathcal I(1)$ holds. Assume that for $k<N$, $\mathcal I(k)$ holds. We prove that $\mathcal I(k+1)$ holds as well. \\

If $\{B_{U^k}e_j\}_{j=1}^{k+1}$ is linearly independent, then set $U^{k+1}=U^k$. Otherwise there exists $\lambda_1,\ldots,\lambda_{k+1}$ not all $0$ such that $\sum_{i=1}^{k+1}\lambda_i B_{U^k} e_{i}=0$. Notice that $\lambda_{k+1}\neq 0$.\\

By \eqref{highenergy:lin_uv} there exists $U_{k+1}$ such that $B_{U_{k+1}}\sum_{i=1}^{k+1}\lambda_i e_{i}\neq 0$. We can choose a neighbourhood $V_{k+1}$ and $\tilde V^k\subseteq V^k$ disjoint from $V_{k+1}$ such that $\{B_{U^k}e_j\}_{j=1}^k$ is linearly independent in $\tilde V^k$ and $$B_{\rho_{k+1} U_{k+1}}\sum_{i=1}^{k+1}\lambda_i e_{i}\neq 0\qquad \text{ in } V_{k+1}$$
In particular, we can define functions $\rho_{k+1}$ compactly supported in $V_{k+1}$, $\rho_k$ compactly supported in $\tilde V^k$. Define $U^{k+1} := \rho_k U^k + \rho_{k+1}U_{k+1}$.\\

It remains to show that $\{B_{U^{k+1}}e_j\}_{j=1}^{k+1}$ is linearly independent. Assume there exists $\beta_1, \ldots, \beta_{k+1}$ such that
\be\label{highenergy:lin_us}\sum_{j=1}^{k+1}\beta_j B_{\rho_k U^{k} + \rho_{k+1}U_{k+1}}e_j = \sum_{j=1}^{k+1}\beta_j B_{U^{k+1}}e_j = 0.\ee
In the neighbourhood $\tilde V^k$, we have that
$$\sum_{j=1}^{k+1}\beta_j B_{U^{k}}e_j =0.$$
Then $(\beta_1,\ldots,\beta_{k+1}) = c(\lambda_1,\ldots,\lambda_{k+1})$ for some constant $c$. Hence, in $V_{k+1}$, we have that
$$c\sum_{j=1}^{k+1}\lambda_j B_{U_{k+1}}e_j=0.$$
By the choice of $U_{k+1}$, we obtain that $c=0$. Hence $\beta_1=\ldots=\beta_{k+1} = 0$. To conclude, define $V^{k+1}=\tilde V^k \cup V_{k+1}$. This proves the induction. \\

Hence, we have obtained $U=U^N$ such that $\{B_{U}e_j\}_{j=1}^N$ are linearly independent, where $e_j$ the orthonormal basis of  $Ker L_0$ we have picked initially. It follows that $B_U$ is injective on $Ker L_0$.
\end{proof}

\vspace{4mm}
We are now ready to prove the result of this section. 

\begin{Lm}\label{highenergy:k_null_lin}
There exists a small constant $\beta\in[0,1]$ and $U\in C^\infty(B^4,M_n(\C))$ such that 
$$Ker L_{\beta,U} = \{0\}$$
where 
$$L_{\beta,U}: W_D^{2,2}(\Omega^{0,2} B^4\otimes M_n(\C))\rightarrow L^2(\Omega^{0,2}B^4\otimes M_n(\C)).$$
Hence, $L_{\beta,U}$ is an invertible operator.
\end{Lm}

\begin{proof}[Proof of Lemma \ref{highenergy:k_null_lin}]
The case when $\beta=0$ and $Ker L_0 = \{0\}$ is trivial. We focus on the case when $Ker L_0 \neq \{0\}$. \\

We assume the worst case scenario $dim G_0 =\infty$. By Proposition \ref{highenergy:k_null_u}, there exists $U\in C^\infty(B^4,\mathfrak{u}(n))$ so that $B_U$ is injective on $Ker L_0$. Furthermore it follows from Proposition \ref{highenergy:p_iso} that there exists an isomorphism $P_{\beta,U}$ between $G_0$ and $G_{\beta,U}$ for all $\beta<\beta_0$ for some $\beta_0>0$. We want to show the existence of $\beta$ so that $Ker L_{\beta,U} = \{0\}$.\\

Assume that for all $\beta<\beta_0$ we have that  $Ker L_{\beta,U} \neq \{0\}$. We aim at showing by contradiction that for some $\beta<\beta_0$ we will get that $Ker L_{\beta,U} = \{0\}$. Thus, let the space $$S_{\beta,U} = P_{\beta,U}^{-1}(Ker L_{\beta,U}).$$
Since $Ker L_{\beta, U}$ is finite dimensional and its dimension is bounded by the dimension of $Ker L_0$ by Proposition \ref{highenergy:no_ev}. Because $P_{\beta, U}$ is an isomorphism, then $S_{\beta, U}$ is a finite dimensional space in $G_0$ of dimension at most $dimKer L_0$. Moreover, $Ker L_0\cap S_{\beta, U}$ is also finite dimensional and there exists an orthonormal basis of this space. We can complete it, to obtain an orthonormal basis $\{e_j^{\beta,U}\}_{j=1}^N$ on $S_{\beta, U}$, where $N=dim Ker L_{\beta, U} = dim S_{\beta, U}$. Fix $1\leq j\leq N$ and $\e>0$ small enough such that $\lambda\in \rho(L_0)$ for all $|\lambda|=\e$. We compute the following:
\begin{align}\label{high:lin_longcomp}
0 =\ & L_{\beta,U}P_{\beta,U}e_j^{\beta_U} = -\frac{1}{2\pi i} L_{\beta,U}\oint_{|\lambda|=\epsilon} ( L_{\beta,U}  - \lambda)^{-1}e_j^{\beta,U} \,d\lambda\nonumber\\
 =\ & -\frac{1}{2\pi i} L_{\beta,U}\oint_{|\lambda|=\epsilon} ( L_0 + \beta B_U  - \lambda)^{-1}e_j^{\beta,U} \,d\lambda\nonumber\\
=\ & -\frac{1}{2\pi i} (L_0 + \beta B_U) \oint_{|\lambda|=\epsilon} ( L_0 + \beta B_U  - \lambda)^{-1}e_j^{\beta,U} \,d\lambda\nonumber\\
 =\ &-\frac{1}{2\pi i} (L_0 + \beta B_U) \oint_{|\lambda|=\epsilon} (( L_0  - \lambda)^{-1}  - \sum\limits_{i=1}^\infty \beta^i (L_0 - \lambda)^{-1} (B_U (L_0-\lambda)^{-1})^i)e_j^{\beta,U} \,d\lambda\nonumber\\
=\ &-\frac{1}{2\pi i} L_0  \oint_{|\lambda|=\epsilon} ( L_0  - \lambda)^{-1} e_j^{\beta,U} \,d\lambda\nonumber \\
 & + \beta \frac{1}{2\pi i} L_0  \oint_{|\lambda|=\epsilon}   \sum\limits_{i=1}^\infty \beta^{i-1} (L_0 - \lambda)^{-1} (B_U (L_0-\lambda)^{-1})^ie_j^{\beta,U}  \,d\lambda\nonumber\\
 & - \beta \frac{1}{2\pi i} B_U  \oint_{|\lambda|=\epsilon} ( L_0  - \lambda)^{-1} e_j^{\beta,U} \,d\lambda\nonumber\\
 & + \frac{1}{2\pi i} \beta B_U \oint_{|\lambda|=\epsilon} \sum\limits_{i=1}^\infty \beta^i (L_0 - \lambda)^{-1} (B_U (L_0-\lambda)^{-1})^i e_j^{\beta,U} \,d\lambda\nonumber\\
=\ & L_0e_j^{\beta,U} -\beta \frac{1}{2\pi i} \left(B_Ue_j^{\beta,U} - L_0  \oint_{|\lambda|=\epsilon} ( L_0  - \lambda)^{-1}B_U ( L_0  - \lambda)^{-1}e_j^{\beta,U} \,d\lambda\right) \\
 &+ O (\beta^2)\nonumber
\end{align}

The last equality holds because $e_j^{\beta,U}\in G_0$ and we have $$P_0e_j^{\beta,U} = -\frac{1}{2\pi i} \oint_{|\lambda|=\epsilon} ( L_0  - \lambda)^{-1} e_j^{\beta,U} \,d\lambda = e_j^{\beta,U}.$$
We further discuss two cases:\\

\textbf{Case 1.} $e_j^{\beta,U}\notin Ker L_0$\\

Because $e_j^{\beta,U}$ is an element of the orthonormal basis, then $e_j^{\beta,U}\in (Ker L_0)^\perp$. Moreover, since $L_0$ is Fredholm, we have
$$1=\norm{e_j^{\beta,U}}{W_D^{2,2}} \leq C\norm{L_0e_j^{\beta,U}}{L^2}  \leq O(\beta)\norm{e_j^{\beta,U}}{W_D^{2,2}} = O(\beta),$$
where $C>0$ is a constant independent of $\beta$. Since $\beta<\beta_0$ is small, we get a contradiction. \\

\textbf{Case 2.} $e_j^{\beta,U}\in Ker L_0$ \\

Because the equation \eqref{high:lin_longcomp} vanishes for any $\beta<\beta_0$ and $Ker L_0\neq \{0\}$ it follows that
\begin{equation}\label{highenergy:k_null_lin_eqn}
B_Ue_j^{\beta,U}  =  L_0  \oint_{|\lambda|=\epsilon} ( L_0  - \lambda)^{-1}B_U ( L_0  - \lambda)^{-1} e_j^{\beta,U} \,d\lambda.
\end{equation}
Using the invertibility of the operator $L_0-\lambda$, where $\lambda\in \rho(L_0)$ then $( L_0  - \lambda)^{-1} (L_0-\lambda) = I$. Thus, by expanding we obtain that
$$(L_0-\lambda)^{-1}L_0 - I = (L_0-\lambda)^{-1}\lambda.$$
$(( L_0  - \lambda)^{-1} \frac{L_0}{\lambda }- \frac I\lambda)e_j^{\beta,U} = (L_0-\lambda)^{-1}e_j^{\beta,U}$. Since $e_j^{\beta,U}\in Ker L_0$, then $(L_0-\lambda)^{-1}e_j^{\beta,U} =  -\frac 1\lambda e_j^{\beta,U}$. We obtain the following:
\begin{equation}\label{highenergy:k_null_lin_eqn2}\begin{array}{lll}
\text{(\ref{highenergy:k_null_lin_eqn})} &= & -L_0  \oint_{|\lambda|=\epsilon} ( L_0  - \lambda)^{-1} \frac{1}{\lambda} B_U e_j^{\beta,U} \,d\lambda
\\
 &=& - \oint_{|\lambda|=\epsilon}  \left(( L_0  - \lambda)^{-1}  + \frac{I}{\lambda}\right) B_Ue_j^{\beta,U}\,d\lambda\\
 & =& - 2\pi i B_U e_j^{\beta,U} - \oint_{|\lambda|=\epsilon} ( L_0  - \lambda)^{-1} B_Ue_j^{\beta,U} \,d\lambda
 \end{array}
 \end{equation}
 
Putting the above equalities {\ref{highenergy:k_null_lin_eqn}} and \ref{highenergy:k_null_lin_eqn2} together, we have that
 $$(1+2\pi i) B_Ue_j^{\beta,U} = -\oint_{|\lambda|=\epsilon} ( L_0  - \lambda)^{-1} B_Ue_j^{\beta,U} \,d\lambda = 2\pi i P_0 B_Ue_j^{\beta,U}$$
We apply $P_0$ on both sides of the equation to get $(1+2\pi i) P_0 B_Ue_j^{\beta,U} = 2\pi i P_0^2 B_Ue_j^{\beta,U}$. Moreover, since $P_0$ is a projection, and thus satisfies $P_0^2=P_0$, our computations then give us the following equality
$$(1+2\pi i) P_0 B_Ue_j^{\beta,U} = 2\pi i P_0 B_Ue_j^{\beta,U}.$$
This can be true only if $P_0 B_Ue_j^{\beta,U} = 0$. Together with $$ (1+2\pi i) B_Ue_j^{\beta,U} = 2\pi i P_0 B_Ue_j^{\beta,U},$$ it implies that $B_Ue_j^{\beta,U} = 0$. This is a contradiction by the choice of our initial $U$. \\

We conclude that for some $\beta<\beta_0$ we have $Ker L_{\beta,U} = \{0\}$.   
\end{proof}

\subsection{Gauge perturbation}

After having acquainted ourselves with the linear perturbation in the section before, we are now in a position to generalise the previous results. First consider operators of the form
$$T_{(A^{0,1})^{g(\beta U)}} = \delb\delbs \cdot + [(A^{0,1})^{g(\beta U)}, \delbs \cdot]$$
where  $$T_{(A^{0,1})^{g(\beta U)}}: W_D^{2,2}(\Omega^{0,2} B^4\otimes M_n(\C))\rightarrow L^2(\Omega^{0,2} B^4\otimes M_n(\C))$$ and $g(\beta U) := \exp (\beta U) \in C^\infty(B^4,U(n))$. For the following proofs we will denote $T_{A^{0,1}}$ and $T_{{A^{0,1}}^{g(\beta U)}}$ by $T_0$ and $T_{\beta,U}$ respectively. We can remark the fact that $T_0 = L_0$. \\

Similar to the linear case, we can deduce that $ T_{\beta,U}$ has a discrete spectrum for any $U$ and $\beta$, with $\beta<1$ small (so that $\exp$ is defined). Moreover, the family of operators have discrete spectrum with no accumulation points and we can express them as
$$\begin{array}{lll}
T_{\beta,U} &=&  \delb\delbs \cdot + [(A^{0,1})^{g(\beta U)}, \delbs \cdot] \\
&=& \delb\delbs \cdot + \sum\limits_{n=0}^\infty [\beta^n A_n,\delbs \cdot]\\
& = & \delb\delbs \cdot  + [A^{0,1},\delbs\cdot] + \sum\limits_{n=1}^\infty \beta^n[ A_n,\delbs \cdot]
\end{array}
$$
where $A_1 = A^{0,1}$ and $A_n$ are $(0,1)$-forms. This shows that the resolvent is an analytic function of $\beta$. Thus, the operator is analytic in the sense of Kato (see \cite{reed1978iv}). In a completely analogous way to Proposition \ref{highenergy:no_ev} we have the existence of $m$ not necessarily distinct eigenvalues corresponding to $T_{\beta, U}$, namely $\lambda_U^{(1)}(\beta),\ldots,\lambda_U^{(m)}(\beta)$. There exists $\e>0$ so that $|\lambda_U^{(i)}(\beta)|<\e$ for all $\beta\geq 0$  and $1\leq i\leq m$. In this section we denote $P_{\beta,U}$ by
$$P_{\beta,U} = -\frac{1}{2\pi i} \oint_{|\lambda|=\e} ( T_{\beta,U} -\lambda)^{-1} \,d\lambda$$
and it is an analytic function of $\beta$ on $|\lambda|=\e$. Similarly, we define $P_0$ for the operator $T_0$.\\

We denote the generalised eigenspace corresponding to $0$ for $T_0$ by $$G_0 = \bigcup\limits_{k=1}^\infty \left\{v\in W_D^{2,2}(\Omega^{0,2} B^4\otimes M_n(\C)) | T_0^k v = 0\right\}$$
and the range of $P_{\beta, U}$ by
$$G_{\beta,U} =  \bigcup\limits_{k=1}^\infty\bigoplus\limits_{i=1}^m \left\{v\in W_D^{2,2}(\Omega^{0,2} B^4\otimes M_n(\C)) | (T_{\beta,U} - \lambda_U^{(i)}(\beta))^k v = 0\right\}$$
respectively. \\

Since $g(\beta U) = \exp(\beta U)$, we then have the existence of an operator $B_{\beta, U}$ so that $$\beta B_{\beta,U}\cdot  = [(A^{0,1})^{g(\beta U)} - A^{0,1},\delbs\cdot]$$ and $B_{\beta, U}$ is analytic in $\beta$ (in particular it does not have any poles). Since $U$ is a smooth Hermitian function we can obtain
$$B_{\beta,U}\cdot = [\delb U + [A^{0,1},U],\delbs\cdot]+O(\beta),$$
by expanding $B_{\beta,U}$ in $\beta$. Define $$B_{0,U}\cdot:= [\delb U + [A^{0,1},U],\delbs\cdot]$$ as a map from $W_D^{2,2}(\Omega^{0,2}B^4\otimes M_n(\C))$ to $L^2(\Omega^{0,2} B^4\otimes M_n(\C))$. Thus, by again using the smoothness of $U$, we can expand $T_{\beta,U}$ in $U$ and obtain
$$T_{\beta,U} = T_0 +\beta B_{\beta,U} = T_0 + \beta B_{0,U} + O(\beta^2).$$

In a completely analogous way we also have that $P_{\beta,U}$ is an isomorphism between $G_0$ and $G_{\beta,U}$. Hence, we can assume this and prove the reciprocal version of Lemma \ref{highenergy:k_null_lin}. Firstly, we prove an analogue of Proposition \ref{highenergy:k_null_u}. The proof will follow very similar steps as before.

\begin{Prop}\label{highenergy:k_null_nu}There exists a smooth Hermitian function $U\in C^\infty(B^4,\mathfrak{u}(n))$ such that $B_{0,U}$ is injective on $Ker T_0$.
\end{Prop}

\begin{Rm}
It is important to remark that this proof will give us a function $U$ that belongs to the Lie algebra $\mathfrak{u}(n)$. This, in turn, will yield a perturbation by a gauge that is unitary, since $g(\beta U)=\exp(\beta U)$. It is crucial to find a unitary gauge, because it will preserve our Hermitian vector bundle structure later on.
\end{Rm}

\begin{proof}[Proof of Proposition \ref{highenergy:k_null_nu}]
Since $Ker T_0$ is finite dimensional, let $\{e_1,\ldots, e_N\}$ be an orthonormal basis of it.\\

Let $v\in Ker T_0$, $v\neq 0$. We show that for each such $v$, we can find $U_v$ such that $B_{0,U_v}v\neq 0$. Assume by contradiction that $B_{0,U}v= 0$ for all smooth functions $U$ on $B^4$. Define the linear operators $$H_0(\omega) := \left([[A^{0,1},\omega],\delbs v]\right)^{0,2}:C^\infty(B^4, M_n(\C))\rightarrow C^\infty(\Omega^2 B^4\otimes M_n(\C))$$ and
$$H_1(\omega) := \left([\omega,\delbs v]\right)^{0,2}:C^\infty(\Omega^1 B^4\otimes M_n(\C))\rightarrow C^\infty(\Omega^2 B^4\otimes M_n(\C))$$ which satisfies the fact that $H_0(\omega(z))=H_0(\omega)(z)$. Moreover, we have that $$0= B_{0,U}v= [\delb U + [A^{0,1},U],\delbs v] = H_1(dU) + H_0(U)$$ for all smooth functions $U$ on $B^4$. Applying Proposition \ref{app:vanishing2} to $H_1$ and $H_0$, we obtain that $H_1\circ d=0$ and $H_0=0$. In particular, we have obtained that for all $U$ smooth functions on $B^4$, $$H_1(dU) = [\delb U,\delbs v] = 0.$$ This is a contradiction by Proposition \ref{highenergy:k_null_lin}. Hence, there exists $U_v\in  C^\infty(B^4,M_n(\C))$ so that $B_{0,U_v} v \neq 0$.\\

Next, we show that such a $U_v$ can be chosen to be Hermitian. Indeed since $U_v\in M_n(\C)$, there exists a decomposition in terms of its Hermitian and anti-Hermitian part: $$U_v = U_1+U_2,$$ where $U_1\in C^\infty(B^4,\mathfrak{u}(n))$ and $U_2\in C^\infty(B^4,i\mathfrak{u}(n))$. Assume that $B_{0,U_1}v= 0$, otherwise we redefine $U_v := U_1$. Under this assumption, by linearity it then necessarily follows that $$B_{0,U_2}v\neq 0.$$ If this condition holds, then by multiplying with $i$, $$iB_{0,U_2}v = B_{0,iU_2}v\neq 0.$$ Moreover, $iU_2\in C^\infty(B^4,\mathfrak{u}(n))$ and in this case we redefine $U_v:=iU_2$. Hence, there exists 
\be\label{highenergy:nlin_uv}U_v\in  C^\infty(B^4,\mathfrak{u}(n)) \text{ so that } B_{0,U_v}v\neq 0, \text{ for any }v\in Ker T_0, v\neq 0.\ee

\textbf{Claim.} There exists $U$ smooth Hermitian function such that  $B_{0,U}$ is injective on $Ker T_0$.\\

We formulate the following inductive hypothesis: 
$$
\mathcal I(k) =\left\{\begin{array}{ll} &\text{there exists } U^k\in C^\infty(B^4,\mathfrak u(n))\text{ supported in } V^k\subsetneq B^4 \text{ such that }\\
& \{B_{0,U^k}e_j\}_{j=1}^k \text{ is linearly independent},
\end{array}\right.$$ where $k\leq N$
We show by induction that $\mathcal I(N)$ holds from which it follows that $B_{0,U^N}$ is injective on $Ker T_0$.\\

By \eqref{highenergy:nlin_uv}, there exists $U_1$ such that $B_{0,U_1}e_1\neq 0$. Without loss of generality, by multiplying with a compactly supported $\rho_1$, we can localise $U_1$ in a neighbourhood $V_1\subsetneq B^4$. Hence $\mathcal I(1)$ holds. Assume that for $k<N$, $\mathcal I(k)$ holds. We prove that $\mathcal I(k+1)$ holds as well. \\

If $\{B_{0,U^k}e_j\}_{j=1}^{k+1}$ is linearly independent, then set $U^{k+1}=U^k$. Otherwise there exists $\lambda_1,\ldots,\lambda_{k+1}$ not all $0$ such that $\sum_{i=1}^{k+1}\lambda_i B_{0,U^k} e_{i}=0$. Notice that $\lambda_{k+1}\neq 0$.\\

By \eqref{highenergy:lin_uv} there exists $U_{k+1}$ such that $B_{0,U_{k+1}}\sum_{i=1}^{k+1}\lambda_i e_{i}\neq 0$. We can choose a neighbourhood $V_{k+1}$ and $\tilde V^k\subseteq V^k$ disjoint from $V_{k+1}$ such that $\{B_{0,U^k}e_j\}_{j=1}^k$ is linearly independent in $\tilde V^k$ and $$B_{0,\rho_{k+1} U_{k+1}}\sum_{i=1}^{k+1}\lambda_i e_{i}\neq 0\qquad \text{ in } V_{k+1}$$
In particular, we can define functions $\rho_{k+1}$ compactly supported in $V_{k+1}$, $\rho_k$ compactly supported in $\tilde V^k$. Define $U^{k+1} := \rho_k U^k + \rho_{k+1}U_{k+1}$.\\

It remains to show that $\{B_{0,U^{k+1}}e_j\}_{j=1}^{k+1}$ is linearly independent. Assume there exists $\beta_1, \ldots, \beta_{k+1}$ such that
\be\label{highenergy:nlin_us}\sum_{j=1}^{k+1}\beta_j B_{0,\rho_k U^{k} + \rho_{k+1}U_{k+1}}e_j = \sum_{j=1}^{k+1}\beta_j B_{0,U^{k+1}}e_j = 0.\ee
In the neighbourhood $\tilde V^k$, we have that
$$\sum_{j=1}^{k+1}\beta_j B_{0,U^{k}}e_j =0.$$
Then $(\beta_1,\ldots,\beta_{k+1}) = c(\lambda_1,\ldots,\lambda_{k+1})$ for some constant $c$. Hence, in $V_{k+1}$, we have that
$$c\sum_{j=1}^{k+1}\lambda_j B_{0,U_{k+1}}e_j=0.$$
By the choice of $U_{k+1}$, we obtain that $c=0$. Hence $\beta_1=\ldots=\beta_{k+1} = 0$. To conclude, define $V^{k+1}=\tilde V^k \cup V_{k+1}$. This proves the induction. \\

Hence, we have obtained $U=U^N$ such that $\{B_{0,U}e_j\}_{j=1}^N$ are linearly independent, where $e_j$ the orthonormal basis of  $Ker T_0$ we have picked initially. It follows that $B_{0,U}$ is injective on $Ker T_0$.
\end{proof}
\vspace{4mm}

The following Lemma proves our perturbation result. 

\begin{Lm}{\label{highenergy:k_null_nlin}}
There exists a small constant $\beta\in[0,1]$ and $U\in C^\infty(B^4,\mathfrak{u}(n))$ such that 
$$Ker T_{\beta,U} = \{0\}$$
where 
$$T_{\beta, U} : W_D^{2,2}(\Omega^{0,2} B^4\otimes M_n(\C))\rightarrow L^2(\Omega^{0,2} B^4\otimes M_n(\C)).$$
Hence, $T_{\beta,U}$ is an invertible operator.
\end{Lm}

\begin{proof}[Proof of Lemma \ref{highenergy:k_null_nlin}]
We argue as before:\\

The case when $\beta=0$ and $Ker T_0 = \{0\}$ is trivial. We focus on the case when $Ker T_0 \neq \{0\}$.\\ 

We assume the worst case scenario $dim G_0 =\infty$. By Proposition \ref{highenergy:k_null_nu}, there exists $U\in C^\infty(B^4,\mathfrak{u}(n))$ so that $B_{0,U}$ is injective on $Ker T_0$. To further set up our proof, it follows from Proposition \ref{highenergy:p_iso} that there exists an isomorphism $P_{\beta,U}$ between $G_0$ and $G_{\beta,U}$ for all $\beta<\beta_0$ for some $\beta_0>0$. We want to show the existence of $\beta$ so that $Ker T_{\beta,U} = \{0\}$.\\

Assume that for all $\beta<\beta_0$ we have that  $Ker T_{\beta,U} \neq \{0\}$. We aim at showing by contradiction that for some $\beta<\beta_0$ we will get that $Ker T_{\beta,U} = \{0\}$. Thus, let the space $$S_{\beta,U} = P_{\beta,U}^{-1}(Ker L_{\beta,U}).$$
Since $Ker T_{\beta, U}$ is finite dimensional, its dimension is bounded by the dimension of $Ker T_0$ by Proposition \ref{highenergy:no_ev}. Because $P_{\beta, U}$ is an isomorphism, then $S_{\beta, U}$ is a finite dimensional space in $G_0$ of size at most $dim KerT_0$. Moreover, $Ker T_0\cap S_{\beta, U}$ is also finite dimensional and there exists an orthonormal basis of this space. We can complete it, to obtain an orthonormal basis $\{e_j^{\beta,U}\}_{j=1}^N$ on $S_{\beta, U}$, where $N=dim Ker T_{\beta, U} = dim S_{\beta, U}$. Fix $1\leq j\leq N$ and $\e>0$ small enough such that $\lambda\in \rho(T_0)$ for all $|\lambda|=\e$. We compute the following:

We compute the following:
\begin{align}\label{high:gauge_vanish}
0 &= T_{\beta,U}P_{\beta,U} e_j^{\beta,U}  =  -\frac{1}{2\pi i} T_{\beta,U}\oint_{|\lambda|=\epsilon} ( T_{\beta,U}  - \lambda)^{-1} e_j^{\beta,U} \,d\lambda\nonumber\\
&=-\frac{1}{2\pi i} (T_0 + \beta B_{\beta,U})\oint_{|\lambda|=\epsilon} ( T_0 + \beta B_{\beta,U} - \lambda)^{-1} e_j^{\beta,U} \,d\lambda\nonumber\\
&=-\frac{1}{2\pi i} T_0\oint_{|\lambda|=\epsilon} ( T_0 - \lambda)^{-1} e_j^{\beta,U} \,d\lambda \nonumber\\
&\quad  -\frac{1}{2\pi i}\beta \left( B_{\beta,U}  -T_0  \oint_{|\lambda|=\epsilon} ( T_0  - \lambda)^{-1}B_{\beta,U} ( T_0  - \lambda)^{-1}  e_j^{\beta,U} \,d\lambda\right) + O(\beta^2)\nonumber\\
&=T_0 e_j^{\beta,U} - \frac{1}{2\pi i}\beta \left( B_{\beta,U}  -T_0  \oint_{|\lambda|=\epsilon} ( T_0  - \lambda)^{-1}B_{\beta,U} ( T_0  - \lambda)^{-1}  e_j^{\beta,U} \,d\lambda\right) + O(\beta^2)
\end{align}

The last equality holds because $e_j^{\beta,U}\in G_0$ and we have $$P_0e_j^{\beta,U} = -\frac{1}{2\pi i} \oint_{|\lambda|=\epsilon} ( T_0  - \lambda)^{-1} e_j^{\beta,U} \,d\lambda = e_j^{\beta,U}.$$
By further expanding $B_{\beta,U}$ in $\beta$ we can rewrite the vanishing equation \eqref{high:gauge_vanish} as such:
\be\label{high:gauge_vanish2}0 = T_0 e_j^{\beta,U} - \frac{1}{2\pi i}\beta \left( B_{0,U}  -T_0  \oint_{|\lambda|=\epsilon} ( T_0  - \lambda)^{-1}B_{0,U} ( T_0  - \lambda)^{-1}  e_j^{\beta,U} \,d\lambda\right) + O(\beta^2).\ee We discuss two cases:\\

\textbf{Case 1.} $e_j^{\beta,U}\notin Ker T_0$\\

Because $e_j^{\beta,U}$ is an element of the orthonormal basis, then $e_j^{\beta,U}\in (Ker T_0)^\perp$. Moreover, since $T_0$ is Fredholm, we have
$$1=\norm{e_j^{\beta,U}}{W_D^{2,2}} \leq C\norm{T_0e_j^{\beta,U}}{L^2}  \leq O(\beta)\norm{e_j^{\beta,U}}{W_D^{2,2}} = O(\beta),$$
where $C>0$ is a constant independent of $\beta$. Since $\beta<\beta_0$ is small, we get a contradiction. \\

\textbf{Case 2.} $e_j^{\beta,U}\in Ker T_0$ \\

Because the equation \eqref{high:gauge_vanish2} holds for any $\beta<\beta_0$ and $Ker T_0\neq \{0\}$ we then have that
$$B_{0,U}e_j^{\beta,U}  =  T_0  \oint_{|\lambda|=\epsilon} ( T_0  - \lambda)^{-1}B_{0,U} ( T_0  - \lambda)^{-1} e_j^{\beta,U} \,d\lambda.$$
By computing in an analogous way as in Lemma \ref{highenergy:k_null_lin} we get that the equation above implies that $B_{0,U} e_j^{\beta,U} = 0$. This is a contradiction with our initial choice of $U$. Thus, we found $\beta$ and $U$ so that $Ker T_{\beta,U} =\{0\}$.

\end{proof}

\subsection{Local density result}\label{highenergy:mainsec}

We start this section by first proving the existence of trivial kernels for approximating smooth connection $1$-forms. Secondly, we prove the existence of perturbations that give us the integrability condition \eqref{II.1} in $B^4$. Finally, we will end this section with proving the main result - that we can always approximate connection forms by smooth ones in $B^4$, in such a way that we satisfy the integrability condition \eqref{II.1} throughout 

\begin{Prop}{\label{highenergy:seq_triv_ker}}
Let $A\in W^{1,2}(\Omega^1 B^4\otimes \mathfrak{u}(n))$ and a smooth sequence of forms $A_k\rightarrow A$ in $W^{1,2}$. Then there exists a gauge $g\in C^\infty(B^4, U(n))$ and $k_0\in\N$ such that
\begin{enumerate}
\item[(i)]
$$Ker T_{\left({A^{0,1}}\right)^g}=\{0\} \qquad\text{and}\qquad Ker T_{\left(A_k^{0,1}\right)^g}=\{0\}$$ for all $k\geq k_0$. In particular, the operators  $T_{\left({A_k^{0,1}}\right)^g}$ and  $T_{\left({A^{0,1}}\right)^g}$ are all invertible.
\item[(ii)]
$$\sup_{k\geq k_0}\opnorm{T_{\left(A_k^{0,1}\right)^g}^{-1}}\leq 2\opnorm{T_{\left(A^{0,1}\right)^g}^{-1}}.$$
\end{enumerate}
\end{Prop}

\begin{proof}[Proof of Proposition \ref{highenergy:seq_triv_ker}]\ 

(i) The existence of a unitary smooth gauge $g$ is given by Lemma \ref{highenergy:k_null_nlin}. We have that $$Ker T_{\left(A^{0,1}\right)^g}=\{0\}.$$ 
It remains to prove that there exists $k_0\in\N$ so that $Ker T_{\left(A_k^{0,1}\right)^g}=\{0\}$ for all $k\geq k_0$. \\

In order to prove this statement, we assume by contradiction that $Ker T_{\left(A_k^{0,1}\right)^g}\neq \{0\}$ and let $0\neq \omega_k\in Ker T_{\left(A_k^{0,1}\right)^g}$. We can also assume without loss of generality that $\norm{\omega_k}{W_D^{2,2}(B^4)} = 1$. Since $Ker T_{\left(A^{0,1}\right)^g}$ is trivial and Fredholm, we then get (see \cite[Lemma 4.3.9]{salamon2018functional}):

$$1= \norm{\omega_k}{W_D^{2,2}(B^4)} \leq C_{A,g}\norm{T_{\left(A^{0,1}\right)^g}\omega_k}{L^2(B^4)},$$
for some constant $C_{A,g}>0$ depending on the initial $1$-form $A$ and on the gauge change $g$.

We compute this further:
$$1 \leq C_{A,g}\norm{T_{\left(A^{0,1}\right)^g}\omega_k}{L^2(B^4)} = \norm{T_{\left({A^{0,1}}\right)^g}\omega_k - T_{\left(A_k^{0,1}\right)^g}\omega_k }{L^2(B^4)} = \norm{\left[\left(A_k^{0,1}\right)^g- \left(A^{0,1}\right)^g,\delbs\omega_k\right]}{L^2(B^4)}.$$

Indeed, we can bound the last bracket above by $\norm{A_k-A}{W^{1,2}(B^4)}$:
$$\begin{array}{ll}
\norm{\left[\left(A_k^{0,1}\right)^g- \left(A^{0,1}\right)^g,\delbs\omega_k\right]}{L^2(B^4)} &\leq C_g \norm{A_k-A}{L^4(B^4)}\norm{\omega_k}{L^4(B^4)}\\
&\leq C_g\norm{A_k-A}{W^{1,2}(B^4)}\norm{\omega_k}{W_D^{2,2}(B^4)} \\
&= C_g\norm{A_k-A}{W^{1,2}(B^4)}.
\end{array}$$
for some constant $C_g$ depending on $g$. Since the constants are independent of $k$, it follows that
$$1\leq  C_{A,g}\cdot C_g\norm{A_k-A}{W^{1,2}(B^4)} \rightarrow 0\text{ as } k\rightarrow \infty.$$
Thus, for $k$ large enough the above inequality yields a contradiction. We then have that there exists $k_0$ large so that  $\omega_k=0$ and that $Ker T_{\left(A_k^{0,1}\right)^g}= \{0\}$ for all $k\geq k_0$. We conclude that since  $T_{\left({A^{0,1}}\right)^g}$ and $ T_{\left({A_k^{0,1}}\right)^g}$ for each $k\geq k_0$ are all operators of index zero and their kernel is trivial, they are invertible.\\

(ii) Since $T_{\left(A^{0,1}\right)^g}$ has trivial kernel and is an operator of index zero, its inverse exists mapping $L^2$ to $W_D^{2,2}$ and we have that $\opnorm{T_{\left(A^{0,1}\right)^g}^{-1}}<\infty$. Since $A_k$ converges strongly to $A$ in $W^{1,2}$ and $g$ is smooth by construction, then we can assume without loss of generality that $\opnorm{T_{\left(A_k^{0,1}\right)^g}-T_{\left(A^{0,1}\right)^g}}\opnorm{T_{\left(A^{0,1}\right)^g}^{-1}}<\frac 12$ for any $k\geq k_0$. Hence \cite[Theorem 1.5.5(iii)]{salamon2018functional}, gives that $$\opnorm{T_{\left(A_k^{0,1}\right)^g}^{-1}-T_{\left(A^{0,1}\right)^g}^{-1}}\leq \cfrac{\opnorm{T_{\left(A_k^{0,1}\right)^g}-T_{\left(A^{0,1}\right)^g}}\opnorm{T_{\left(A^{0,1}\right)^g}^{-1}}}{1-\opnorm{T_{\left(A_k^{0,1}\right)^g}-T_{\left(A^{0,1}\right)^g}}\opnorm{T_{\left(A^{0,1}\right)^g}^{-1}}}\opnorm{T_{\left(A^{0,1}\right)^g}^{-1}}\leq \opnorm{T_{\left(A^{0,1}\right)^g}^{-1}}.$$
Thus, for $k\geq k_0$, we have
$$\opnorm{T_{\left(A_k^{0,1}\right)^g}^{-1}}\leq \opnorm{T_{\left(A_k^{0,1}\right)^g}^{-1}-T_{\left(A^{0,1}\right)^g}^{-1}} + \opnorm{T_{\left(A^{0,1}\right)^g}^{-1}}\leq 2\opnorm{T_{\left(A^{0,1}\right)^g}^{-1}}.$$
Hence, by taking the $\sup$ over all $k\geq k_0$, we obtain the result.
\end{proof}

\vspace{4mm}

The following Lemma proves the existence of a perturbation under the conditions that $T_{A^{0,1}}$ has trivial kernel and that $F_A^{0,2}$ is small in $L^2$ norm.

\begin{Lm}{\label{highenergy:seq_flat}}
There exists a constant $C>0$ such that for every $A\in W^{1,2}(\Omega^1 B^4\otimes \mathfrak{u}(n))$ with $T_{A^{0,1}}$ invertible and satisfying $\norm{F_{A}^{0,2}}{L^2} \leq \frac{C}{\opnorm{T_{A^{0,1}}^{-1}}^2}$, there exists $\omega\in W_D^{2,2}(\Omega^{0,2} B^4\otimes M_n(\C))$ a solution of the PDE:
$$\delb\delbs \omega + [{A^{0,1}},\delbs \omega] + \delbs \omega\wedge\delbs\omega = -F_{A}^{0,2}$$
 satisfying the estimate\be\label{highenergy:bounds}\norm{\omega_0}{W_D^{2,2}(B^4)} \leq C'\opnorm{T_{A^{0,1}}^{-1}}\norm{F_A^{0,2}}{L^2(B^4)}\ee
 where $C'>0$ is a constant independent of $A$.
\end{Lm}

\begin{proof}[Proof of Lemma \ref{highenergy:seq_flat}]
We construct the following sequence of solutions:
$$\begin{array}{rl}
T_{A^{0,1}}\omega_0 &= - F_{A}^{0,2}\\
T_{A^{0,1}}\omega_1& = -  \delbs \omega_0 \wedge\delbs\omega_0 - F_{A}^{0,2}\\
T_{A^{0,1}}\omega_2 &= -  \delbs \omega_1 \wedge\delbs\omega_1 - F_{A}^{0,2}\\
\ldots&\\
T_{A^{0,1}}\omega_k&= -  \delbs \omega_{k-1} \wedge\delbs\omega_{k-1} - F_{A}^{0,2}\\
\ldots&
\end{array}
$$
\textbf{Claim}. $\{\omega_k\}_{k=0}^\infty$ is a Cauchy sequence in $W_D^{2,2}$.\\\\
We first show by induction the uniform bound on the sequence $\norm{\omega_k}{W_D^{2,2}}\leq 2\opnorm{T_{A^{0,1}}}\norm{F_{A}^{0,2}}{L^2(B^4)}$. Since $T_{A^{0,1}}$ is invertible, then we have the identity:
$$\omega_0 =T_{A^{0,1}}^{-1} T_{A^{0,1}}\omega_0.$$
Hence, from the definition of the norm of operators, it follows that:
$$\norm{\omega_0}{W_D^{2,2}(B^4)} \leq \opnorm{T_{A^{0,1}}^{-1}}\norm{T_{{A^{0,1}}}\omega_0}{L^2(B^4)}= \opnorm{T_{A^{0,1}}^{-1}}\norm{F_{A}^{0,2}}{L^2(B^4)} < 2\opnorm{T_{A^{0,1}}^{-1}}\norm{F_{A}^{0,2}}{L^2(B^4)}$$
Let $k>0$. By the Sobolev embedding of $W^{1,2}\hookrightarrow L^4$ there exists a constant $C_1>0$ so that
$$
\norm{\delbs \omega_k}{L^4(B^4)}\leq C_1\norm{\delbs \omega_k}{W^{1,2}(B^4)}\leq C_1 \norm{\omega_k}{W_D^{2,2}(B^4)}.$$ Then we have:
$$\begin{array}{rl}
\norm{\omega_k}{W_D^{2,2}(B^4)} &\leq  \opnorm{T_{A^{0,1}}^{-1}}\norm{T_{{A^{0,1}}}\omega_k}{L^2(B^4)}\\
&\leq \opnorm{T_{A^{0,1}}^{-1}}\norm{\delbs\omega_{k-1}\wedge\delbs\omega_{k-1}}{L^2(B^4)} +\opnorm{T_{A^{0,1}}^{-1}}\norm{F_A^{0,2}}{L^2(B^4)}\\
&\leq C_1\opnorm{T_{A^{0,1}}^{-1}}\norm{\omega_{k-1}}{L^4(B^4)}^2 +\opnorm{T_{A^{0,1}}^{-1}}\norm{F_A^{0,2}}{L^2(B^4)}\\
&\leq C_1^2\opnorm{T_{A^{0,1}}^{-1}}\norm{\omega_{k-1}}{W_D^{2,2}(B^4)}^2 +\opnorm{T_{A^{0,1}}^{-1}}\norm{F_A^{0,2}}{L^2(B^4)}
\end{array}$$
By the induction hypothesis we have that $\norm{\omega_{k-1}}{W_D^{2,2}(B^4)}<2\opnorm{T_{A^{0,1}}^{-1}}\norm{F_A^{0,2}}{L^2(B^4)}$. Thus, 
$$\norm{\omega_k}{W_D^{2,2}(B^4)}\leq 4C_1^2\opnorm{T_{A^{0,1}}^{-1}}^3\norm{F_A^{0,2}}{L^2(B^4)}^2+ \opnorm{T_{A^{0,1}}^{-1}}\norm{F_A^{0,2}}{L^2(B^4)}.$$
Choosing the constant $C>0$ such that $\frac{1}{C}<4C_1^2$, we obtain by assumption $$4C_1^2\opnorm{T_{A^{0,1}}^{-1}}^2\norm{F_A^{0,2}}{L^2(B^4)}<C\opnorm{T_{A^{0,1}}^{-1}}^2\norm{F_A^{0,2}}{L^2(B^4)}\leq 1,$$ and we conclude that
$$\norm{\omega_k}{W_D^{2,2}}\leq 2\opnorm{T_{A^{0,1}}^{-1}}\norm{F_A^{0,2}}{L^2(B^4)}$$
Hence, by induction it follows that the sequence is uniformly bounded in $W_D^{2,2}$:
$$\norm{\omega_k}{W_D^{2,2}(B^4)}\leq 2\opnorm{T_{A^{0,1}}^{-1}}\norm{F_A^{0,2}}{L^2(B^4)}$$ for all $k\geq 0$. 
It remains to show that $\{\omega_k\}_{k=0}^\infty$ is a Cauchy sequence.\\

Let $k>0$. It follows  that 
$$\begin{array}{rll}
\norm{\omega_{k+1} - \omega_{k}}{W_D^{2,2}(B^4)} &\leq& \opnorm{T_{A^{0,1}}^{-1}}\norm{T_{{A^{0,1}}}(\omega_{k+1}-\omega_k)}{L^2(B^4)}\\
&\leq& \opnorm{T_{A^{0,1}}^{-1}}\norm{\delbs(\omega_k-\omega_{k-1})\wedge\delbs\omega_k}{L^2(B^4)}\\
& &+\opnorm{T_{A^{0,1}}^{-1}}\norm{\delbs\omega_{k-1}\wedge \delbs(\omega_k-\omega_{k-1})}{L^2(B^4)}\\
&\leq& 4C_1\opnorm{T_{A^{0,1}}^{-1}}^2\norm{F_A^{0,2}}{L^2(B^4)}  \norm{\omega_k-\omega_{k-1}}{W_D^{2,2}(B^4)}
\end{array}
$$
Choosing the constant $C>0$ such that $\frac{1}{C}<4C_1$, we obtain by assumption $$4C_1\opnorm{T_{A^{0,1}}^{-1}}^2\norm{F_A^{0,2}}{L^2(B^4)} < \frac{1}{C}\opnorm{T_{A^{0,1}}^{-1}}^2\norm{F_A^{0,2}}{L^2(B^4)}\leq 1,$$ and we can conclude that the sequence $\{\omega_k\}_{k=0}^\infty$ is Cauchy. Hence, we have proven the claim. \\

Moreover, because $W_D^{2,2}$ is a Banach space and $\{\omega_k\}_{k=0}^\infty$ is a Cauchy sequence, there exists $\omega_\infty$ such that $\omega_k \rightarrow \omega_\infty$ as $k\rightarrow \infty$ in $W_D^{2,2}$. Moreover, by the strong convergence and the uniform bound of the sequence we obtain that $\norm{\omega_\infty}{W_D^{2,2}(B^4)} \leq 2\opnorm{T_{A^{0,1}}^{-1}}\norm{F_A^{0,2}}{L^2(B^4)} $ and 
$$\delb\delbs \omega_\infty + [{A^{0,1}},\delbs \omega_\infty] + \delbs \omega_\infty\wedge\delbs\omega_\infty = -F_{A}^{0,2}.$$
\end{proof}

We can conclude this section with the main result. 
\addtocounter{Th}{-2}
\begin{Th}\label{highenergy:mainthm_lin}
Let $A\in W^{1,2}(\Omega^1 B^4\otimes \mathfrak{u}(n))$, with $F_A^{0,2} = 0$. There exists a smooth sequence of forms $A_k \in C^\infty(\Omega^1 B^4\otimes \mathfrak{u}(n))$, $F_{A_k}^{0,2}=0$ and $A_k \rightarrow A$ in $W^{1,2}(\Omega^1 B^4\otimes \mathfrak{u}(n))$.
\end{Th}

\begin{proof}[Proof of Theorem \ref{highenergy:mainthm_lin}]
By Lemma \ref{highenergy:k_null_nlin} there exists a unitary gauge change $g\in C^\infty(B^4,U(n))$ so that $T_{\left(A^{0,1}\right)^g}$ is invertible $W_D^{2,2}$ to $L^2$. Moreover, we can obtain a smooth sequence $\tilde{A}_{k}$ by simple convolution such that $\norm{F_{\tilde{A}_{k}}^{0,2}}{L^2(B^4)}\rightarrow 0 $ and $\tilde{A}_{k} \rightarrow A$ in $W^{1,2}$ as $k\rightarrow \infty$. By Proposition \ref{highenergy:seq_triv_ker}(i) there exists $k_0$ so that for all $k\geq k_0$, $T_{{\tilde{A}_{k}}^g}$ is invertible and that $\tilde{A}_{k}^g \rightarrow A^g$. Moreover, since the change of gauge is unitary we also have that $\norm{F_{\tilde{A}_{k}}^{0,2}}{L^2(B^4)} = \norm{F_{{\tilde{A}_{k}}^g}^{0,2}}{L^2(B^4)}.$\\

Moreover, by Proposition \ref{highenergy:seq_triv_ker}(ii) we know that 
$$\sup_{k\geq k_0} \opnorm{T_{\left(\tilde A_k^{0,1}\right)^g}^{-1}}\leq 2\opnorm{T_{\left(A^{0,1}\right)^g}^{-1}}.$$
Thus, for all $k\geq k_0$, it follows that 
$\frac 1 {\opnorm{T_{\left(\tilde A_k^{0,1}\right)^g}^{-1}}^{2}}\geq \frac 1 {2\opnorm{T_{\left( A^{0,1}\right)^g}}^2}$.
Let $C>0$ the constant given by Lemma \ref{highenergy:seq_flat}. There exists $k_1>k_0$ such that $$\norm{F_{{\tilde{A}_{k}}^g}^{0,2}}{L^2(B^4)}\leq \frac C {2\opnorm{T_{\left( A^{0,1}\right)^g}}^2}\leq \frac C {\opnorm{T_{\left(\tilde A_k^{0,1}\right)^g}^{-1}}^{2}}.$$
for all $k\geq k_1$. Hence, for each $k\geq k_1$, Lemma \ref{highenergy:seq_flat} gives the existence of $(0,2)$ forms $\omega_k$ satisfying the estimate 
$$\norm{\omega_k}{W_D^{2,2}(B^4)}\leq C'\opnorm{T_{\left(A_k^{0,1}\right)^g}^{-1}}\norm{F_{{\tilde{A}_{k}}^g}^{0,2}}{L^2(B^4)}\leq 2C'\opnorm{T_{\left(A^{0,1}\right)^g}^{-1}}\norm{F_{{\tilde{A}_{k}}}^{0,2}}{L^2(B^4)},$$
where $C'>0$ is a constant independent of $k$. Moreover, each $\omega_k$ solve the PDE:
\be\label{high:flat}\delb\delbs \omega_k + \left[{\left(\tilde A_k^{0,1}\right)^g},\delbs \omega_k\right] + \delbs \omega_k\wedge\delbs\omega_k = -F_{\tilde A_k^g}^{0,2},\ee
i.e. $F_{\tilde A_k^g + \delbs\omega_k}^{0,2}=0$. Since $\norm{F_{{\tilde{A}_{k}}}^{0,2}}{L^2(B^4)}$ converges strongly to $0$, the estimates on the $(0,2)$ forms $\omega_k$ give $\omega_k\rightarrow 0$ in $W_D^{2,2}$ as $k\rightarrow \infty$. Thus, we obtain the strong convergence $$\left({\tilde{A}_{k}}^{0,1}\right)^{g} +\delbs\omega_k\rightarrow \left(A^{0,1}\right)^g\text{ in }W^{1,2}.$$
Define the sequence of connection forms
$$A_k := \left(({\tilde{A}_{k}}^{0,1})^{g} +\delbs\omega_k- \ov{({\tilde{A}_{k}}^{0,1})^{g} +\delbs\omega_k}^T\right)^{g^{-1}} \in W^{1,2}(\Omega^1 B^4\otimes\mathfrak{u}(n))\cap C^\infty.$$ Because $g$ is a smooth unitary gauge, and $A_k^g\rightarrow A^g$ in $W^{1,2}$ by construction,  then this sequence of forms are unitary and convergent in $W^{1,2}$. We need to establish that $A_k\rightarrow A$ in $W^{1,2}$. Indeed, we obtain the following $L^2$ convergence: 
$$\norm{A_k^g - A^g}{L^2}\rightarrow 0 \iff  \norm{g^{-1} \left(A_k - A\right)g}{L^2}\rightarrow 0 \stackrel{g\in U(n)}{\iff} \norm{A_k - A}{L^2} \rightarrow 0.$$
Because the limit is unique, then $A_k\rightarrow A$ in $W^{1,2}$. Moreover, the smooth sequence $A_k$ satisfies the integrability condition $F_{A_k}^{0,2}=0$ by construction of $\omega_k$ in \eqref{high:flat}. 
\end{proof}

\subsection{Global density result}\label{highenergy:globalmainsec}

In this section we will use the result we have proven in the above section in order to obtain a global result for a closed K\"ahler manifold $X^2$.  In order to be able to generalise, we will work on sections of the vector bundle $(E,h_0)$ over $X^2$. The $\delb$ operator is well-defined and acts on the space of $E$-valued $(p,q)$-forms:
$$\delb : \mathcal A^{p,q}(E) \to \mathcal A^{p,q+1}(E).$$
Its corresponding dual operator, $\delbs$, is defined as a map:
$$\delbs:\mathcal A^{p,q}(E)\to \mathcal A^{p,q-1}(E).$$ On the space $\mathcal A^{p,q}(E)$ the $\delb$-Hodge theorem \cite{griffiths2014principles} gives the orthogonal $L^2$ decomposition:
\be\label{high:globaldec}\mathcal A^{p,q}(E) = \delb \mathcal A^{p,q-1}(E) + \delbs \mathcal A^{p,q+1}(E) + \mathcal H^{p,q}(E),\ee
where $\mathcal H^{p,q}(E)$ is the space of holomorphic $(p,q)$-sections. Since $X^2$ is a closed K\"ahler manifold, then we remark that $\mathcal H^{p,q}(E)$ is finite dimensional. In particular, by (\ref{high:globaldec}) any $(0,2)$-form $\omega\in \mathcal A^{0,2}(E)$ can be orthogonally decomposed as follows:
$$\omega = \delb\delbs \alpha+ h,$$
where $h\in \mathcal H^{0,2}(E)$ and $\alpha\in\mathcal A^{0,2}(E)$. \\

Since $\delb$ and $\delbs$ define elliptic complexes over closed K\"ahler surfaces (see for example \cite[Chapter IV]{nash1991differential}):
$$0\rightarrow \mathcal A^{0,0}(E)\overset{\delb}{ \rightarrow} \mathcal A^{0,1}(E)\overset{\delb}{ \rightarrow}\mathcal A^{0,2}(E) \rightarrow 0$$
and
$$0\rightarrow \mathcal A^{0,2}(E)\overset{\delbs}{ \rightarrow} \mathcal A^{0,1}(E)\overset{\delbs}{ \rightarrow}\mathcal A^{0,0}(E) \rightarrow 0$$
then the operator $\delb\delbs$ is a elliptic on $(0,2)$-sections over closed K\"ahler surfaces. In particular it is Fredholm and moreover, $\delb\delbs$ is self-adjoint. Thus, its Fredholm index vanishes:
\be\label{high:index} index(\delb\delbs) = dim Ker(\delb\delbs)- dim Coker(\delb\delbs) = dim Ker(\delb\delbs)- dim Ker(\delb\delbs) = 0.\ee 

We redefine our operator $T_{A^{0,1}}$ as such:
$$T_{\nabla}: \Gamma_{W^{2,2}}(\mathcal A^{0,2}(E))\to \Gamma_{L^2}(\mathcal A^{0,2}(E))$$
where $\nabla$ is a $W^{1,2}$ unitary connection over $X^2$ and $\Gamma_{W^{p,q}}$ is the space of locally $W^{p,q}$ sections. We can directly apply Lemma \ref{highenergy:k_null_nlin} to obtain the existence of a global smooth section $g$ so that $Ker T_{\nabla^g}=0$. Moreover, by (\ref{high:index}) and Proposition \ref{highenergy:fred} applied to $T_{\nabla^g}$ it follows that $T_{\nabla^g}$ is a Fredholm operator of index $0$. Hence, $Coker T_{\nabla^g}$ is empty and the operator $T_{\nabla^g}$ is invertible. \\

It follows that we can apply Lemma \ref{highenergy:seq_flat} to $\nabla^g$ and we can conclude that, similarly to Theorem \ref{highenergy:mainthm}, we have proven:

\begin{Th}{\label{highenergy:globalmainthm}}
Let $\nabla$ a $W^{1,2}$ unitary connection over $X^2$, satisfying the integrability condition $F_{\nabla}^{0,2} = 0$. Then there exists a sequence of smooth unitary connections $\nabla_k$, with $F_{\nabla_k}^{0,2}=0$ such that
$$d_2(\nabla_k,\nabla)\rightarrow 0.$$
\end{Th}

\section{Proof of Theorem \ref{thm2}}\label{sec:thm2}

We are ready to prove the second result of our paper. 

\begin{proof}\ 
By Theorem \ref{highenergy:globalmainthm} applied to the given $W^{1,2}$ connection $\nabla$, we obtain the existence of a sequence of smooth connections $\nabla_k$ that converge to $\nabla$ in the sense of $$d_2(\nabla_k,\nabla)\rightarrow 0.$$

\textit{Step 1}. There exists $r>0$ and a finite good cover $\{B_r^4(x_i)\}$ such that
$$\nabla = d+ A_i\text{ in }B_r^4(x_i)$$
with $$\norm{A_i}{L^2(B_r^4(x_i))}\leq \e_0,$$
where $\e_0>0$ is given by Theorem \ref{holo:mainthm}.
Since $\nabla_k$ converges to $\nabla$ in the sense of $d_2$, it follows that in each ball $B_r^4(x_i)$ with $\nabla_k = d+A_i^k$
we have $F_{A_i^k}^{0,2} = 0$, $A_i^k\rightarrow A_i$ in $W^{1,2}$ and $F_{A_i^k}\rightarrow F_{A_i}$ in $L^2$ as $k\rightarrow \infty$. \\

 By Theorem \ref{holo:mainthm} there exists $r'<r$ and $\sigma_i\in W^{2,p}(B_{r'}^4(x_i), GL_n(\C))$ for any $p<2$ such that  $$A_i^{0,1} = -\delb \sigma_i\cdot \sigma_i.$$
 Moreover, the gauges $\sigma_i$ define a global section $\sigma$.\\

Let $\delta>0$ be the constant in Corollary \ref{holo:stability}. There exists $k_0\geq 0$ so that $\norm{A_i^k - A_i}{W^{1,2}}\leq \delta$ for all $k\geq k_0$. Corollary \ref{holo:stability} applied to each $A_i^k$, $k\geq k_0$ gives the existence of a sequence $$\sigma_i^k\in W^{2,p}(B_{r'/2}^4(x_i), GL_n(\C))$$ of gauges that holomorphically trivialise $A_i^k$:
$$\left(A_i^k\right)^{0,1} = -\delb \sigma_i^k\cdot\left(\sigma_i^k\right)^{-1}$$
with the estimates
\be\label{thm2:estimate}\norm{\sigma_i^k-\sigma_i}{L^{p}(B_{r'/2}^4(x_i))}\leq C\norm{A_i^k - A_i}{W^{1,2}(B_r^4(x_i)}\ee
for some constant $C>0$ and any $p<12$, and for each $q>2$ there exists $C_q>0$ such that
\be\label{thm2:estimate2}\norm{\sigma_i^k-id}{L^{p}(B_{r'/2}^4(x_i))}\leq C_q\left(\norm{A_i^k}{W^{1,2}(B_r^4(x_i)} + \norm{A_i}{W^{1,2}(B_r^4(x_i)}\right)\ee
To abuse notation from now on we will use $r'$ to denote $r'/2$. From \eqref{thm2:estimate}, $\sigma_i^k\rightarrow \sigma_i$ in $L^{p}(B_{r'}^4(x_i))$ for any $p<12$. Moreover, the uniform bound \eqref{thm2:estimate2} gives that the sequence converges strongly in $W^{2,q}$ for any $q<2$. Since limits are unique, we obtain that $\sigma_i^k$ converges strongly to $\sigma_i$ in $W^{2,q}$ for any $q<2$. $A_i^k$ is smooth, it follows that each gauge $\sigma_i^k$ is smooth. \\

Define $h_i^k = \ov{\sigma_i^k}^T\cdot \sigma_i^k$. Then each $h_i^k$ is uniformly bounded in  $W^{2,p}(B_{r'}^4, iU(n))$ and $$\left(A_i^k\right)^{\sigma_i^k} = \left(h_i^k\right)^{-1}\p h_i^k\rightarrow h_i^{-1}\p h_i\qquad \text{ in } W^{1,p}(B_{r'}^4(x_i))\text{ for any }p<2.$$

Moreover, the corresponding curvature forms satisfy $$F_{\left(A_i^k\right)^{\sigma_i^k}} = \left(\sigma_i^k\right)^{-1}\left(F_{A_i^k}\right)\sigma_i^k.$$
Combining the estimate (\ref{thm2:estimate}) with the $L^2$ convergence of the sequence $F_{A_i^k}$, we have
$$F_{\left(A_i^k\right)^{\sigma_i^k}} \rightarrow F_{A_i^{\sigma_i}}\qquad \text{ in } L^p(B_{r'}^4(x_i))\text{ for any }p<2.$$
Since $i$ is arbitrary, $\{h_i^k\}_i$ defines a smooth section $h_k$ of a positive Hermitian bundle. Thus, there exists a smooth holomorphic structure $\mathcal E_k$ such that
$$\nabla_k^{\sigma_k} = \p_0 +\delb_{\mathcal E_k} + h_k^{-1}\p_0 h_k$$
and for any $p<2$ we obtain the required convergence:
$$d_p(\nabla_k^{\sigma_k}, \nabla^\sigma)\rightarrow 0,$$
where $\sigma_k:=\{\sigma_i^k\}_i\in \Gamma (GL_n(\C))$.\\

\textit{Step 2}. It remains to show that there exists bundle isomorphisms $\mathcal H_k$ between the holomorphic bundles $\mathcal E_k$ and $\mathcal E$ such that $\delb_{\mathcal E_k} = \mathcal H_k^{-1}\circ \delb_{\mathcal E}\circ  \mathcal H_k$.\\

Let $i$,$j$ such that there exist gauge transition functions $g_{ij}\in W^{2,2}(B_{r}^4(x_i)\cap B_{r}^4(x_j),U(n))$ and $g_{ij}^k\in C^\infty(B_{r}^4(x_i)\cap B_{r}^4(x_j),U(n))$ satisfying:
$$A_i^{g_{ij}} = A_j\qquad \text{ in }B_{r}^4(x_i)\cap B_{r}^4(x_j)$$
and
$$\left(A_i^k\right)^{g_{ij}^k} = A_j^k\qquad \text{ in }B_{r}^4(x_i)\cap B_{r}^4(x_j).$$
Since $g_{ij}^k\rightarrow g_{ij}$ in $W^{2,2}$ by construction, from \cite{isobe2009topological} there exist $\phi_i^k\in W^{2,2}(B_{r'}^4(x_i),U(n))$  and $\phi_j^k\in W^{2,2}(B_{r'}^4(x_i),U(n))$ such that
$$g_{ij}^k =\left( \phi_i^k\right)^{-1}\cdot  g_{ij} \cdot \phi_j^k.$$
Define $$h_{ij} = \sigma_i^{-1} \cdot g_{ij}\cdot \sigma_j\qquad \text{ and }\qquad h_{ij}^k = \left(\sigma_i^k\right)^{-1}\cdot g_{ij}^k\cdot \sigma_j^k.$$ Since $A_i^{\sigma_i}$, $A_j^{\sigma_j}$, $\left(A_i^k\right)^{\sigma_i^k}$ and $\left(\sigma_j^k\right)^{g_j^k}$ have vanishing $(0,1)$-parts, then $h_{ij}$ and $h_{ij}^k$ are transition functions for the holomorphic bundles $\mathcal E$ and $\mathcal E_k$ respectively and they satisfy
\be\label{thm:preservation}\delb h_{ij} = 0\qquad \text{ and }\qquad  \delb h_{ij}^k = 0\qquad \text{ in }B_{r'}^4(x_i)\cap B_{r'}^4(x_j)
\ee
and $\left(A_i^{\sigma_i}\right)^{h_{ij}} = A_j^{\sigma_j}$, $\left(\left(A_i^k\right)^{\sigma_i^k}\right)^{h_{ij}^k} = \left(A_j^k\right)^{\sigma_j^k}$.\\

Define the functions $\mathcal H_{i}^k := \sigma_i^{-1}\cdot \phi_i^k \cdot \sigma_i^k$ and  $\mathcal H_{j}^k = \sigma_j^{-1}\cdot \phi_j^k\cdot \sigma_j^k$. By construction, we have:
$$h_{ij}^k = \left(\mathcal H_{i}^k \right)^{-1}\cdot h_{ij}\cdot \mathcal H_{j}^k.$$
Thus, $\mathcal H_k = \{\mathcal H_i^k\}_i$ defines a bundle isomorphism and by (\ref{thm:preservation}) preserves the holomorphic structure:
 $$\delb_{\mathcal E_k} = \mathcal H_k^{-1}\circ \delb_{\mathcal E}\circ  \mathcal H_k.$$
 This finishes our proof.
\end{proof}

\section{Appendix}

\subsection*{Linear Operators}

We prove the following results that will be used in Section 4:

\begin{Prop}\label{app:vanishing1}
Let $D$ be a bounded domain in $\R^n$, $\mathfrak g$ a Lie algebra and for $k\in \N$, let $$H: C^\infty(\Omega^1 D\otimes \mathfrak g) \rightarrow C^\infty(\Omega^k D\otimes \mathfrak g)$$ be a linear operator such that $H(A)(x) = H(A(x))$ for all $x\in D$ and $A\in C^\infty(\Omega^1 D\otimes \mathfrak g)$. If for all $U\in C^\infty(D, \mathfrak g)$ we have
\be\label{app:rel1}H(dU)= 0\ee
then $H=0$.
\end{Prop}

Before we prove this statement we remark that the condition $H(A)(x) = H(A(x))$ in general prevents $H$ from being a differential operator acting on forms $A$. Otherwise the statement cannot be true. For example take $H=d$, $H(dU) = 0$ - since $d^2=0$, but $d\neq 0$.

\begin{proof}[Proof of Proposition \ref{app:vanishing1}]
Fix $A \in  C^\infty(\Omega^1 D\otimes \mathfrak g)$, then we can write it as $A = \sum_{i=1}^n a_i dx_i$. Fix $x_0\in D$ arbitrary and define the function $V(x) := \sum_{i=1}^n a_i(x_0)x_i$. Then $dV = A(x_0)$. Moreover, by \eqref{app:rel1} we have
$$H(dV) = H(A(x_0)) = H(A)(x_0)=0.$$
Hence, because $x_0$ is arbitrary, we have that $H(A) = 0$ and since $A$ was an arbitrarily chosen smooth $1$-form, then $H=0$.
\end{proof}

Next, we prove a more general statement than the one above:

\begin{Prop}\label{app:vanishing2} Let $D$ be a bounded domain in $\R^n$, $\mathfrak g$ a Lie algebra and for $k\in \N$, let $$H_0: C^\infty(D, \mathfrak g) \rightarrow C^\infty(\Omega^k D\otimes \mathfrak g)$$
and $$H_1:C^\infty(\Omega^1 D\otimes \mathfrak g) \rightarrow C^\infty(\Omega^k D\otimes \mathfrak g)$$ be linear operators such that $H_0(U)(x) = H_0(U(x))$ for all $x\in D$. If for all $U\in C^\infty(D, \mathfrak g)$ we have
\be\label{app:rel}H_0(U) + H_1(dU) = 0\ee
then $H_0=0$ and $H_1\circ d= 0$.
\end{Prop}

\begin{proof}[Proof of Proposition \ref{app:vanishing2}]
Fix $U\in C^\infty(D,\mathfrak g)$ and $x\in D$. Define $V_x: = U(x)$ a constant function. Then by \eqref{app:rel}, we have
$$H_0(V_x) = H_0(U(x))= H_0(U)(x) = 0.$$
Since $x$ is arbitrary in $D$, then we have that $H_0(U)=  0$. Since $U$ is an arbitrarily chosen smooth function, then $H_0 = 0$. Hence, from this and \eqref{app:rel}, we also obtain $H_1\circ d = 0$. This concludes the proof.
\end{proof}

\subsection*{Sobolev Estimates}

In this section of the Appendix, we will prove a few results that help us bootstrap certain $\delb$-equations. These results have been heavily used in order to prove the regularity of sections which yield a holomorphic structure over a given closed K\"ahler surface. 

\begin{Lm}\label{bootstrapcriticalgauge}
Let $\e>0$ be a small constant, $D$ a domain holomorphically embedded in $\CP^2$, a gauge $g\in L^4(D, M_n(\C))$ and $\omega\in W^{1,2}(\Omega^{0,1} D\otimes M_n(\C))$, $\norm{\omega}{W^{1,2}(D)}\leq \e$ such that $\omega$ satisfies the integrability condition (\ref{II.1}), the equation:
$$\delb g=-\omega\cdot g$$ 
is solved in a distributional sense and $\norm{g-id}{L^4(D)}\leq C\norm{\omega}{W^{1,2}(D)}$.
Then there exists a subdomain $D_0\subseteq D$, and for each $q\in (1,2)$, a constant $C_q>0$ such that: $$\norm{g-id}{W^{2,q}(D_0)}\leq C_q\norm{\omega}{W^{1,2}(D)}.$$
\end{Lm}

\begin{proof}[Proof of Lemma \ref{bootstrapcriticalgauge}] We start with a bootstrapping procedure. Firstly, let $D_0$ be a slightly smaller subdomain of $D$ such that we obtain the existence of a constant $C>0$ and the following inequality holds:
$$\norm{g-id}{W^{1,2}(D_0)}\leq C\left(\norm{\delb g}{L^2(D)} + \norm{g-id}{L^2(D)}\right)\leq C\left(\norm{\omega}{L^4(D)}\norm{g}{L^4(D)} + \norm{g-id}{L^2(D)}\right).$$
By using the embedding of $W^{1,2}$ into $L^4$ in $4$-dimensions, it follows that for some constant $C>0$, we have:
$$\norm{g-id}{W^{1,2}(D_0)}\leq C\left(\norm{\omega}{W^{1,2}(D)}\norm{g-id}{L^4(D)} +\norm{\omega}{W^{1,2}(D)} +  \norm{g-id}{L^2(D)}\right).$$
Hence, by $\norm{g-id}{L^4(D)}\leq C\norm{\omega}{W^{1,2}(D)}$, we obtain:
$$\norm{g-id}{W^{1,2}(D_0)}\leq C\norm{\omega}{W^{1,2}(D)}.$$
Once we have obtained the $W^{1,2}$ estimate on $D_0$, we can proceed to bootstrapping to $W^{2,q}$ regularity. We can apply $\delbs$ to the $\delb$-equation to obtain the elliptic PDE:
$$\delbs\delb g = -\vartheta(\omega\cdot g) = -\vartheta \omega \cdot g - \ast(\ast \omega \wedge \p g),$$
which holds in a distributional sense. Since $\delbs\delb$ is equal to the Hodge Laplacian $d^\ast d$ acting on functions, we can apply Proposition \ref{bootstrapmorrey2} and obtain that $g\in W_{loc}^{2,q}(D_0, M_n(\C)$ for all $q<2$ and for each $q<2$ there exists a constant $C_q>0$ such that \be\label{app:gest}\norm{g}{W_{loc}^{2,q}(D_0)}\leq C_q\norm{\omega}{W^{1,2}(D_0)}.\ee Without loss of generality, we can assume $g\in W^{2,q}(D_0, M_n(\C)$ for all $q<2$, otherwise we pick a slightly smaller domain than $D_0$. It remains to show the required bound. Since $\delbs\delb$ is elliptic, we have the a-priori estimate:
\be\label{app:elliptic}\norm{g-id}{W^{2,q}(D_0)}\leq C\left(\norm{\delbs\delb g}{L^q(D_0)} + \norm{g-id}{L^q(D_0)}\right),\ee
for any $q\in (1,2)$. We estimate $\delbs\delb g$:
$$\norm{\delbs\delb g}{L^q(D_0)}\leq \norm{\nabla \omega}{L^2(D_0)}\norm{g}{L^{2q/(2-q)}(D_0)} + \norm {\omega}{L^4(D_0)}\norm{\nabla g}{L^{4q/(4-q)}(D_0)}.$$
Since $W^{1,2}$ embeds into  $L^4$, it follows that
$$\norm{\delbs\delb g}{L^q(D_0)}\leq  C\norm{\omega}{W^{1,2}(D)}\left(\norm{g}{L^{2q/(2-q)}(D_0)} + \norm{\nabla g}{L^{4q/(4-q)}(D_0)}\right),$$
for some constant $C>0$. 
Moreover, $W^{2,q}$ embeds into $L^{2q/(2-q)}$ and $W^{1,q}$ embeds into $L^{4q/(4-q)}$. Using these embeddings, there exists a constant $C_q>0$ and $C>0$ such that:
$$\norm{\delbs\delb g}{L^q(D_0)}\leq C_q\norm{\omega}{W^{1,2}(D)}\norm{g}{W^{2,q}(D_0)} +C\norm{\omega}{W^{1,2}(D)}.$$
By the bound \eqref{app:gest}, it follows that there exists $C_q>0$:
$$\norm{\delbs\delb g}{L^q(D_0)}\leq C_q\norm{\omega}{W^{1,2}(D)}\norm{\omega}{W^{1,2}(D_0)} +C\norm{\omega}{W^{1,2}(D)}$$
and since $\norm{\omega}{W^{1,2}(D)}\leq \e<1$, then for some constant $C_q>0$, we have
$$\norm{\delbs\delb g}{L^q(D_0)}\leq C_q\norm{\omega}{W^{1,2}(D)}.$$
Putting this together with the fact that $\norm{g-id}{L^q(D)}\leq \norm{g-id}{L^4(D)}\leq  C\norm{\omega}{W^{1,2}(D)}$ and (\ref{app:elliptic}), there exists a constant $C_q>0$ such that 
$$\norm{g-id}{W^{2,q}(D_0)}\leq C_q\norm{\omega}{W^{1,2}(D)}.$$
Since $q\in(1,2)$ is arbitrary, we have proven the result.
\end{proof}

\begin{Rm}\label{bootstrapsubcritical}\ \\
\begin{enumerate}
\item[(i)] In the statement above, if $D=\CP^2$, then $D_0=D=\CP^2$. This is the case because $\delb$ is elliptic on $\CP^2$. 
\item[(ii)] Assume instead of $\norm{g-id}{L^4(D)}\leq C\norm{\omega}{W^{1,2}(D)}$, the slightly perturb inequality: $$\norm{g-id}{L^4(D)}\leq C\norm{\omega}{W^{1,2}(D)} +C_0,$$ where $C_0>0$ is a small constant. We can conclude from the proof of the statement that all arguments pass through and we can reach the natural conclusion: $$\norm{g-id}{W^{2,q}(D_0)}\leq C\left(\norm{\omega}{W^{1,2}(D)} +C_0\right)\qquad\text{ for all }q<2.$$
\item[(iii)] If we  higher regularity of $\omega$, we can obtain similar estimates using classical elliptic regularity results. Let $p>2$, a $(0,1)$-form $\omega\in W^{1,p}(\Omega^{0,1}D\otimes M_n(\C))$, satisfying a smallness condition and the integrability condition (\ref{II.1}). Moreover, $$\norm{g-id}{L^\infty(D)}\leq C\norm{\omega}{W^{1,p}(D)}.$$ Then we can bootstrap the equation solved by $g$ to show that $g\in W_{loc}^{2,p}$ with the expected estimate:
$$\norm{g-id}{W_{loc}^{2,p}(D)}\leq C_p\norm{\omega}{W^{1,p}(D)}.$$
This means that there exists a domain $D_0\subseteq D$ such that
$\norm{g-id}{W^{2,p}(D_0)}\leq C_p\norm{\omega}{W^{1,p}(D)}.$
\end{enumerate}

\end{Rm}
\vspace{4mm}

We can start proving two bootstrap procedures for two types of PDE-s. The general technique is to use show the boundedness of Morrey norms in order to bootstrap beyond the critical embedding level. We use the ideas from \cite{tristan2014yangmills}.

\begin{Prop} \label{bootstrapmorrey2}
Let $N\in \N^\ast$, $A\in W^{1,2}(B^4,\C^N)$ and $f_A\in C^\infty(\C^N,\C^N)$ such that there exists $C>0$ satisfying:
$$|f_A(\xi)|\leq C|\xi||\nabla A|+|A||\nabla \xi|$$
and $u\in W^{1,2}(B^4, \R^N)$ solving the equation:
$$\Delta u = f_A(u)$$
in a distributional sense, then $u\in W_{loc}^{2,p}(B^4, \C)$ for any $p<2$ and for any $p<2$ there exists $C_p>0$ such that $\norm{u}{W_{loc}^{2,p}(B^4)}\leq C_p\norm{A}{W^{1,2}(B^4)}$.
\end{Prop}
\begin{proof}[Proof of Proposition \ref{bootstrapmorrey2}]
Dimension 4 is critical in this case because $W^{2,4/3} \hookrightarrow L^4$ and we cannot directly bootstrap. In order to improve on the regularity of $u$, we will use the Adams-Morrey embedding.\\

\textbf{Claim}. $\exists \gamma>0$ such that
$$\sup_{x_0\in B_{1/2}^4(0),\ 0<\rho<1/4} \rho^{-\gamma} \int_{B_\rho^4(x_0)} |u|^4+|\nabla u|^2 dx^4<\infty$$

Let $\e>0$ to be fixed later. There exists $\rho_0>0$ such that:
$$\sup_{x_0\in B_{1/2}^4(0),\ 0<\rho<\rho_0} \norm{A}{W^{1,2}(B_\rho^4(x_0))}<\varepsilon.$$

We can always find such  $\e$ and $\rho_0$ since $\rho\mapsto  \int_{B_{\rho}(x_0)}$ is continuous.
Fix $x_0\in B_{1/2}^4(0)$ and $\rho<\rho_0$ arbitrary. To prove this claim we first consider :
$$\begin{array}{rcll}
\Delta \varphi & = & f_A(u)&\text{ in } B_\rho^4(x_0)\\
\varphi & = &0&\text{ on }\p B_\rho^4(x_0)
\end{array}$$
Let $v:= u - \varphi$. Then $\Delta v = 0$ and it easy to see that
$\Delta |v|^4\geq 0, \text { and } \Delta|\nabla v|^2\geq 0$
in $B_{\rho}(x_0)$. Applying the divergence theorem, we get that $\forall r<\rho$:
$$\int_{\p B_r^4(x_0)} \cfrac{\p |v|^4}{\p r} \geq 0, \text{ and }\int_{\p B_r^4(x_0)} \cfrac{\p |\nabla v|^2}{\p r} \geq 0.$$
These inequalities imply that:
$$\cfrac{d}{dr} \left[ \cfrac{1}{r^4} \int_{B_r^4(x_0)} |v|^4 dx^4 \right]\geq 0\text { and } \cfrac{d}{dr} \left[ \cfrac{1}{r^4} \int_{B_r^4(x_0)} |\nabla v|^2 dx^4 \right]\geq 0.$$
Since these derivatives are non-negative, it follows that the functions $r\mapsto \cfrac{1}{r^4} \int_{B_r^4(x_0)} |v|^4 dx^4$  and $r\mapsto \cfrac{1}{r^4} \int_{B_r^4(x_0)} |\nabla v|^2 dx^4$ are increasing in $r$. In particular:
$$\int_{B_{\rho/4}^4(x_0)} |v|^4 dx^4 \leq 4^{-4} \int_{B_\rho^4(x_0)} |v|^4 dx^4\quad\text{ and }\quad\int_{B_{\rho/4}^4(x_0)} |\nabla u|^2 dx^4 \leq 4^{-4} \int_{B_\rho^4(x_0)} |\nabla u|^2 dx^4.$$
Using these decays, we can bound $\int_{B_{\rho/4}^4(x_0)} | u|^4 dx^4$ and $\int_{B_{\rho/4}^4(x_0)} | \nabla u|^2 dx^4$ as such:
\be\label{app:u}
\begin{array}{lcl}
\int_{B_{\rho/4}^4(x_0)} | u|^4 dx^4 & \leq & 8 \int_{B_{\rho/4}^4(x_0)} | v|^4  + | \varphi|^4 dx^4\\
& \leq & 2^{-5} \int_{B_{\rho}^4(x_0)} | v|^4 dx^4 + 8\int_{B_{\rho}^4(x_0)} | \varphi|^4 dx^4\\
& \leq & 2^{-2} \int_{B_{\rho}^4(x_0)} | u|^4 dx^4 + 16 \int_{B_{\rho}^4(x_0)} | \varphi|^4 dx^4
\end{array}
\ee
Similarly, for $\nabla u$ we get the bound:
\be\label{app:nablau}
\int_{B_{\rho/4}^4(x_0)} | \nabla u|^2 dx^4 \leq 2^{-6} \int_{B_{\rho}^4(x_0)} | \nabla u|^2 dx^4 + 4 \int_{B_{\rho}}^4(x_0) | \nabla \varphi|^2 dx^4.
\ee
There exists a constant $C>0$ so that we can bound $\Delta \varphi$ in the $L^{4/3}$ norm:
$$\begin{array}{lll}
||\Delta \varphi||_{L^{4/3}(B_{\rho}^4(x_0))} &=& ||f_A(u)||_{L^{4/3}(B_{\rho}^4(x_0))}\\ &\leq& C\left(\norm{A}{L^4(B_{\rho}^4(x_0))} \norm{\nabla u}{L^2(B_{\rho}^4(x_0))} + \norm{\nabla A}{L^2(B_{\rho}^4(x_0))}\norm{u}{L^4(B_{\rho}^4(x_0))}\right).
\end{array}$$
Since $\varphi$ vanishes on the boundary, by Calderon-Zygmund inequality \cite{stein2016singular}, it follows that
$$\norm{\varphi}{W^{2,4/3}(B_{\rho}^4(x_0))}\leq C\left(\norm{A}{L^4(B_{\rho}^4(x_0))} \norm{\nabla u}{L^2(B_{\rho}^4(x_0))} + \norm{\nabla A}{L^2(B_{\rho}^4(x_0))} \norm{u}{L^4(B_{\rho}^4(x_0))} \right),$$
for some constant $C>0$. Moreover, the Sobolev embedding $W^{1,2}\hookrightarrow L^4$ gives:
\be\label{app:43norm}\norm{\varphi}{W^{2,4/3}(B_{\rho}^4(x_0))}\leq C\norm{A}{W^{1,2}(B_{\rho}^4(x_0))}\left(\norm{\nabla u}{L^2(B_{\rho}^4(x_0))}+ \norm{u}{L^4(B_{\rho}^4(x_0))} \right),\ee Thus, combining (\ref{app:43norm}) with the inequalities (\ref{app:u}) and (\ref{app:nablau}), we obtain the decay: 
$$
\int_{B_{\rho/4}^4(x_0)} |u|^4+|\nabla u|^2 dx^4  \leq  \left(2^{-2} + C_0\norm{A}{W^{1,2}(B_{\rho}^4(x_0))}\right)\int_{B_\rho^4(x_0)} |u|^4+|\nabla u| ^2 dx^4
$$
for some constant $C_0>0$. We can choose $\e>0$ so that $C_0\e^4\leq 2^{-2}$ to get:
\be\label{app:dec}\int_{B_{\rho/4}^4(x_0)} |u|^4 + |\nabla u|^2 dx^4  \leq  2^{-1}\int_{B_\rho^4(x_0)} |u|^4+|\nabla u| ^2 dx^4.\ee
This estimate gives the required existence of $\gamma>0$, and proves the claim.\\

It remains to prove the main regularity result using the claim. From the equation satisfied by $u$ and the decay inequality (\ref{app:dec}), we obtain the bound:
$$\sup_{x_0\in B_{1/2}^4(0),\ 0<\rho<1/4} \rho^{-\gamma} \int_{B_\rho^4(x_0)} |\Delta u|^{4/3} dx^4<\infty$$

By Adams-Morrey embedding, we get a bound on $||I_1 \Delta u||_{L^p(B_{1/2}^4(0))}$, $p>2$ where $I_1$ is the Riesz potential (see \cite{adams1975note}). We obtain $\nabla u\in{L^p_{loc}} (B^4,\C)$ for $p>2$. Hence, the PDE becomes sub-critical and we can bootstrap to get $u\in W^{2,p}(B^4,\C)$ for any $p<2$.
\end{proof}
\vspace{4mm}

\begin{Prop} \label{bootstrapmorrey}
Let $A\in L^4(B^4,\C)$ and $f_A\in C^\infty(B^4,\C)$ such that there exists $C>0$ satisfying:
$$|f_A(\xi)|\leq C|\xi|^2+|A||\xi|$$
and $u\in W^{2,2}(B^4, \R^N)$ satisfying
$$\Delta u = f_A(\nabla u)$$
then $u\in W_{loc}^{2,p}(B^4, \C)$ for any $p<4$ and $\norm{u}{W_{loc}^{2,p}}\leq C_p \norm{A}{L^4(B^4)}$ where $C_p$ is a constant.
\end{Prop}
\begin{proof}[Proof of Proposition \ref{bootstrapmorrey}]
Dimension 4 is critical in this case because $\nabla u\in W^{1,2} \hookrightarrow L^4$ and we cannot directly bootstrap. In order to improve on the regularity of $u$, we will use the Adams-Morrey embedding.\\

\textbf{Claim}. $\exists \gamma>0$ such that
$$\sup_{x_0\in B_{1/2}^4(0),\ 0<\rho<1/4} \rho^{-\gamma} \int_{B_\rho^4(x_0)} |\nabla u|^4 dx^4<\infty$$

Let $\e>0$ to be fixed later. There exists $\rho_0>0$ such that:
$$\sup_{x_0\in B_{1/2}^4(0),\ 0<\rho<\rho_0} \norm{A}{L^4(B_{\rho}^4(x_0))}<\varepsilon$$

We can always find such  $\e$ and $\rho_0$ since $\rho\mapsto  \int_{B_{\rho}^4(x_0)}$ is continuous.
Fix $x_0\in B_{1/2}^4(0)$ and $\rho<\rho_0$ arbitrary. To prove this claim we first consider :
$$\begin{array}{rcll}
\Delta \varphi & = & f_A(\nabla u)&\qquad\text{ in }B_\rho^4(x_0)\\
\varphi & = &0&\qquad\text{ on }\p B_\rho^4(x_0)
\end{array}$$
Let $v:= u - \varphi$. Then $\Delta v = 0$ and it easy to see that
$\Delta |\nabla v|^4\geq 0$
in $B_r^4(x_0)$, for some $r<\rho$. Applying the divergence theorem, we get that $\forall r<\rho$:
$$\int_{\p B_r^4(x_0)} \cfrac{\p |\nabla v|^4}{\p r} \geq 0$$
This implies that
$$\cfrac{d}{dr} \left[ \cfrac{1}{r^4} \int_{B_r^4(x_0)} |\nabla v|^4 dx^4 \right]\geq 0$$
and consequently the function $r\mapsto \cfrac{1}{r^4} \int_{B_r^4(x_0)} |\nabla v|^4 dx^4$ is increasing. In particular,
$$\int_{B_{\rho/4}^4(x_0)} |\nabla v|^4 dx^4 \leq 4^{-4} \int_{B_\rho^4(x_0)} |\nabla v|^4 dx^4$$
Using this decay, we can obtain a bound for $\int_{B_{\rho/4}^4(x_0)} |\nabla u|^4 dx^4$:
$$
\begin{array}{lcl}
\displaystyle\int_{B_{\rho/4}^4(x_0)} |\nabla u|^4 dx^4 & \leq & 8 \displaystyle\int_{B_{\rho/4}^4(x_0)} |\nabla v|^4  + |\nabla \varphi|^4 dx^4\\
& \leq & 2^{-5} \displaystyle\int_{B_{\rho}^4(x_0)} |\nabla v|^4 dx^4 + 8\displaystyle\int_{B_{\rho}^4(x_0)} |\nabla \varphi|^4 dx^4\\
& \leq & 2^{-2}\displaystyle \int_{B_{\rho}^4(x_0)} |\nabla u|^4 dx^4 + 16 \displaystyle\int_{B_{\rho}^4(x_0)} |\nabla \varphi|^4 dx^4
\end{array}
$$
Moreover, there exists a constants $C_1>0$ and $C_2>0$ such that
$$\begin{array}{lll}
||\Delta \varphi||_{L^2(B_{\rho}^4(x_0))}^4 &\leq & C_1||f_A(\nabla u)||_{L^2(B_{\rho}^4(x_0))}^4\\
& \leq& C_2\left(\norm{A}{L^4(B_{\rho}^4(x_0))} \norm{\nabla u}{L^4(B_{\rho}^4(x_0))} + \norm{\nabla u}{L^4(B_{\rho}^4(x_0))}^2\right)^4\\
&\leq & 8C_2\left(\norm{A}{L^4(B_{\rho}^4(x_0))}^4 \norm{\nabla u}{L^4(B_{\rho}^4(x_0))}^4 + \norm{\nabla u}{L^4(B_{\rho}^4(x_0))}^8\right)
\end{array}$$
Since $\varphi$ vanishes on the boundary of $B_{\rho}^4(x_0)$, then by elliptic estimates, we have for some constant $C>0$ the inequality:
$$\norm{\varphi}{W^{2,2}(B_{\rho}^4(x_0))}^4\leq C||\Delta \varphi||_{L^2(B_{\rho}^4(x_0))}^4,$$
from which we deduce that $\norm{\nabla \varphi}{L^4(B_{\rho}^4(x_0))}^4\leq C||\Delta \varphi||_{L^2(B_{\rho}^4(x_0))}^4$.
Putting this inequality together with the bound on $\Delta \varphi$, we get the following decay:
$$
\int_{B_{\rho/4}^4(x_0)} |\nabla u|^4 dx^4  \leq  \left(2^{-2} + C_0 || A||^4_{L^4(B_\rho^4(x_0))}\right)\int_{B_\rho^4(x_0)} |\nabla u| ^4 dx^4
$$
We can choose $\e>0$ so that $C_0\e^4\leq 2^{-2}$ and obtain:
$$\int_{B_{\rho/4}^4(x_0)} |\nabla u|^4 dx^4  \leq  2^{-1}\int_{B_\rho^4(x_0)} |\nabla u| ^4 dx^4.$$
This decay implies the existence of $\gamma>0$ and proves the claim.\\

Because $|\Delta u|^2 = |f(\nabla u)|^2\leq C\left( |\nabla u|^4 + |\nabla u|^2|A|^2\right)$, we can use the claim above to obtain the following bound:
$$\sup_{x_0\in B_{1/2}^4(0),\ 0<\rho<1/4} \rho^{-\gamma} \int_{B_\rho^4(x_0)} |\Delta u|^2 dx^4<\infty$$

By Adams-Morrey embedding, we get a bound on $||I_1 \Delta u||_{L^p}$, $p>2$ where $I_1$ is the Riesz potential (see \cite{adams1975note}). Thus, it follows that $\nabla u\in{L^p_{loc}} (B^4,\C)$ for $p>4$. Since the PDE becomes sub-critical, we can bootstrap to get $u\in W_{loc}^{2,p}(B^4,\C)$ for any $r<2$.
\end{proof}

\subsection*{Results in Several Complex Variables Theory in $\C^2$}
We will briefly recall some of the results we will be using from the theory of several complex variables and apply them to the case of the unit ball $B^4$ embedded into $\C^2$. At the end of this section we prove regularity results for the operators $T_1$ and $T_2$ as defined in \cite{ovrelid1972integral}.

\begin{Dfi}
We say that $D$ be a bounded domain into $\C^n$ has boundary of class $C^k$ if for every $p\in \p D$ and $U$ neighbourhood of $p$, there exists a function $r:U\to\R$ such that $U\cap D = \{z\in U|\ r(z)<0\}$, $U\cap \p D=\{z\in U|\ r(z)=0\}$ and $\nabla r(z) \neq 0$ on $U\cap\p D$. Then $r$ is called a $C^k$ \textbf{local defining function} for $ D$. If $\overline{ D}\subset U$, then $r$ is a \textbf{global defining function}.
\end{Dfi} 

In particular, on $B^4$ we define $$r(z) = |z|^2-1.$$ $r$ is a global defining function for $B^4$.

\begin{Dfi}
Let $ D$ be a bounded domain into $\C^n$ and $r$ a $C^2$ defining function. $ D$ is \textbf{pseudoconvex} at $p\in \p D$ if the Levi form 
$$L_p(r,t) = \sum\limits_{i,j=1}^n \frac{\p^2 r}{\p z_i\p \ov{z}_j} t_j\ov{t}_k\geq 0$$
for all $t\in T_p^{1,0}(\p  D)$. $ D$ is \textbf{strictly pseudoconvex} at $p$ if $L_p(r,t)>0$ whenever $t\neq 0$. If $ D$ is (strictly) pseudoconvex for all $p\in\p D$ then $ D$ is (strictly) pseudoconvex.
\end{Dfi}

$B^4$ is an example of a strictly pseudoconvex domain. Indeed, if the pick the defining function above, we have that the Levi form
$$L_p(r,t) = |t|^2>0$$
for all $t\in T_p^{1,0}(\p B^4)$, $t\neq 0$.\\

On $B^4$ take the canonical complex structure $J$. At each point $p\in \p B^4$, we can find an orthonormal $(0,1)$ fields $L_{\ov{\tau}}$ and $L_{\delb r}$ that span the complexified tangetial space $T^{\C}_p(\p B^4)$. We will explictly compute them.\\

Let $e^\ast$, $Je^\ast$, $dr$, $Jdr$ define the Hopf frame. These define the a orthonormal basis for $(0,1)$-forms. Namely, $\ov{\tau} = e^\ast +iJe^\ast$ and $\delb r = dr+iJdr$. In terms of $\ov{z}_1$ and $\ov{z}_2$, they satisfy:
\be
\begin{array}{rll}
dr & =  \frac{1}{2r} (\ov{z_1}dz_1 + z_1d\ov{z_1} +z_2d\ov{z_2} + \ov{z_2}dz_2) & =  \frac{1}{2}(\p r + \delb r) \\

Jdr & =  \frac{1}{2ir} (- \ov{z_1}dz_1 + z_1d\ov{z_1} +z_2d\ov{z_2} -  \ov{z_2}dz_2) & =  \frac{1}{2}(\delb r - \p r) \\

e^\ast & =  \frac{1}{2r} (z_2dz_1 - z_1dz_2 +\ov{z_2}d\ov{z_1} - \ov{z_1}d\ov{z_2}) & =  \frac{1}{2}(\tau + \ov{\tau})\\

Je^\ast & =  \frac{1}{2ir} (\ov{z_2}d\ov{z_1} - \ov{z_1}d\ov{z_2}  - z_2dz_1 + z_1dz_2) & =  \frac{1}{2r}(\ov{\tau} - \tau).
\end{array}
\ee
We obtain the explicit formulation of $\ov{\tau}$ and $\delb r$:
\be
\begin{array}{rl}
\ov{\tau} & =  \frac{1}{r}\left(\ov{z_2}d\ov{z_1} - \ov{z_1}d\ov{z_2}\right)\\
\delb r & =  \frac{1}{r}\left(z_1d\ov{z_1} +z_2d\ov{z_2}\right)\\
\end{array}
\ee
Moreover, the vector fields $L_{\ov{\tau}}$ and $L_{\delb r}$ can be computed as follows:
$$\begin{array}{rl}
L_{\ov{\tau}} &= \frac{1}{2}\left(\p_{e^\ast} - i \p_{Je^{\ast}}\right) =\frac{1}{r}\left( z_2\p_{\ov{z}_1} - z_1\p_{\ov{z}_2}\right)\\
L_{\delb r} &= \frac{1}{2}\left(\p_{r} - i \p_{Jdr}\right) = \frac{1}r\left(\ov{z}_1\p_{\ov{z}_1} + \ov{z}_2\p_{\ov{z}_2}\right).
\end{array}$$
In particular, $L_{\ov{\tau}}$ is a \textbf{tangential Cauchy-Riemann} vector field - it satisfies $L_{\ov{\tau}} r=0$. Using these vector fields and their conjugates $L_{\tau}$ and $L_{\p r}$, we obtain the formal adjoint $$\vartheta = -\ast\delb \ast.$$
Moreover, we have the following relation:
$$(\vartheta \alpha, \beta) = (\alpha, \delb \beta) + \int_{\p B^4}\langle \sigma(\vartheta, dr)\alpha, \beta\rangle dS,$$
where $\alpha$ is a $(p,q)$-form and $\beta$ a $(p,q-1)$-form, where we are using the notation in the literature $\sigma(\vartheta, dr)\alpha$ to denote
$$\sigma(\vartheta, dr)\alpha = \ov{\ast \delb r\wedge \ast \ov{\alpha}}.$$
More explicitly, $\sigma(\vartheta, dr)\alpha$ is the form of whose components are in the $\delb r$ frame of $\alpha$. Thus, if $\sigma(\vartheta, dr)\alpha = 0$, then $\vartheta = \delbs$, where $\delbs$ is the Hilbert adjoint. In this case we say that $\alpha\in Dom(\delbs)$ to emphasize that $\alpha$ is in the domain of $\delbs$.\\

It is also useful to remark the fact that if a $(0,q)$-form $\alpha$ vanishes on the boundary, it follows that $\alpha$ vanishes component wise on the boundary. This is because the frame $\delb r$ does not vanish on the boundary, unlike $dr$ which does vanish! In terms of the notation above, $\alpha=0$ on $\p B^4$ is equivalent to $\sigma(\vartheta, dr)\alpha = 0$ and $\sigma(\delb, dr)\alpha=0$ on $\p B^4$, where $\sigma(\delb, dr)\cdot =\delb r \wedge \cdot $ is the adjoint of $\sigma(\vartheta, dr)$.\\

We recall the \textbf{Integral Representation Theorem} which was proven in \cite{ovrelid1972integral} for $(0,q)$-forms and initially in \cite{henkin1970integral} for $(0,1)$-forms. Before doing so, we will have to define a few key operators. Let $D$ be a strictly pseudoconvex domain in $\C^n$ with defining function $r$ such that
$$\sum_{i,j=1}^n \frac{\p^2 r}{\p x_i\p x_j} t_it_j \geq c|t|^2$$
for some $c\in \R$ and $t\in \R^{2n}$.\\

The Bochner-Martinelli-Koppelman kernel (see \cite[Theorem 11.1.2]{chen2001partial}) is given by:
\be\label{app:BMK}
K(\zeta, z) = \frac{1}{(2\pi i)^n} \frac{\sum_{i=1}^n (\ov{\zeta}_i-\ov{z}_i)d\zeta_i}{|\zeta-z|^2}\wedge\left(\frac{\sum_{i=1}^n (d\ov{\zeta}_i-d\ov{z}_i)\wedge d\zeta_i}{|\zeta-z|^2}\right)^{n-1}.
\ee
We define the kernel $K_q$ as being the form of $(0,q)$ degree in $z$ and $(n,n-q-1)$ degree in $\zeta$. Moreover, the boundary kernels $K^\p$, $K^\p_{00}$ given by
$$
\begin{array}{rl}
\displaystyle K^\p =&  \cfrac{1}{(2\pi i)^n} \cfrac{\sum_{i=1}^n (\ov{\zeta}_i-\ov{z}_i)d\zeta_i}{|\zeta-z|^2}\wedge \cfrac{\sum_{i=1}^n \p_{{\zeta}_i} r(\zeta)  d\zeta_i}{\sum_{i=1}^n \p_{{\zeta}_i} r(\zeta) (\zeta_i-z_i)}\\
& \wedge \sum\limits_{k_1+k_2=n-2}  \left(\cfrac{\sum_{i=1}^n (\ov{\zeta}_i-\ov{z}_i)d\zeta_i}{|\zeta-z|^2}\right)^{k_1}\wedge \left(\cfrac{\sum_{i=1}^n \p_{\ov{\zeta}_i}\p_{{\zeta}_i} r(\zeta)  d\ov{\zeta}_i\wedge d\zeta_i}{\sum_{i=1}^n \p_{{\zeta}_i} r(\zeta) (\zeta_i-z_i)}\right)^{k_2}
\end{array}
$$

and

$$
\displaystyle K^\p_{00} =  \cfrac{1}{(2\pi i)^n} \cfrac{\sum_{i=1}^n (\ov{\zeta}_i-\ov{z}_i)d\zeta_i}{|\zeta-z|^2}\wedge \left(\cfrac{\sum_{i=1}^n \p_{\ov{\zeta}_i}\p_{{\zeta}_i} r(\zeta)  d\ov{\zeta}_i\wedge d\zeta_i}{\sum_{i=1}^n \p_{{\zeta}_i} r(\zeta) (\zeta_i-z_i)}\right)^{n-1}.
$$
For each $q\geq 1$ we can, thus, define
$$T_q:C^\infty(\Omega^{0,q}  D)\to C^\infty(\Omega^{0,q-1}D)$$
given by
\be \label{app:T1}T_1 (\alpha)= \int_D K_0 \wedge \alpha - \int_{\p D} K^{\p}_0\wedge \alpha\qquad \text{ when }q=1
\ee
and
\be\label{app:T_q} T_q (\alpha)= \int_D K_{q-1} \wedge \alpha \qquad \text{ when }q>1,\ee
where by $K_0^{\p}$ we understand the form of $(0,0)$ degree in $z$. 
We now formulate the representation theorem (this can be found in \cite[Section 3]{ovrelid1972integral} and \cite[Theorem 11.2.7]{chen2001partial}):
\begin{Th}\label{app:repr}
Let $D$ be a bounded strictly pseudoconvex domain in $\C^n$ with $C^2$ boundary, $0\leq q\leq n$ and $\alpha\in C^\infty(\Omega^{0,q}D, \C)$. Then we have the following representations:
$$\begin{array}{rll}
\alpha(z) &= \displaystyle\int_{\p D} K^\p_{00}(\zeta,z)\alpha(\zeta) + T_1(\delb \alpha)&\text{ when }q=0\\
\alpha(z) &= \delb (T_q \alpha)+ T_{q+1}(\delb \alpha)&\text{ when }q>0.
\end{array}$$
\end{Th}
Moreover, for results concerning regularity of operators $T_q$, we recommend to the reader \cite{krantz1979estimates} (optimal $L^p$ results for $(0,1)$-forms) and \cite{ovrelid1972integral} ($L^p$ and H\"older regularity results for $(0,q)$-forms).\\

We prove a regularity result for $(0,2)$-forms in the domain $B^4$, which comes in-handy in our paper. We are unaware of such a result being available in the literature.
\begin{Prop}\label{app:T2res}
The operator $T_2$ maps $L^p(\Omega^{0,2}B^4)$ into $W^{1,p}(\Omega^{0,1}B^4)$ whenever $p>1$.
\end{Prop}
\begin{proof}[Proof of Proposition \ref{app:T2res}]
From the formula (\ref{app:BMK}), we obtain
$$K_1(\zeta,z) =  -\frac{1}{4\pi^2} \frac{\sum_{i=1}^2 (\ov{\zeta}_i-\ov{z}_i)d\zeta_i}{|\zeta-z|^2}\wedge \frac{\sum_{i=1}^2 -d\ov{z}_i\wedge d\zeta_i}{|\zeta-z|^2}.$$
We expand the equation above to get:
$$K_1(\zeta,z) = -\frac{1}{4\pi^2 |\zeta-z|^4} \left( (\ov{\zeta}_1-\ov{z}_1) d\zeta_1\wedge d\ov{z_2}\wedge d\zeta_2 + (\ov{\zeta}_2-\ov{z}_2) d\zeta_2\wedge d\ov{z_1}\wedge d\zeta_1\right).$$
Let $\alpha = f d\ov{z}_1\wedge d\ov{z}_2\in L^p(\Omega^{0,2}B^4)$. From the formula of $T_2$  it follows that:
$$\begin{array}{rl}
\displaystyle T_2(\alpha) = &\left(-\cfrac{1}{4\pi^2} \displaystyle\int_{B^4}  \cfrac{1}{|\zeta-z|^4} (\ov{\zeta}_2-\ov{z}_2) f d\zeta_1\wedge d\ov{\zeta}_1\wedge d\zeta_2\wedge d\ov{\zeta}_2\right)d\ov{z_1} \\
&+  \left(-\cfrac{1}{4\pi^2} \displaystyle\int_{B^4}  \cfrac{1}{|\zeta-z|^4} (\ov{\zeta}_1-\ov{z}_1) f d\zeta_1\wedge d\ov{\zeta}_1\wedge d\zeta_2\wedge d\ov{\zeta}_2\right)d\ov{z_2}.
\end{array}$$
Since each component of $K_1$ is a quasi-potential in the sense of \cite[Definition 3.7.1]{morrey2009multiple}, then we can apply \cite[Theorem 3.7.1]{morrey2009multiple} component wise to $T_2(\alpha)$ to get the required result:
$$\norm{T_2(\alpha)}{W^{1,p}(B^4)}\leq \norm{\alpha}{L^p(B^4)}.$$
\end{proof}

In addition, the following result builds upon the sharp estimates of the Henkin operator ($T_1$ in our notation) found by \cite{krantz1979estimates}. In particular, we show that for estimating $T_1\alpha$, where $\alpha$ is a $(0,1)$ form, we can relax the condition $\delb \alpha =0$. The estimates we find are not sharp.

\begin{Prop}\label{app:T1res}
Let $p>6$ and $q>6$ such that $W^{1,p}(B^4)\hookrightarrow L^q(B^4)$ and $\alpha\in L^q(\Omega^{0,1}B^4)$ satisfying $\delb \alpha\in L^p(B^4)$. Then there exists a constant $C>0$ depending on $p$ and $q$ such that:
$$\norm{T_1\alpha}{L^\infty(B^4)} + \norm{\delb T_1\alpha}{L^q(B^4)}\leq C\left(\norm{\alpha}{L^q(B^4)} + \norm{\delb\alpha}{L^p(B^4)}\right).$$
\end{Prop}

\begin{proof}[Proof of Proposition \ref{app:T1res}]
We refer to the proofs presented in \cite{krantz1979estimates}. We recall:
$$T_1(\alpha) = \int_{B^4} K_0 \wedge \alpha - \int_{\p B^4} K^{\p}_0\wedge \alpha.$$ 
In \cite[Section 5]{krantz1979estimates} it is shown that the first term has good regularity. In particular that $\int_{B^4} K_0 \wedge \alpha $ belongs to a Lipschitz space when $\alpha\in L^q$ for $q>6$. We focus our attention to the second term which is problematic. By Stokes we obtain:
$$\int_{\p B^4} K^{\p}_0\wedge \alpha = \int_{B^4} \delb K^{\p}_0\wedge \alpha - \int_{B^4} K^{\p}_0\wedge \delb \alpha.$$
Since $\delb \alpha$ does not vanish, we obtain two terms. The first integral is estimated in \cite[Section 5,6]{ krantz1979estimates} and yields the regularity result:
$$\norm{T_1(\alpha)}{L^\infty(B^4)}\leq C\norm{\alpha}{L^q(B^4)}.$$
for some constant $C>0$. It remains to deal with the term: $\int_{B^4} K^{\p}_0\wedge \delb \alpha$. However, since $\delb K^{\p}_0$ is more singular than $K^{\p}_0$, since $\delb \alpha\in L^p$, we have at least the estimate:
$$\norm{\int_{B^4} K^{\p}_0\wedge \delb \alpha}{L^\infty(B^4)}\leq C\norm{\delb \alpha}{L^p(B^4)}.$$
Hence, we have that there exists a constant $C>0$ such that 
$$\norm{T_1\alpha}{L^\infty(B^4)} \leq C\left(\norm{\alpha}{L^q(B^4)} + \norm{\delb\alpha}{L^p(B^4)}
\right).$$
Since $\delb\alpha$ is well-defined, by density of smooth forms, we obtain by Theorem \ref{app:repr} the following equation:
$$\alpha = \delb T_1(\alpha) + T_2(\delb \alpha),$$
and
$$\norm{\delb T_1(\alpha)}{L^q(B^4)}\leq \norm{\alpha}{L^q(B^4)} + \norm{T_2(\delb \alpha)}{L^q(B^4)}.$$
By Proposition \ref{app:T2res}, we have that $T_2(\delb\alpha)\in W^{1,p}\hookrightarrow L^q$. In particular, we obtain constants $C,C'>0$  such that:
$$\norm{\delb T_1(\alpha)}{L^q(B^4)}\leq C\left(\norm{\alpha}{L^q(B^4)} + \norm{T_2(\delb \alpha)}{W^{1,p}(B^4)}\right)\leq C'\left(\norm{\alpha}{L^q(B^4)} + \norm{\delb \alpha}{L^{p}(B^4)}\right).$$
Hence, by putting everything together we get:
$$\norm{T_1\alpha}{L^\infty(B^4)} + \norm{\delb T_1\alpha}{L^q(B^4)}\leq C\left(\norm{\alpha}{L^q(B^4)} + \norm{\delb\alpha}{L^p(B^4)}\right).$$
\end{proof}
\newpage

\addcontentsline{toc}{section}{References}
\bibliographystyle{plain}

\end{document}